\documentclass{base_class}

\title[]{Integral Geometry on the Octonionic Plane} 

\author{Jan Kotrbat{\'y}}\thanks{JK was supported by DFG grants WA3510/1-1 and BE 2484/5-2.}
\author{Thomas Wannerer}\thanks{TW was supported by DFG grant WA3510/3-1}
 
\email{kotrbaty@math.uni-frankfurt.de}
\email{thomas.wannerer@uni-jena.de}
\address{Goethe University Frankfurt, Institute of Mathematics, Robert-Mayer-Str. 10, 60629 Frankfurt, Germany}
\address{Friedrich-Schiller-Universit\"at Jena, Fakult\"at f\"ur Mathematik und Informatik, Institut für Mathematik, Ernst-Abbe-Platz 2, 07743 Jena, Germany}
\subjclass[2020]{}
\date{\today}

\begin{document}

\maketitle

\begin{abstract}
	We describe explicitly the algebra of $\Spin(9)$-invariant, translation-invariant, continuous valuations on the octonionic plane.  Namely, we present a basis in terms of invariant differential forms and determine the Bernig--Fu convolution on this space. The main technical ingredient we introduce is an extension of the invariant theory of the Lie group $\Spin(7)$ to the isotropy representation of the action of $\Spin(9)$ on the 15-dimensional sphere, reflecting the underlying octonionic structure. As an application, we compute the principal kinematic formula on the octonionic plane and express in our basis certain $\Spin(9)$-invariant valuations introduced previously by Alesker.
\end{abstract}

\section{Introduction}

 The kinematic formulas of integral geometry have attracted considerable attention throughout the twentieth century, including important contributions by eminent mathematicians such as Blaschke, Chern, and Federer.  The profound connections of this subject with convex geometry were first systematically investigated in  Hadwiger's seminal monograph \emph{Vorlesungen \"uber Inhalt, Oberfl\"ache und Isoperimetrie} \cite{Hadwiger:Vorlesungen}: Characterizing linear combinations of the intrinsic volumes $\mu_0,\dots,\mu_n$ as the only continuous, rigid-motion-invariant valuations on convex bodies (finitely additive functions on the collection of convex compact subsets of Euclidean space),  Hadwiger obtained a strikingly elegant proof  of the  kinematic formulas in Euclidean space:
	\begin{align}
		\label{eq:kf}
		\int_{\b{\SO(n)}}\mu_k(A\cap \b g B)d\b g&=\sum_{i,j=0}^n c_{i,j}^k\mu_i(A)\mu_j(B),
	\end{align}
	where $A,B\subset\RR^n$ are arbitrary convex bodies, $c_{i,j}^k\in\RR$ explicitly known constants, and $d\b g$ a  Haar measure on the group $\b{\SO(n)}=\SO(n)\ltimes\RR^n$ of rigid motions.

The groundbreaking work of Alesker \cite{Alesker:OnConjecture,Alesker:Irreducibility} on valuations on convex bodies deepened the connection between convex and integral geometry 
	by revealing  remarkable algebraic structures  behind kinematic formulas. The prospect of a more structural understanding and the availability of powerful new tools triggered a renewed interest and led to brilliant recent advances in integral geometry.

 A fundamental discovery of  Alesker generalizing Hadwiger's characterization of the intrinsic volumes is that the subspace $\Val^G\subset\Val$ of $G$-invariant valuations is finite-dimensional for any closed subgroup $G\subset\SO(n)$ acting transitively on the sphere $S^{n-1}$. As a consequence, the kinematic formulas \eqref{eq:kf} exist for any such  $G$ instead of $\SO(n)$ with the intrinsic volumes replaced by a basis of $\Val^G$. The classification of groups acting transitively on spheres is a classical result due to Montgomery and Samelson \cite{MontgomerySamelson:Spheres} and Borel \cite{Borel:Transitive}; namely, there are six infinite series
\begin{align*}
	\SO(n),\quad \U(k),\SU(k)\subset\SO(2k), \quad \Sp(k),\Sp(k)\U(1),\Sp(k)\Sp(1)\subset\SO(4k),
\end{align*}
and three exceptions
\begin{align*}
	\G_2\subset\SO(7),\quad \Spin(7)\subset\SO(8), \quad\Spin(9)\subset\SO(16).
\end{align*}
The ultimate goal of understanding the integral geometry in each of these cases (which are evidently related to the four normed division algebra of real and complex numbers, the quaternions, and the octonions), remains, in spite of significant progress, a challenge.

An obvious hurdle in establishing \eqref{eq:kf} for the other groups $G$ is to determine the constants $c_{i,j}^k$. While in the classical case $G=\SO(n)$ the constants are immediately computed by plugging in Euclidean balls for $A$ and $B$, this  method fails misserably in all other cases due to the greater numbers of unknowns and lack of bodies for which the integrals can be directly evaluated. Fortunately,  Alesker's theory of valuations provides a framework for investigating kinematic formulas which is far superior to any such ad hoc attempts: Building on his solution of McMullens's conjecture, Alesker defined in \cite{Alesker:Product} a natural product of valuations that satisfies, among other remarkable properties, a Poincar\'e-type duality and turns $\Val^G$ into a subalgebra. The \emph{fundamental theorem of algebraic integral geometry} (ftaig) is the crucial observation of Fu \cite{Fu:Unitary}, enhanced later by Bernig and Fu \cite{BernigFu:Convolution},  that the structure of the algebra $\Val^G$ is encoded in the corresponding kinematic formulas and vice versa.

Applying the ftaig and using different bases of $\Val^{\U(k)}$ previously introduced in \cite{Alesker:HL,Tasaki:Generalization1,Tasaki:Generalization2,Park:PHD} as well as Fu's elegant description \cite{Fu:Unitary} of the algebra structure, Bernig and Fu accomplished in their fundamental paper \cite{BernigFu:HIG} the determination of the full array of kinematic formulas in the Hermitian space $\CC^k$. Bernig further extended these results to the group $\SU(k)$ and fully resolved also the two exceptional cases $G_2$ and $\Spin(7)$, see \cite{Bernig:SU,Bernig:G2, Bernig:ProductFormula}. In the case of symplectic groups, which are closely related to the quaternions, on the contrary, our current knowledge is rather limited: First, combinatorial formulas for the Betti numbers of the valuation algebras were found by Bernig \cite{Bernig:Quaternionic}. Second, the algebra of $\Sp(2)\Sp(1)$-invariant valuations on the quaternionic plane $\HH^2$ and the corresponding kinematic formulas were determined  by Bernig and Solanes \cite{BernigSolanes:Classification,BernigSolanes:H2}. However, quaternionic integral geometry in general remains elusive, and so did the remaining exceptional case $G=\Spin(9)$ to which the present article is devoted. For completeness, let us also mention recent initial progress in extending the above developments  to tensor valuations \cite{BernigHug:Tensor}, to curved space forms \cite{SolanesW:Spheres, BFS:ComplexSpaceForms}, non-compact isotropy groups \cite{BFS:PseudoRiemannian,BFS:Crofton}, and valuations on function spaces \cite{CLM:Hadwiger2,Knoerr:Unitarily1}.

\subsection{Our results}

In this article we describe a basis of $\Val^{\Spin(9)}$, determine the multiplicative structure on this space, and compute explicitly the corresponding principal kinematic formula as an application. Before stating our results, let us summarize the a priori knowledge we have. As we will see, many of the structures we deal with throughout the text are closely related to the octonionic  multiplication; in particular, this holds for the group $\Spin(9)\subset\SO(16)$. In what follows, we therefore tacitly identify 
$\RR^8$ with $\OO$ and $\RR^{16}$ with $\OO^2$ (more details are given in Section \ref{ss:OO}).

First of all, one has the McMullen grading  $\Val^{\Spin(9)}=\bigoplus_{k=0}^{16}\Val_k^{\Spin(9)}$, where $\Val_k^{\Spin(9)}$ is the subspace of $k$-homogeneous valuations. The dimensions $ b_k=\dim\Val_k^{\Spin(9)}$ were recently determined by Bernig and Voide \cite[Theorem 1]{BernigVoide:Spin} as follows:
\begin{align}
	\renewcommand{\arraystretch}{1.3}
	\label{eq:BettiSpin9}
	\begin{array}{|c||c|c|c|c|c|c|c|c|c|c|c|c|c|c|c|c|c|}
		\hline
		k & 0&1&2&3&4&5&6&7&8&9&10&11&12&13&14&15&16
		\\\hline
		b_k& 1&1&2&3&6&10&15&20&27&20&15&10&6&3&2&1&1
		\\\hline
	\end{array}\,.
\end{align}
In particular, $\dim\Val^{\Spin(9)}=143$. Clearly, we have the intrinsic volumes $\mu_k\in\Val_k^{\Spin(9)}$. Further valuations $T_k\in\Val_k^{\Spin(9)}$ and $U_{16-k}\in\Val_{16-k}^{\Spin(9)}$, $0\leq k\leq 8$, naturally reflecting the octonionic structure, were introduced by Alesker   \cite{Alesker:Spin9}. However, \eqref{eq:BettiSpin9} shows that there are in fact many more.

In order to determine a basis of the space $\Val^{\Spin(9)}$, we will use a by now standard construction of valuations via the integration of differential forms over the normal cycle: Each $\phi\in\Val_k^{\Spin(9)}$, $0\leq k\leq 15$, can be represented by a $\b{\Spin(9)}$-invariant smooth 
 15-form on the sphere bundle $S\OO^2=\OO^2\times S^{15}$. The representation is far from unique but the ambiguity is precisely quantified by a certain second-order differential operator by virtue of the Bernig--Bröcker kernel theorem  \cite{BernigBroecker:Rumin}. This approach is particularly convenient for determining the convolution  of valuations introduced by Bernig and Fu \cite{BernigFu:Convolution}.

The first step of our construction, namely the description of the algebra $\Omega^*(S\OO^2)^{\b{\Spin(9)}}$, reduces by transitivity and invariance to describing $\left[\largewedge(T_pS\OO^2)^*\right]^{\Stab_p\b{\Spin(9)}}$ for a single point $p$. The tangent space decomposes under the action of $\Stab_p\b{\Spin(9)}\cong\Spin(7)$ as
\begin{align}
	\label{eq:Tp}
	T_pS\OO^2\cong\RR\oplus\Imag\OO\oplus\OO\oplus\Imag\OO\oplus\OO.
\end{align}
Consider the corresponding ($\OO$-valued) coordinate 1-forms $\alpha,\theta_0,\theta_1,\varphi_0,\varphi_1$ and recall that the usual wedge product extends naturally to a (non-associative) product of $\OO$-valued forms, see \cite{Kotrbaty:O-valued}. Extending the invariant theory of $\Spin(7)$ to the direct sum $\Imag\OO\oplus\OO$ of the vector and spin representation, we obtain:
\begin{theorem}
	The algebra $\left[\largewedge(T_pS\OO^2)^*\right]^{\Spin(7)}$ has 97 generators. All of them are polynomials (in the sense of the wedge product of $\OO$-valued forms) in $\alpha,\theta_0,\theta_1,\varphi_0,\varphi_1$, and their conjugates.
\end{theorem}
Precise formulas for the generators are given in Theorem \ref{thm:Spin7invforms}, see also the proof of Theorem \ref{thm:Spin9invforms}. Let us point out that the generators themselves are real-valued; however, in order to express them in a closed form it is crucial that they are regarded as elements of the larger algebra of $\OO$-valued forms. 

The second task, the accomplishment of which is necessary for determining the kernel of the map
$\Omega^{15}(S\OO^2)^{\b{\Spin(9)}}\to\Val^{\Spin(9)}$ as well as for computing the convolution, is to compute the exterior derivative of the invariant forms. Since we only know the forms explicitly in a point, we employ Cartan's method of moving frames and regard them temporarily as differential forms on the group $\b{\Spin(9)}$ itself. Using the Maurer--Cartan equations and the description of the Lie algebra $\spin(9)\subset \End(\OO^2)$ we introduce, we obtain the following formal differentiation rules made rigorous by Theorem \ref{thm:d} below:

\begin{theorem}
	\label{thm:d0}
	The exterior derivative on $\Omega^*(S\OO^2)^{\b{\Spin(9)}}\cong\left[\largewedge(T_pS\OO^2)^*\right]^{\Spin(7)}$ is completely determined by the following  equations
	\begin{align*}
		d  (\alpha+\theta_0)&=-\varphi_0\wedge(\alpha+\theta_0)-\theta_1\wedge\b{\varphi_1},\\
		d \theta_1&=(\alpha+\theta_0)\wedge\varphi_1-\varphi_0\wedge\theta_1+\theta_1\wedge\varphi_0,\\
		d \varphi_0&=-\varphi_0\wedge\varphi_0-\varphi_1\wedge\b{\varphi_1},\\
		d \varphi_1&=\varphi_1\wedge\varphi_0.
	\end{align*}
\end{theorem}

Based on the two main ingredients, we can construct in terms of the invariant forms a basis of $\Val^{\Spin(9)}$ and 
 determine the convolution product on this space. We thus obtain our main result:
\begin{theorem}
\label{thm:main}
	There is an isomorphism of graded algebras
	\begin{align*}
		\Val^{\Spin(9)}\cong\RR[t,s,v,u_1,u_2,w_1,w_2,w_3,x_1,x_2,y,z]/I,
	\end{align*}
	where $t,s,v,u_i,w_i,x_i,y,z$ have degrees $1,\dots,8$, respectively, and $I=(f_1,\dots,f_{78})$ is the ideal generated by the  homogeneous elements $f_1,\dots,f_{78}$ listed in Appendix \ref{app:relations}.  Moreover, both homogeneous generating sets $\{t,s,v,u_1,u_2,w_1,w_2,w_3,x_1,x_2,y,z\}$ and $\{f_1,\dots,f_{78}\}$ are minimal.
\end{theorem}

 An explicit isomorphism is given in the proof of this theorem in Section \ref{s:main}. An explicit basis of $\Val^{\Spin(9)}$ consisting of monomials in the generating valuations is listed in Corollary \ref{cor:basis}.

Let us mention two applications of the main result. First, our information about the algebra suffices to compute explicitly via the ftaig the principal kinematic formula on the octonionic plane:
\begin{align}
	\label{eq:pkf0}
	\int_{\b{\Spin(9)}}\chi(A\cap \b g B)d\b g.
\end{align}
Here $\chi\in\Val_0^{\Spin(9)}$ is the Euler characteristic and $A,B\subset\OO^2$ are any convex bodies. Finally, we use \eqref{eq:pkf0} to express Alesker's valuations $T_k$, $0\leq k\leq 8$,  in our monomial basis.

\subsection{Plan of the paper and notation}

In Section \ref{s:background}, we summarize necessary background about valuations, the octonions, the Lie groups $\Spin(9)$ and $\Spin(7)$ and the Maurer--Cartan form, and collect several general facts from commutative algebra. In Section \ref{s:spin9}, the Lie algebra $\spin(9)$ is described in terms of a system of linear equations using the octonionic structure. Section \ref{s:InvSpin7} is devoted to the extension of the invariant theory of $\Spin(7)$ to the representation $\Imag\OO\oplus\OO$. The generators of the algebra $\left[\largewedge(T_pS\OO^2)^*\right]^{\Spin(7)}$ are constructed in Section \ref{s:Spin7invforms} and interpreted as generators of $\Omega^*(S\OO^2)^{\b{\Spin(9)}}$ in Section \ref{s:Spin9invforms}. The differentiation rules (Theorem \ref{thm:d0}) are proved in Section \ref{s:d}. In  Section \ref{s:main}, our main result is established, i.e., $\Val^{\Spin(9)}$ is explicitly described as an algebra. The applications, namely, the computation of the principal kinematic formula and the expression of the Alesker's valuations in our basis are discussed in Sections \ref{s:pkf} and \ref{s:Tk}, respectively. Finally, in Section \ref{s:Q} we formulate several questions arising from our work.

The following conventions will be adhered to throughout the article: All vector spaces are real and all tensor products are taken over $\RR$. $\Grass_k$ denotes the Grassmannian of $k$-dimensional linear subspaces of $\RR^n$ and $dE$ the Haar probability measure on it. $\calK(\RR^n)$ is the set of convex bodies, i.e., non-empty convex compact subsets of $\RR^n$. On this space, the Minkowski addition is defined as $A+B=\{a+b\mid a\in A,b\in B\}$. $D^n\in\calK(\RR^n)$ is the $n$-dimensional closed Euclidean unit ball, and $\omega_n$  is its volume, in particular, $\omega_{2k}=\frac{\pi^k}{k!}$. Finally, we denote $[0]=1$ and $[k]=\frac{k\omega_k}{2\omega_{k-1}}$, $k\in\NN$, and with $[k]!=[k][k-1]\cdots[1]$ we set $\begin{bmatrix}n\\k\end{bmatrix}=\frac{[n]!}{[k]![n-k]!},\quad 0\leq k\leq n.$

\subsection*{Acknowledgements}

The first-named author is indebted to Andreas Bernig and Gil Solanes for their hospitality, encouragement, and useful discussions at early stages of this work.
 The second-named author wishes to thank Hendrik S\"u\ss{} for helpful discussions about commutative algebra.

\section{Background}
\label{s:background}

\subsection{Smooth valuations on convex bodies}

Let us begin by reviewing basic facts about the space of continuous and translation-invariant valuations, with emphasis on the dense subspace of smooth valuations and certain algebraic structures  it admits. For more information we refer the reader to Schneider's book \cite[Chapter 6]{Schneider:BM}, the lecture notes by Alesker and Fu \cite{AleskerFu:Barcelona}, and the references therein.

In this paper, a valuation is a function $\phi\maps{\calK(\RR^n)}\RR$ satisfying
\begin{align*}
\phi(A\cup B)+\phi(A\cap B)=\phi(A)+\phi(B)
\end{align*}
for any $A,B\in\calK(\RR^n)$ with $A\cup B\in\calK(\RR^n)$. Let $\Val$ denote the linear space of all continuous and translation-invariant valuations on $\calK(\RR^n)$. Basic examples are the Euler characteristic $\chi$, which is constant $1$, and the Lebesgue measure $\vol_n$. More generally, the intrinsic volumes given by the Steiner formula
\begin{align}
\label{eq:Steiner}
\vol_n(A+\lambda D^n)=\sum_{k=0}^n\omega_k\mu_{n-k}(A)\lambda^k,\quad A\in\calK(\RR^n), \quad \lambda \geq 0,
\end{align}
are elements of $\Val$. Observe that, in particular, $\mu_0=\chi$ and $\mu_n=\vol_n$.

 By a result of McMullen \cite{McMullen:EulerType}, the space $\Val=\bigoplus_{k=0}^n \Val_k$ is a Banach space, graded by the degree of homogeneity of a valuation. Moreover, $\Val_k=\Val_k^+\oplus\Val_k^-$ with respect to the parity. Observe that $\mu_k\in\Val_k^+$. Recall also that $\Val_0=\RR\chi$ and, by Hadwiger's characterization of volume, $\Val_n=\RR\vol_n$. Based on the latter, the Klain function $\Klain_\phi\maps{\Grass_k}{\RR}$ of $\phi\in\Val_k^+$ is defined by $\phi|_E=\Klain_\phi(E)\vol_k$. An important result of Klain \cite{Klain:Even} is that $\phi$ is uniquely determined by $\Klain_\phi$.

Consider the natural (left) action of $\GL(n)$ on $\Val$. The fundamental result of Alesker \cite{Alesker:Irreducibility} is that the representations $\Val_k^\pm$ of $\GL(n)$ are topologically irreducible. This has numerous important consequences. First, as conjectured by McMullen \cite{McMullen:Continuous}, the valuations $\vol_n(\Cdot+A)$, $A\in\calK(\RR^n)$, span in $\Val$ a dense subset. Second, denote $S\RR^n=\RR^n\times S^{n-1}$ and recall that the normal cycle $N(A)=\{(y,v)\in S\RR^n\mid \ip{v}{y-z}\geq 0\text{ for all }z\in A\}$ is for any $A\in\calK(\RR^n)$ a naturally oriented $(n-1)$-dimensional Lipschitz submanifold of $S\RR^n$. Then the valuations
\begin{align*}
A\mapsto c\vol_n(A)+\int_{N(A)}\omega,
\end{align*}
where $c\in\RR$ and $\omega\in\Omega^{n-1}(S\RR^n)$ is translation-invariant, are precisely the $\GL(n)$-smooth vectors in $\Val$. Such valuations are called smooth and the dense subspace of them is denoted by $\Val^\infty$. Recall that $\Val^\infty$ is equipped with a natural Fr\'echet space topology (stronger than the Banach space topology induced from $\Val$).

Observe that the space $\Omega^*(S\RR^n)^{\tr}$ of translation-invariant forms is naturally bi-graded and that, via integration over the normal cycle, $(k,n-1-k)$-forms give rise to $k$-homogeneous valuations. An important result due to Bernig and Bröcker determines the kernel of this map. Let $\rho\maps{S\RR^n}{\RR^n}$ be the obvious projection and let
\begin{align}
\label{eq:ContactForm} 
\alpha(X)=\ip{v}{d\rho(X)},\quad X\in T_{p,v} S\RR^n,
\end{align}
be the contact 1-form on $S\RR^n$. For any $\omega\in\Omega^{n-1}(S\RR^n)^{\tr}$ there exists by symplectic linear algebra a unique form $\alpha\wedge\xi\in\Omega^{n-1} (S\RR^n)^{\tr}$ such that $D\omega=d(\omega+\alpha\wedge\xi)$ is a multiple of $\alpha$. $D\omega$ is referred to as the Rumin differential of $\omega$.
\begin{theorem}[Bernig--Bröcker \cite{BernigBroecker:Rumin}]
\label{thm:kernel}
Let $1\leq k\leq n-1$ and $\omega\in\Omega^{k,n-1-k}(S\RR^n)^{\tr}$. The valuation defined by $\omega$ is identically zero if and only if $D\omega=0$.
\end{theorem}

Finally, it is a consequence of the irreducibility theorem that the space $\Val^\infty$ carries the structure of a commutative graded algebra;  more precisely, Bernig and Fu \cite{BernigFu:Convolution} showed that
\begin{align*}
\vol_n(\Cdot+A)*\vol_n(\Cdot+B)=\vol_n(\Cdot+A+B),
\end{align*}
where $A,B\in\calK(\RR^n)$ have smooth and strictly positively curved boundary, defines on $\Val^\infty$ an associative, commutative, continuous, and bilinear product with identity element $\vol_n$, that satisfies $\Val_{n-k}^\infty*\Val_{n-l}^\infty\subset\Val_{n-k-l}^\infty$ and Poincar\'e duality, meaning that the convolution pairing $\Val_k^\infty\times\Val_{n-k}^\infty\to\Val_0^\infty\cong\RR$ is non-degenerate. The convolution admits a simple algebraic formula in terms of differential forms: Let $\phi_i\in\Val^\infty$ be defined by $\omega_i\in\Omega^{n-1}(S\RR^n)^{\tr}$. Then
\begin{align}
\label{eq:convolutionForms}
\phi_1*\phi_2=\int_{N(\Cdot)}*_1^{-1}(*_1\omega_1\wedge*_1D\omega_2),
\end{align}
where $*_1$ acts on $\Omega^{k,n-1-k}(S\RR^n)^{\tr}\cong\largewedge^k(\RR^n)^*\otimes\Omega^{n-1-k}(S^{n-1})$ as $*_1(\sigma\otimes\tau)=(-1)^{{n-k\choose 2}}(*\sigma)\otimes\tau$, $*\sigma$ being the standard Hodge dual of $\sigma\in\largewedge(\RR^n)^*$. Moreover, it is easy to see that
\begin{align}
\label{eq:Dconv}
D\left[*_1^{-1}(*_1\omega_1\wedge*_1D\omega_2)\right]=*_1^{-1}(*_1D\omega_1\wedge*_1D\omega_2).
\end{align}
For completeness, let us point out that there exists another, geometrically different, but algebraically isomorphic  product structure on valuations, namely the Alesker product \cite{Alesker:Product}.

Recall that the intrinsic volumes are smooth and consider the degree $-1$ linear operator on $\Val^\infty$ given by $\Lambda\phi=\mu_{n-1}*\phi$ or equivalently (see \cite[Corollary 1.8]{BernigFu:Convolution})
\begin{align*}
(\Lambda\phi)(A)=\frac12\dlambda \phi(A+\lambda D^n),\quad A\in\calK(\RR^n).
\end{align*}
An immediate consequence of the Steiner formula \eqref{eq:Steiner} is 
\begin{equation}\label{eq:SteinerMu}
\Lambda \mu_k = [ n-k+1] \mu_{k-1},
\end{equation}
which directly implies
\begin{align}
\label{eq:nthpower}
\mu_{n-1}^{n}=\Lambda^n\vol_n=\frac{n!\omega_n}{2^n}\chi.
\end{align}
Although it is not strictly necessary for this paper, let us recall that the operator $\Lambda$ satisfies the hard Leschetz theorem and the Hodge--Riemann relations, see \cite{BernigBroecker:Rumin,KW:HR}.

We conclude the review of valuation theory by recalling a connection between the algebraic structures on $\Val^\infty$ and integral geometry. Let $G\subset\SO(n)$ be a closed subgroup acting transitively on the sphere $S^{n-1}$. Alesker \cite{Alesker:OnConjecture} observed that the subspace $\Val^G\subset\Val$ of $G$-invariant valuations satisfies $\dim\Val^G<\infty$ and $\Val^G\subset\Val^\infty$. Moreover, since the convolution clearly commutes with the action of $\SO(n)$, the latter is a subalgebra. Let $d\b g$ denote the product of the Haar probability measure and Lebesgue measure on $\b G=G\ltimes\RR^n$. The following is a special case of the ftaig:
\begin{theorem}[Fu \cite{Fu:Unitary}]
\label{thm:ftaig}
Let $\phi_1,\dots,\phi_N$ be a basis of $\Val^G$ and let $M=(M_{i,j})\in\RR^{N\times N}$ be the matrix of the Poincar\'e pairing in this basis, i.e., $M_{i,j}\chi$ is the $0$-homogeneous component of $\phi_i*\phi_j$. Then for any $A,B\in\calK(\RR^n)$ one has
\begin{align}
\label{eq:PKF}
\int_{\b G}\chi(A\cap\b gB)d\b g=\sum_{j,k=1}^N(M\inv)_{j,k}\phi_j(A)\phi_k(B).
\end{align}
\end{theorem}

\subsection{The octonions}
\label{ss:OO}

 Below we will consider $\Val^G$  in the case   $G=\Spin(9)$ in some detail. For a  useful concrete description of this Lie group as well as important subgroups and their invariants, let us recall here basic facts about the algebra of octonions. For reference and more about the octonions and related topics we refer  to \cite{Harvey:Spinors,Baez:Octonions,SalamonWalpuski:Notes}.

The octonionic multiplication is a bilinear product on $\RR^8$ that admits a unit and is compatible with the standard Euclidean structure. According to the Hurwitz theorem, such a structure only exists in dimension $d\in\{1,2,4,8\}$ and in each case it is unique up to an isomorphism. Explicitly, one possible way to define the algebra $\OO$ of octonions which we will adopt here is as follows: Denoting by  $e_0,e_1,e_2,\dots,e_7$ the standard (orthonormal) basis of $\RR^8$, we impose
\begin{enuma}
\item $e_0^2=e_0$,
\item $e_0e_i=e_ie_0=e_i$ and $e_i^2=-e_0$ for $i\in\{1,\dots,7\}$,
\item $e_ie_j=-e_je_i$ for $i,j\in \{1,\dots,7\}$ with $i\neq j$, and
\item $e_1e_2=e_4$, $e_2e_4=e_1$, $e_4e_1=e_2$, and so forth, according to the following scheme:
\begin{align*}
\begin{matrix}1&2&3&4&5&6&7\\2&3&4&5&6&7&1\\4&5&6&7&1&2&3\end{matrix}
\end{align*}
\end{enuma}
which we hope is self-explanatory (and easy to memorize). It is straightforward to verify that
\begin{align}
\label{eq:nda}
|xy|^2=|x|^2|y|^2,\quad  x,y\in\OO.
\end{align}
Observe that $\OO$ is nor commutative neither associative.   As usual, we will  interchangeably use the symbols $e_0$ and $1$. Then one has the orthogonal decomposition $\OO=\RR\oplus\Imag\OO$, where $\Imag\OO=1^\perp$ is the subspace of imaginary octonions. Let $\Real\maps{\OO}{\RR}$ and $\Imag\maps{\OO}{\Imag\OO}$ be the orthogonal projections. We define the conjugation of $x\in\OO$ as $\b x=\Real(x)-\Imag(x)$. Conversely, the real and imaginary part are $\Real(x)=\frac12(x+\b x)$ and $\Imag(x)=\frac12(x-\b x)$, respectively. Also note that
	\begin{align}
	\label{eq:xybar}
	\b{xy}=\b y\, \b x.
	\end{align}
and 
\begin{align}
\label{eq:ipbar}
\ip xy=\Real(\b xy)=\Real(x\b y)=\ip{\b x}{\b y}.
\end{align}
In particular, $\sqnorm{x}=\b x x=x\b x$ and hence each $x\neq0$ has a multiplicative inverse.

Important non-trivial algebraic relations hold in $\OO$ in consequence of \eqref{eq:nda}. Namely, for $x,y,z,w\in\OO$ the linear operators $R_w\mape{x}{xw}$ and $ L_w\mape{x}{wx}$ satisfy
\begin{align}
\label{eq:RwRz}
R_w^*=R_{\b w},\quad L_w^*=L_{\b w}, \quad\text{and}\quad R_{\b w}R_z+R_{\b z}R_w=L_{\b w}L_z+L_{\b z}L_w=2\ip wz\id.
\end{align}
Moreover, the associator $[x,y,z]=(xy)z-x(yz)$ is an alternating trilinear map $\OO\times\OO\times\OO\to\OO$ and the famous Moufang identities hold:
\begin{align*}
x(y(xz))&=(xyx)z,\\
((zx)y)x&=z(xyx),\\
(xy)(zx)&=x(yz)x.
\end{align*}
Notice that no additional brackets are needed here as the corresponding associators vanish. Clearly, the associator is also trivial when at least one of its arguments is real. Consequently,
\begin{align}
\label{eq:assocbar}
[x,y,z]=-[\b x,y,z].
\end{align}
In fact, according to the celebrated theorem of Artin, any subalgebra of $\OO$ generated by two elements is associative. Observe that such a subalgebra is in fact generated by imaginary parts of the respective elements or, equivalently, by the two elements {\it and} their conjugates.

\subsection{The Lie groups $\Spin(9)$ and $\Spin(7)$}
\label{ss:Spin97}

Recall that $\Spin(n)$, $n\geq3$, is the two-fold universal covering group of $\SO(n)$. There is a unified approach via Clifford algebras in which all the $\Spin$ groups are constructed simultaneously. Remarkably, for certain $n$ this abstract construction can be made fairly explicit using the contents of the previous section. We will recall here how the group structure can be encoded into the octonionic multiplication if $n=7$ or $n=9$, following \cite[Chapter 14]{Harvey:Spinors}.

To begin with the latter, \cite[Lemma 14.77]{Harvey:Spinors} allows us to define $\Spin(9)$ to be the subgroup of $\GL(\OO^2)$ generated by
\begin{align}
\label{eq:Spin9}
\left\{\begin{pmatrix}r&R_{x}\\R_{\b x}&-r\end{pmatrix}\mid r\in\RR,x\in\OO,r^2+\sqnorm x=1\right\}.
\end{align}
Notice that the generators act on $\OO^2$ blockwise and from the left, as $2\times2$ block operators. Using that they clearly preserve the standard inner product on $\OO^2=\RR^{16}$, one sees that $\Spin(9)\subset\SO(16)$ in fact. Observe also that $\Spin(9)$ indeed acts transitively on $S^{15}\subset\OO^2$ as for any $x,y\in\OO$ with $\sqnorm x+\sqnorm y=1$ one has
\begin{align*}
\begin{pmatrix}x\\y\end{pmatrix}=\begin{cases}\begin{pmatrix}0&R_{x}\\R_{\b x}&0\end{pmatrix}\begin{pmatrix}0&1\\1&0\end{pmatrix}\begin{pmatrix}1\\0\end{pmatrix} &\text{if }y=0,\\[2ex]
\begin{pmatrix}0&R_{\frac{\b y}{\norm y}}\\R_{\frac{y}{\norm y}}&0\end{pmatrix}\begin{pmatrix}\norm{y}&R_{\frac{\b y\,\b x}{\norm y}}\\R_{\frac{xy}{\norm y}}&-\norm{y}\end{pmatrix}\begin{pmatrix}1\\0\end{pmatrix}&\text{if }y\neq0.
\end{cases}
\end{align*}

Let us recall  an equivalent definition of $\Spin(9)$. Using the generators \eqref{eq:Spin9}, it is easy to see that the group preserves the set $\OO P^1=\big\{O_a\mid a\in\OO\cup\{\infty\}\big\}\cong S^8$ of octonionic lines, where \begin{align*}
O_a=\left\{\begin{pmatrix}x\\xa\end{pmatrix}\in\OO^2\mid x\in\OO\right\}, a\in\OO,\quad \text{and}\quad 
O_\infty=\left\{\begin{pmatrix}0\\x\end{pmatrix}\in\OO^2\mid x\in\OO\right\}.
\end{align*}
Over this base, the  Hopf fibration $S^{15}\to \OO P^1$ is given by projecting a point on $S^{15}\subset\OO^2$ to the octonionic line it belongs. Observe that such a line always exists and it is unique as two distinct lines intersect only at the origin. Elements of $\Spin(9)$ can thus be viewed as symmetries of the fibration.  In fact, according to the theorem of Gluck--Warner--Ziller \cite{GWZ:Hopf},  a rotation $g\in\SO(16)$ is an element of $\Spin(9)$ if and only if $g$ takes octonionic lines to octonionic lines.

Proceeding to the case $n=7$, according to \cite[Lemma 14.66]{Harvey:Spinors}, $\Spin(7)$ can be defined as the subgroup of $\GL(\OO)$ generated by
\begin{align}
\label{eq:Spin7}
\{L_u\mid u\in\Imag\OO,|u|=1\}.
\end{align}
Take $x\in\OO$ with $|x|=1$. First, observe that $L_x^*=L_x^{-1}$ and, since $\det(L_1)=1$ and $S^7\subset\OO$ is connected, $\det(L_x)=1$ (and similarly for $R_x$). Hence $\Spin(7)\subset\SO(\OO)$. Second, one has $x=L_{u}L_v(1)$ if $v=-\Real(x)u-u\Imag(x)$ and $u\in\Imag\OO$ is chosen subject to $|u|=1$ and $\ip{u}{\Imag(x)}=0$, in which case $v\in\Imag\OO$ and $|v|=1$. Therefore, $\Spin(7)$ acts transitively on $S^7$. The vector representation of $\Spin(7)$ is defined by 
	\begin{align}
	\label{eq:chig}
	\chi_g(u)=g\!\left(u\cdot g^{-1}(1)\right), \quad u\in \Imag \OO.
	\end{align}
Extending $\chi_g$ to $\OO$ with $\chi_g(1)=1$, the relation 
\begin{align}
\label{eq:Spin7a}
g(xy)=\chi_g(x)g(y).
\end{align}
holds  for any $g\in\Spin(7)$ and $x,y\in\OO$. 
This relation, which will be used repeatedly in the sequel (in particular, the results of Section \ref{s:InvSpin7} directly reflect it), can in fact serve as an equivalent  definition of $\Spin(7)$, see \cite[Lemma 14.61]{Harvey:Spinors}.

There is one more important characterization of $\Spin(7)$. Namely, assume $w,x,y,z\in\OO$ and denote $x\times  y\times z=\frac12\big[(x\b y)z-(z\b y)x\big]$. Then the Cayley calibration is defined as
\begin{align}
\label{eq:CCdef}
\Phi(w,x,y,z)=\ip{w}{x\times y\times z}.
\end{align}
It is straightforward to verify $\Phi\in\largewedge^4\OO^*$ and, using the Moufang identities, that it is $\Spin(7)$-invariant. In fact, $\Spin(7)=\{g\in\GL(\OO)\mid g^*\Phi=\Phi\}$. For details see \cite{HarveyLawson:Calibrated,Bryant:Exceptional}.

\begin{remark}
Strictly speaking, our definition of $\Spin(7)$ differs from the conventions of \cite{Harvey:Spinors,HarveyLawson:Calibrated} where $L_u$ is replaced by $R_u$ in \eqref{eq:Spin7}, \eqref{eq:Spin7a}  by $g(xy)=g(x)\chi_g(y)$ with $\chi_g(y)=g(g^{-1}(1)\cdot y)$, and the Cayley calibration is defined via $x\times  y\times z=\frac12\big[x(\b yz)-z(\b yx)\big]$. However, it is easily seen that this group is just $C\cdot \Spin(7)\cdot C$, where  $ C(x)=\b x$ denotes conjugation in $\OO$.
\end{remark}

Let us conclude this section by explaining why the group $\Spin(7)$ is relevant to our work at all. Observe that for  any $g\in \Spin(7)$ the rotation
\begin{align*}
\begin{pmatrix}\chi_g&0\\0&g\end{pmatrix}\in\SO(\OO^2)
\end{align*}
fixes the point $p=\big(\begin{smallmatrix}1\\0\end{smallmatrix}\big)\in S^{15}$ and belongs by triality \cite[Theorem 14.19]{Harvey:Spinors} to $\Spin(9)$. In fact, according to \cite[Theorem 14.79]{Harvey:Spinors},
\begin{align}
\label{eq:stab}
\Stab_p\Spin(9)=\left\{\begin{pmatrix}\chi_g&0\\0&g\end{pmatrix}\mid g\in\Spin(7)\right\},
\end{align}
i.e., $\Spin(9)/\Spin(7)\cong S^{15}$. Consequently, as $\Spin(7)$-representations,
\begin{align}
\label{eq:O2underSpin7}
\OO^2\cong\RR\oplus\Imag\OO\oplus\OO\quad\text{and}\quad T_pS^{15}\cong\Imag\OO\oplus\OO.
\end{align}

\subsection{The Maurer--Cartan form}
Let us collect here serveral basic facts about the Mauer--Cartan form  $\omega\in \Omega^1(G,\mathfrak g)$, which is a  is Lie-algebra-valued $1$-form on a Lie group $G$. In the special case---which is  sufficient for our purposes---where $G$ is Lie subgroup of $\GL(\RR^n)$ the Maurer--Cartan form is given by
$$ \omega_g= g^{-1}dg, \quad g\in G.$$
It is immediate from this definition that 
$$ L_g^* \omega=\omega \quad \text{and} \quad R_g^* \omega= g^{-1} \omega g,$$
where $L_g$ and $R_g$ denote, respectively, left and right translation by $g \in G$. Moreover, $\omega$ satisfies the Maurer--Cartan equation
$d\omega= - \omega \wedge \omega,$ where $(\omega\wedge \omega)(X,Y)=\omega(X)\omega(Y)- \omega(Y)\omega(X)$ in $\End(\RR^n)$.

Below we are going to apply the above to a group $G= K\ltimes \RR^n$, where $K$ is a Lie subgroup of $\GL(\RR^n)$. Identifying $G$ with $$\left \{ \begin{pmatrix} k &  x \\ 0 & 1\end{pmatrix} \mid k\in K, \ x\in \RR^n\right\} \subset \GL(\RR^{n+1})$$
we see that the Maurer--Cartan form of $G$ is $\omega=(k^{-1}dk, k^{-1} dx)=(\varphi, \theta)$ with $\varphi\in \Omega^1(G,\mathfrak k)$ and $\theta
\in \Omega^1(G,\RR^n)$. The Maurer--Cartan equation can thus be expressed as
\begin{equation}\label{eq:MaurerCartan}  d\varphi = -\varphi\wedge \varphi \quad \text{and}\quad d\theta=- \varphi \wedge \theta.\end{equation}
The behavior under right translations specializes to 
\begin{equation}\label{eq:rightTranslations} R_{(k,0)}^* \varphi= k^{-1} \varphi k\quad \text{and} \quad R_{(k,0)}^*\theta=k^{-1} \theta.\end{equation}

\subsection{An excursion  into  commutative algebra}

\label{ss:CA}

To conclude the background section, let us recall general algebraic constructions which will later be applied in the particular setting of Theorem \ref{thm:main}.  Although we are not aware of any reference, we believe that the statements below are most probably well known. For background from commutative algebra see, e.g., \cite{AtiyahMacDonald:Introduction}.

Recall that any finitely generated, commutative, associative algebra over a field $K$ arises as a quotient of the polynomial ring  by an ideal $I$, which is also finitely generated by Hilbert's basis theorem. The following lemma immediately translates into a computationally cheap algorithm that yields a (possibly infinite) generating set of $I$. 

\begin{lemma}
\label{lem:ideal}
Let $\calA=\bigoplus_{k=0}^\infty \calA_k$ be a commutative, associative, graded algebra with identity element, (finitely) generated by  $\hat g_1,\dots,\hat g_m\in\calA$ with $\deg \hat g_i\in\NN$. Consider the graded algebra
\begin{align*}
R=\bigoplus_{k=0}^\infty R_k=K[ g_1,\dots,g_m]
\end{align*}
and the obvious (graded) morphism $\hat{\Cdot}:R\to\calA$. For $k\in\NN$ denote $G_k=\{g_i\mid \deg g_i=k\}$ and construct inductively subsets $A_k,B_k,F_k\subset R_k$ as follows: Set $B_0=\{1\}$ and
\begin{align*}
A_k&=\{xy\mid x\in B_l,y\in G_{k-l},0\leq l<k\},\\
B_k&=\{x_{k,1},\dots,x_{k,b_k}\}\subset A_k\text{ is a maximal subset such that }\hat B_k\subset \calA_k\text{ is linearly independent,}\\
F_k&=\left\{ x-\sum_{i=1}^{b_k}\lambda_{k,i}(x) x_{k,i}\mid x\in A_k\setminus B_k\right\},
\end{align*}
where $\lambda_{k,i}(x)\in K$ are the unique constants such that $\hat x=\sum_{i=1}^{b_k}\lambda_{k,i}(x)\hat x_{k,i}$. Then $\calA\cong R/I$ as graded algebras, where $I\subset R$ is the ideal generated by $\bigcup_{k=1}^\infty F_k$,  and $\hat B_k$ is a basis of $\calA_k$.
\end{lemma}

\begin{proof}
First of all, observe that for each $k\in\NN$, $\hat B_k$ is a basis of $\linspan\hat A_k$ which implies the existence and uniqueness of the constants $\lambda_{k,i}(x)$ for any $x\in A_k$. 

Since $\hat R=\calA$ and $\hat I=0$, we have an epimorphism $R/I \to\calA$. Denote by $x\mapsto \wt x$ the quotient map $R\to R/I$. We will show by induction that $\wt R_k=\linspan\wt B_k$ holds for any $k\in\NN_0$ which implies $\dim \wt R_k\leq\dim\calA_k$ and therefore yields the claim.

Observe that $\wt A_k\subset \linspan\wt B_k$ holds for $k\in\NN$. Indeed, for each $x\in A_k$ one has either $x\in B_k$ or $x\in A_k\setminus B_k$. In the latter case, $\wt x=\sum_{i=1}^{b_k}\lambda_{k,i}(x)\wt x_{k,i}\in \linspan\wt B_k$ since $x-\sum_{i=1}^{b_k}\lambda_{k,i}(x) x_{k,i}\in I$.

Let us prove $\wt R_k=\linspan\wt B_k$. $\wt R_0=\linspan\wt B_0$ is trivial so assume $k\geq1$ and that $\wt R_l=\linspan\wt B_l$ holds for all $l<k$. It suffices to show that for each monomial $z=g_{i_1}g_{i_2}\cdots g_{i_K}\in R_k$ one has $\wt z\in\linspan \wt B_k$. Clearly, $z=xy$ for some $x\in R_l$, $y\in G_{k-l}$, and $0\leq l<k$. By the induction hypothesis, there are $\alpha_i\in\RR$ with $\wt x=\sum_{i=1}^{b_l}\alpha_i \wt x_{l,i}$. Obviously, $x_{l,i}y\in A_k$ and thus $\wt z\in\linspan\wt A_k\subset\linspan\wt B_k$.
\end{proof}

The so constructed generating set $\bigcup_{k=1}^\infty F_k$ of $I$ is by no means minimal. One way to reduce it is  to remove redundant monomials from $A_k$. Moreover,  by considering on $R$ a total monomial order, we have more control over which monomials end up in $B_k$. Recall that the order extends to any family of ordered sets of monomials of fixed cardinality $K<\infty$ in the lexicographical manner, namely,
\begin{align*}
\{x_1,\dots,x_K\}>\{y_1,\dots,y_K\}\text{ if there is }k\leq K\text{ with }x_i=y_i\text{ for }i<k\text{ and }x_k>y_k.
\end{align*}

\begin{corollary}
\label{cor:ideal}
Fix a monomial order on $R$. Then Lemma \ref{lem:ideal} still holds if $A_k$ is replaced by
\begin{align*}
A_k&=\{xy\mid x\in B_l,y\in G_{k-l},0\leq l<k\}\setminus\{uv\mid u\in A_m\setminus B_m,v\in G_{k-m}, 0<m<k\},
\end{align*}
and $B_k$ is chosen to be of  maximal order.
\end{corollary}

\begin{proof}
It suffices to prove that $\wt x\wt y\in \linspan\wt A_k$ for any $x\in B_l$, $y\in G_{k-l}$ and $0\leq l<k$. Indeed, this fact implies $\wt x_{l,i}\wt y\in \linspan\wt A_k$ and hence $\wt z\in\linspan\wt A_k\subset\linspan\wt B_k$ in the very last step of the proof of Lemma \ref{lem:ideal}. The rest of the proof clearly remains valid.


To prove the claim assume that there exist  $x\in B_l$, $y\in G_{k-l}$ with $0\leq l<k$ such that $\wt x\wt y\notin \linspan\wt A_k$. We may further assume that $xy$ has maximal order among all monomials with this property. By the definition of $A_k$, there exist $u\in A_m\setminus B_m$, $v\in G_{k-m}$ with $0<m<k$ such that $xy=uv$. By the definition of $I$, 
we have $\wt u = \sum_{i=1}^{b_m} \lambda_{m,i}(u) \wt x_{m,i}$. Since $B_m$ was chosen to be of  maximal order we have $ u< x_{m,i}$ and, consequently, $ uv <x_{m,i} v$ if $\lambda_{m,i}(u)\neq 0$. The  maximality of $xy$ then implies $\wt x_{m,i} \wt v\in \linspan \wt A_k$ for such $i$. Thus
$$ \wt x\wt y =  \sum_{i=1}^{b_m} \lambda_{m,i}(u) \wt x_{m,i} \wt v \in  \linspan \wt A_k,$$
contradicting our assumption.
\end{proof}

\begin{remark}
\label{rem:IdealBasis}
Observe that the generating set $\bigcup_{k=1}^\infty F_k$ of $I$ constructed either in Lemma \ref{lem:ideal} or in Corollary \ref{cor:ideal} is finite if $\calA$ is finite-dimensional (as a vector space over $K$).
\end{remark}

 Let us recall a particular order on  $R= K[g_1,\dots,g_m]$ we will use. As usual, we abbreviate $g^a=g_1^{a_1}\cdots g_m^{a_m}$ for $a=(a_1,\dots,a_m)\in\NN_0^m$. The graded lexicographical monomial order on $R$ is given by declaring for any $a,b\in\NN_0^m$ that
\begin{align*}
g^a>g^b \text{ if the first nonzero entry of }\left(\sum_{i=1}^m (a_i-b_i)\deg g_i,a_1-b_1,\dots,a_m-b_m\right) \text{ is positive}.
\end{align*}

Finally, using a modification of the argument of Lemma \ref{lem:ideal},  it is straightforward to verify that any finite generating set of the ideal $I$ can be reduced to a minimal one as follows.

\begin{lemma}
\label{lem:minimal}
Let $I=(f_1,\dots,f_r)$ be a homogeneous ideal in $K[ g_1,\dots,g_m]=\bigoplus_{k=0}^\infty R_k$ where $\deg g_i \in\NN$. For $k\in\NN_0$ set $ G_k= \{ g_i\colon \deg g_i =k\}$ and $F_k=\{ f_i \colon \deg f_i = k\}$, and construct inductively subsets $A_k, B_k',B_k,M_k\subset R_k$ as follows:
	\begin{align*}
		A_k&=\{xy\mid x\in B_l,y\in G_{k-l},0\leq l<k\},\\
		B_k'&\subset A_k\text{ is a maximal  linearly independent subset},\\
		M_k&\subset F_k \text{ is maximal with the property that }B_k:=B'_k \cup M_k \text{ is linearly independent}.
	\end{align*}
Then $M=\bigcup_{k=1}^\infty M_k$ is a minimal set of homogeneous generators of $I$.	
\end{lemma}

\section{The Lie algebra $\spin(9)$ revisited}

\label{s:spin9}

An explicit description of the Lie algebra of $\Spin(9)$ is well known in terms of a basis. Namely, consider the following nine elements of the generating set \eqref{eq:Spin9}:
\begin{align*}
\I_i=\begin{pmatrix}0&R_{e_i}\\R_{\b{e_i}}&0\end{pmatrix}, i\in\{0,\dots,7\},\quad\text{and}\quad\I_8=\begin{pmatrix}1&0\\0&-1\end{pmatrix},
\end{align*}
and denote $\I_{i,j}=\I_i\I_j \in\End(\OO^2)$. Then
\begin{align*}
\spin(9)=\linspan\{\I_{i,j}\mid0\leq i<j\leq 8\},
\end{align*}
see \cite[Proposition 8]{PartonPiccinni:Spin9} and \cite[\S2.1]{CastrillonLopez:Canonical8Form}. The purpose of this section is to describe $\spin(9)$ in terms of linear relations that define it as a  subspace of $\End(\OO^2)$.

Identifying $\OO^2=\RR^{16}$, it will be convenient to use the following index convention:
\begin{align}
\label{eq:indexconvention}
x=\begin{pmatrix}x_0^0\\\vdots\\x_0^7\\[1ex]x_1^0\\\vdots\\x_1^7\end{pmatrix}\in  \OO^2
\quad\text{and}\quad
A=\begin{pmatrix}
A_{0,0}^{0,0}&\cdots&A_{0,0}^{0,7}&A_{0,1}^{0,0}&\cdots&A_{0,1}^{0,7}\\
\vdots&&\vdots&\vdots&&\vdots \\
A_{0,0}^{7,0}&\cdots&A_{0,0}^{7,7}&A_{0,1}^{7,0}&\cdots&A_{0,1}^{7,7}\\[1ex]
A_{1,0}^{0,0}&\cdots&A_{1,0}^{0,7}&A_{1,1}^{0,0}&\cdots&A_{1,1}^{0,7}\\
\vdots&&\vdots&\vdots&&\vdots \\
A_{1,0}^{7,0}&\cdots&A_{1,0}^{7,7}&A_{1,1}^{7,0}&\cdots&A_{1,1}^{7,7}
\end{pmatrix}\in  \End(\OO^2),
\end{align}
and to equip the set of indices corresponding to $\RR^8$ with certain algebraic structures naturally compatible with the algebra $\OO$. To this end, consider the set $\calJ=\{\pm0,\pm1,\dots,\pm7\}$ of sixteen formally distinct symbols, i.e., where also $0\neq -0$ and let us define two involutions on it. First, for $a\in\{0,\dots,7\}$ we put $-(\pm a)=\mp a$, second, for $a\in\calJ$,
\begin{align*}
\b a=\begin{cases}a,\quad&\text{if }a=\pm0,\\-a,&\text{if }a\neq\pm0.\end{cases}
\end{align*}
Extending these notions to the standard orthonormal basis of $\OO$ by $e_{-a}=-e_a$, it is obvious that $e_{\b a}=\b{e_a}$ holds for any $a\in\calJ$. Similarly we extend this to any matrix $A=(A^{a,b})_{a,b=0}^7$ by
\begin{align*}
A^{-a,b}=A^{a,-b}=-A^{a,b}=-A^{-a,-b},\quad a,b\in\{0,\dots,7\}.
\end{align*}
Finally, we define a product on $\calJ$: Let $a\cdot b=ab$ be the (unique) element of $\calJ$ with $e_{ab}=e_ae_b$. Importantly, the identity \eqref{eq:xybar} then holds verbatim also on $\calJ$.

\begin{lemma}
If $A\in\spin(9)$, then for all $a,b\in \{0,\dots,7\}$ and $k,l\in\{0,1\}$ one has
\begin{align}
\label{eq:so16}
A_{k,l}^{a,b}&=-A_{l,k}^{b,a},\\
\label{eq:spin9a}
A_{1,0}^{a,b}&=A_{1,0}^{\b ba,0},\\
\label{eq:spin9b}
A_{0,0}^{a,b}-A_{0,0}^{\b ba,0}&=A_{1,1}^{a,b}-A_{1,1}^{\b ba,0},\\
\label{eq:spin9c}
\sum_{c=0}^7A_{1,1}^{ac,c}&=4A_{0,0}^{a,0}.
\end{align}
\end{lemma}

\begin{proof}
First, as $\Spin(9)\subset\SO(16)$, $\spin(9)\subset\so(16)$ which is equivalent to \eqref{eq:so16}.

Second, if $j<8$, then $(\I_{i,j})_{1,0}^{a,b}=0$ holds for any $a,b$. Thus $A_{1,0}^{a,b}=\sum_{i=0}^7\alpha_{i,8}(\I_{i,8})_{1,0}^{a,b}$ for some $\alpha_{i,8}\in\RR$, and \eqref{eq:spin9a} follows from
\begin{align*}
(\I_{i,8})_{1,0}^{a,b}=(R_{\b{e_i}})^{a,b}=\ip{e_a}{e_b\b{e_i}}=\ip{\b{e_b}e_a}{\b{e_i}}=(R_{\b{e_i}})^{\b ba,0}=(\I_{i,8})_{1,0}^{\b ba,0}.
\end{align*}

Third, $(\I_{i,8})_{0,0}^{a,b}=(\I_{i,8})_{1,1}^{a,b}=0$ for all $a,b,i$. Hence $A_{0,0}^{a,b}=\sum_{0\leq i<j\leq7}\alpha_{i,j}(\I_{i,j})_{0,0}^{a,b}$ for some $\alpha_{i,j}\in\RR$, and similarly for $A_{1,1}^{a,b}$. Then, assuming $j\leq7$, one has
\begin{align*}
(\I_{i,j})_{0,0}^{a,b}-(\I_{i,j})_{0,0}^{\b ba,0}&=(R_{e_i}R_{\b{e_j}})^{a,b}-(R_{e_i}R_{\b{e_j}})^{\b ba,0} \\
&=\ip{e_a}{(e_b\b{e_j})e_i}-\ip{\b{e_b}e_a}{\b{e_j}e_i}\\
&=\ip{e_a}{(e_b\b{e_j})e_i}-\ip{e_a}{e_b(\b{e_j}e_i)}\\
&=\ip{e_a}{[e_b,\b{e_j},e_i]}\\
&=\ip{e_a}{[e_b,e_j,\b{e_i}]}\\
&=\ip{e_a}{(e_b e_j)\b{e_i}}-\ip{e_a}{e_b(e_j\b{e_i})}\\
&=\ip{e_a}{(e_b e_j)\b{e_i}}-\ip{\b{e_b}e_a}{e_j\b{e_i}}\\
&=(R_{\b{e_i}}R_{e_j})^{a,b}-(R_{\b{e_i}}R_{e_j})^{\b ba,0} \\
&=(\I_{i,j})_{1,1}^{a,b}-(\I_{i,j})_{1,1}^{\b ba,0},
\end{align*}
using \eqref{eq:assocbar} for the middle equation, and so \eqref{eq:spin9b} follows.

Finally, for $1\leq j\leq7$ we can write
\begin{align*}
\sum_{c=0}^7(\I_{0,j})_{1,1}^{ac,c}&=\sum_{c=0}^7(R_{e_j})^{ac,c} \\
&=\sum_{c=0}^7\ip{e_ae_c}{e_ce_j}\\
&=\sum_{c\in\{0,j\}}\ip{e_ae_c}{e_je_c}-\sum_{c\notin\{0,j\}}\ip{e_ae_c}{e_je_c}\\
&=(2-6)\ip{e_a}{e_j}\\
&=4\ip{e_a}{\b{e_j}}\\
&=4(R_{\b{e_j}})^{a,0}\\
&=4(\I_{0,j})_{0,0}^{a,0}.
\end{align*}
Similarly, if $1\leq i<j\leq 7$, then $0,i,j,\norm{ij}$ are pairwise distinct and
\begin{align*}
\sum_{c=0}^7(\I_{i,j})_{1,1}^{ac,c}&=\sum_{c=0}^7(R_{\b{e_i}}R_{e_j})^{ac,c} \\
&=\sum_{c=0}^7\ip{e_ae_c}{(e_ce_j)\b{e_i}}\\
&=\sum_{c\in\{0,i,j,\norm{ ij}\}}\ip{e_ae_c}{e_c(e_j\b{e_i})}-\sum_{c\notin\{0,i,j,\norm{ ij}\}}\ip{e_ae_c}{e_c(e_j\b{e_i})}\\
&=\sum_{c\in\{0,\norm{ ij}\}}\ip{e_ae_c}{(e_j\b{e_i})e_c}-\sum_{c\in\{i,j\}}\ip{e_ae_c}{(e_j\b{e_i})e_c}+\sum_{c\notin\{0,i,j,\norm{ ij}\}}\ip{e_ae_c}{(e_j\b{e_i})e_c}\\
&=(2-2+4)\ip{e_a}{e_j\b{e_i}}\\
&=4\ip{e_a}{\b{e_j}e_i}\\
&=4(R_{e_i}R_{\b{e_j}})^{a,0}\\
&=4(\I_{i,j})_{0,0}^{a,0},
\end{align*}
where for the third equality we have used $(e_ce_j)\b{e_i}= - e_c(e_j\b{e_i})$ for $c\notin\{0,i,j,\norm{ ij}\}$. The latter is readily verified using \eqref{eq:RwRz}.
	
Altogether, $\sum_{c=0}^7(\I_{i,j})_{1,1}^{ac,c}=4(\I_{i,j})_{0,0}^{a,0}$ holds whenever $0\leq i<j\leq 7$ and \eqref{eq:spin9c} follows.
\end{proof}

\begin{corollary}
\label{cor:spin9}
Let $A\in \End(\OO^2)$. $A\in\spin(9)$ if and only if the following relations hold:
\begin{alignat}{2}
\label{eq:A10ab}
A_{1,0}^{a,b}&=A_{1,0}^{\b ba,0},&\quad &0\leq a\leq 7,\quad 1\leq b\leq 7;\\
A_{0,1}^{a,b}&=-A_{1,0}^{\b ab,0},&\quad &0\leq a,b\leq 7;\\
A_{k,k}^{a,a}&=0, &\quad&0\leq a\leq 7,\quad 0\leq k\leq 1;\\
A_{1,1}^{a,b}&=-A_{1,1}^{b,a},&\quad &1\leq a<b\leq 7;\\
\label{eq:A11a0}
A_{1,1}^{a,0}&=2A_{0,0}^{a,0}-\frac12\sum_{c\neq 0,a}A_{1,1}^{ac,c},&\quad& 1\leq a\leq7;\\
A_{1,1}^{0,a}&=-2A_{0,0}^{a,0}+\frac12\sum_{c\neq 0,a}A_{1,1}^{ac,c},&\quad& 1\leq a\leq7;\\
\label{eq:A00ab}
A_{0,0}^{a,b}&=-A_{0,0}^{\b ba,0}+A_{1,1}^{a,b}+\frac12\sum_{c\neq 0,\norm{\b ba}}A_{1,1}^{(\b ba)c,c},&\quad&1\leq b<a\leq 7;\\
A_{0,0}^{a,b}&=A_{0,0}^{\b ab,0}-A_{1,1}^{b,a}-\frac12\sum_{c\neq 0,\norm{\b ab}}A_{1,1}^{(\b ab)c,c},&\quad&1\leq a<b\leq 7;\\
\label{eq:A000a}
A_{0,0}^{0,a}&=-A_{0,0}^{a,0},&\quad&1\leq a \leq 7.
\end{alignat}
\end{corollary}

\begin{proof}
For `only if', it only remains to show \eqref{eq:A11a0} and \eqref{eq:A00ab}. Indeed, \eqref{eq:A10ab} is just a special case of \eqref{eq:spin9a} and the other relations then follow easily using skew-symmetry \eqref{eq:so16}. First, if $a\neq0$, then by \eqref{eq:spin9c}
\begin{align*}
4A_{0,0}^{a,0}&=\sum_{c=0}^7A_{1,1}^{ac,c}\\
&=A_{1,1}^{a^2,a}+A_{1,1}^{a,0}+\sum_{c\neq0,a}A_{1,1}^{ac,c}\\
&=-A_{1,1}^{0,a}+A_{1,1}^{a,0}+\sum_{c\neq0,a}A_{1,1}^{ac,c}\\
&=2A_{1,1}^{a,0}+\sum_{c\neq0,a}A_{1,1}^{ac,c}.
\end{align*}
Second, for $1\leq b<a\leq 7$, $\norm{\b ba}\neq0$ and so, by \eqref{eq:spin9b} and \eqref{eq:A11a0},
\begin{align*}
A_{0,0}^{a,b}&=A_{0,0}^{\b ba,0}+A_{1,1}^{a,b}-A_{1,1}^{\b{ ba},0} \\
&=A_{0,0}^{\b ba,0}+A_{1,1}^{a,b}-2A_{0,0}^{\b ba,0}+\frac12\sum_{c\neq0,\b ba}A_{1,1}^{(\b ba)c,c}\\
&=-A_{0,0}^{\b ba,0}+A_{1,1}^{a,b}+\frac12\sum_{c\neq0,\b ba}A_{1,1}^{(\b ba)c,c}.
\end{align*}

As for the opposite direction, it is evident that any $A\in \End(\OO^2)$ satisfying \eqref{eq:A10ab}--\eqref{eq:A000a} lies in the intersection of kernels of $56+64+16+21+7+7+21+21+7=220$
linearly independent linear functionals. From this fact the claim follows at once by comparing dimensions.
\end{proof}

\section{Invariant theory of $\Spin(7)$}

\label{s:InvSpin7}

Later on, $\Spin(9)$-invariant valuations will be studied by means of invariant differential forms on the sphere bundle $S\OO^2$. By transitivity, the invariance descends to that under the stabilizer, acting on the corresponding tangent plane. Therefore, according to \S\ref{ss:Spin97}, the first step in the construction is to understand the spaces
\begin{align*}
P_{l,m}=\left[(\Imag\OO^*)^{\otimes l}\otimes(\OO^*)^{\otimes m}\right]^{\Spin(7)}.
\end{align*}
To this end, we will prove in this section the first fundamental theorem (fft) for the isotropy representation $\Imag\OO\oplus\OO$ of $\Spin(7)$, i.e., we will describe the algebra $\bigoplus_{l,m}P_{l,m}$ in terms of its generators. For the sake of readability, the letter `$u$' will always refer to an element of the vector representation $\Imag\OO$, while `$x$' to an element of the spin representation $\OO$.

Let us recall that the the structure of $P_{l,m}$ is well known in two special cases. First, since the representation of $\Spin(7)$ on $\Imag\OO$ factors through the standard representation of $\SO(7)$, the fft is classical for the subalgebra corresponding to $m=0$ (see, e.g., \cite[Proposition F.15]{FultonHarris:RT}).
\begin{theorem}
\label{thm:FFTvector}
Let $l\geq 0$. $P_{l,0}$ is spanned by products of
\begin{alignat*}{2}
&\ip{u_{j_1}}{u_{j_2}},&\quad&1\leq j_1<j_2\leq l;\\
&\det(u_{j_1},\dots,u_{j_7}),&\quad&1\leq j_1<\cdots<j_7\leq l.
\end{alignat*}
\end{theorem}

Let us normalize $\det(e_1,\dots,e_7)=1$ for the standard basis of $\Imag\OO$. Second, as for the case $l=0$, the fft for the spin representation of $\Spin(7)$ was proven by Schwarz.
\begin{theorem}[Schwarz \cite{Schwarz:InvariantG2andSpin7}]
\label{thm:FFTspin}
Let $m\geq0$. $P_{0,m}$ is spanned by products of
\begin{alignat*}{2}
&\ip{x_{k_1}}{x_{k_2}},&\qquad&1\leq k_1<k_2\leq m;\\
&\Phi(x_{k_1},x_{k_2},x_{k_3},x_{k_4}),&&1\leq k_1<\cdots< k_4\leq m.
\end{alignat*}
\end{theorem}

Here we will interpolate between these two results, following the general strategy towards invariant theory of a general representation outlined in \cite[\S6.8]{Procesi:LieGroups}. The following simple but important observation is in the background of our construction.

\begin{proposition}
\label{pro:Spin7equivariant}
The linear maps
\begin{align*}
j\colon \Imag\OO\rightarrow\End(\OO)\cong \OO\otimes \OO, \quad u\mapsto L_u
\end{align*}
and
\begin{align*}
\pi\colon \OO\otimes\OO\rightarrow\Imag\OO, \quad   x\otimes y\mapsto \frac 12 (x\b y-y\b x)
\end{align*}
are $\Spin(7)$-equivariant and $\pi\circ j= id$.
\end{proposition}

\begin{proof}
An easy consequence of the Moufang identities.
\end{proof}

\begin{lemma}
\label{lem:mapFlm}
The map $\calF_{l,m}\maps{P_{l-1,m+2}}{P_{l,m}}$
given by
\begin{align}
\label{eq:mapFlm}
(\calF_{l,m} p)(u_1,\dots,u_l ,x_1,\dots,x_m)=\sum_{i=0}^7p(u_1,\dots,u_{l-1},u_lf_i,f_i,x_1,\dots,x_m),
\end{align}
where $f_0,\dots,f_7$ is some orthonormal basis of $\OO$, is well defined, linear, and onto.
\end{lemma}

\begin{proof}
 By Proposition~\ref{pro:Spin7equivariant}, the linear map 	$(\Imag\OO^*)^{\otimes (l-1)}\otimes(\OO^*)^{\otimes (m+2)}\to (\Imag\OO^*)^{\otimes l}\otimes(\OO^*)^{\otimes m}$ defined by
$$ \alpha_1\otimes \cdots \otimes \alpha_{l-1} \otimes \omega_1 \otimes \cdots \otimes \omega_{m+2}\mapsto 
\alpha_1\otimes \cdots \otimes \alpha_{l-1} \otimes j^*(\omega_1 \otimes \omega_2)\otimes \omega_3 \cdots \otimes \omega_{m+2}$$
is 	$\Spin(7)$-equivariant  and hence restricts to a well-defined linear map  $\calF_{l,m}\maps{P_{l-1,m+2}}{P_{l,m}}$.
To prove surjectivity, note that the linear map 	$(\Imag\OO^*)^{\otimes l}\otimes(\OO^*)^{\otimes m}\to (\Imag\OO^*)^{\otimes (l-1)}\otimes(\OO^*)^{\otimes (m+2)}$ defined by
$$ \alpha_1\otimes \cdots \otimes \alpha_{l} \otimes \omega_1 \otimes \cdots \otimes \omega_{m}\mapsto 
\alpha_1\otimes \cdots \otimes \alpha_{l-1} \otimes \pi^*\alpha_l\otimes \omega_1 \cdots \otimes \omega_{m}$$
restricts to a map $P_{l,m}\rightarrow P_{l-1,m+2}$ which is right inverse to $\calF_{l,m}$.
\end{proof}

\begin{corollary}
\label{cor:calF}
The map $\calG_{l,m}=\calF_{l,m}\circ\calF_{l-1,m+2}\circ\cdots\circ\calF_{1,m+2l-2}\maps{P_{0,m+2l}}{P_{l,m}}$ is onto.
\end{corollary}

Explicitly, for the (standard) orthonormal basis $e_0,\dots,e_7$ of $\OO$ one has
\begin{align}
\label{eq:mapGlm}
(\calG_{l,m} p)(u_1,\dots,u_l,x_1,\dots,x_m)=\sum_{i_1,\dots,i_l=0}^7p(u_1e_{i_1},e_{i_1},\dots,u_le_{i_l},e_{i_l},x_1,\dots,x_m).
\end{align}
Let us show two more simple auxiliary statements before we will finally proceed to the proof of the general first fundamental theorem.

\begin{lemma}
\label{pro:IPcontraction}
For any $u_1,\dots,u_r\in\Imag\OO$ and $x\in\OO$ one has
\begin{align*}
\sum_{i_1,\dots,i_r=0}^7 u_1e_{i_1}\ip{e_{i_1}}{u_2e_{i_2}}\cdots\ip{e_{i_{r-1}}}{u_re_{i_r}}\ip{e_{i_r}}{x}=L_{u_1}\cdots L_{u_r}(x).
\end{align*}
\end{lemma}

\begin{proof}
This follows at once by iterating the orthogonal decomposition $\sum_i\ip{e_i}ae_i$ of $a\in\OO$.
\end{proof}

\begin{lemma}
For any $w,x,y,z\in\OO$ one has
\begin{align}
\label{eq:PhiIP}
\Phi(w,x,y,z)&=\ip{w}{(x\b{y})z}-\ip{w}{x}\ip{y}{z}+\ip{w}{y}\ip{x}{z}-\ip{w}{z}\ip{x}{y}.
\end{align}
\end{lemma}

\begin{proof}
Using \eqref{eq:RwRz} and \eqref{eq:ipbar}, one gets
\begin{align*}
2 x\times y\times z-(x\b{y})z&=-(z\b{y})x\\
&=(y\b{z})x-2\ip{y}{z}x\\
&=-(y\b{x})z-2\ip{y}{z}x+2\ip{x}{z}y\\
&=(x\b{y})z-2\ip{y}{z}x+2\ip{x}{z}y-2\ip{x}{y}z,
\end{align*}
hence $x\times y\times z=(x\b{y})z-\ip{y}{z}x+\ip{x}{z}y-\ip{x}{y}z$, and \eqref{eq:PhiIP} follows.
\end{proof}

\begin{theorem}
\label{thm:Plm}
$P_{l,m}$ is spanned by products of the following functions:
\begin{alignat}{2}
\label{eq:invb}&\ip{u_{j_1}}{u_{j_2}},&\quad&1\leq j_1<j_2\leq l;\\
&\det(u_{j_1},\dots,u_{j_7}),&\quad&1\leq j_1<\cdots<j_7\leq l;\\
&\ip{x_{k_1}}{L_{u_{j_1}}\cdots L_{u_{j_r}}(x_{k_2})},&\qquad&0\leq r\leq 7,\quad1\leq j_1<\cdots<j_r\leq l,\quad 1\leq k_1<k_2\leq m;\\
&\Phi(x_{k_1},x_{k_2},x_{k_3},x_{k_4}),&&1\leq k_1<\cdots< k_4\leq m;\\
\label{eq:inve}&\Phi(u_jx_{k_1},x_{k_2},x_{k_3},x_{k_4}),&&1\leq k_1<\cdots< k_4\leq m,\quad 1\leq j\leq l;
\end{alignat}
$u_i\in\Imag\OO$ and $x_i\in\OO$, such that each of the variables $u_1,\dots,u_l,x_1,\dots,x_m$ occurs exactly once.
\end{theorem}

\begin{proof}
Let $p\in P_{0,m+2l}$ be a product of
\begin{alignat*}{2}
&\ip{x_{k_1}}{x_{k_2}},&\quad&1\leq k_1<k_2\leq m+2l,\\
&\Phi(x_{k_1},x_{k_2},x_{k_3},x_{k_4}),&&1\leq k_1<\cdots<k_4\leq m+2l.
\end{alignat*}
Note that since $p$ is multilinear, each $x_i$ appears exactly once in such a product. According to Theorem \ref{thm:FFTspin} and Corollary \ref{cor:calF}, $\calG_{l,m} p\in P_{l,m}$ and, furthermore, the latter space is spanned by elements of this type. By \eqref{eq:mapGlm}, the map $\calG_{l,m}$ merges the factors of $p$ into (possibly branched or closed) `chains' by plugging $u_je_{i_j}$ in some factor, $e_{i_j}$ in another, and summing over $i_j$. Since for any non-zero $u\in\Imag\OO$, $\frac{u}{|u|}e_1,\dots, \frac{u}{|u|}e_7$ is another orthonormal basis of $\OO$ and since $u^2=-|u|^2$, swapping  $u_je_{i_j}$ and $e_{i_j}$ only affects the sign of $p$. Observe also that the number of $x$'s in any chain is necessarily even. In the rest of the proof, we will investigate systematically which chains in general occur in $\calG_{l,m} p$,  distinguishing four cases. All summations are taken from 0 to 7 if not stated otherwise. 
\begin{enuma}

\item \emph{Chains containing no $x$.} Such products belong to $P_{r,0}$ for some $r\geq1$. According to Theorem \ref{thm:FFTvector}, these must be polynomials in inner products and determinants on $\Imag\OO$.

\item \emph{Chains without $\Phi$ that contain two $x$'s.} Up to sign, these are of the form
\begin{align*}
\sum_{i_1,\dots,i_r}\ip{x_{k_1}}{u_{j_1}e_{i_1}}\ip{e_{i_1}}{u_{j_2}e_{i_2}}\cdots\ip{e_{i_{r-1}}}{u_{j_r}e_{i_r}}\ip{e_{i_r}}{x_{k_2}},\quad r\geq 0,
\end{align*}
where for $r=0$ we have $\ip{x_{k_1}}{x_{k_2}}$, or equivalently, by Proposition \ref{pro:IPcontraction}, $\ip{x_{k_1}}{L_{u_{j_1}}\cdots L_{u_{j_r}}(x_{k_2})}$. Let us explain why there is no loss of generality in assuming $r\leq7$. Suppose $r=8$. Then
\begin{align*}
\sum_{\pi\in\calS_8}\sgn(\pi)\, L_{u_{j_{\pi(1)}}}\cdots L_{u_{j_{\pi(8)}}}=0
\end{align*}
since $\dim\Imag\OO=7$. Further, since $L_uL_v+L_vL_u=-2\ip uv\id$ for $u,v\in\Imag\OO$,
\begin{align*}
\sgn(\pi)\, L_{u_{j_{\pi(1)}}}\cdots L_{u_{j_{\pi(8)}}}=L_{u_{j_{1}}}\cdots L_{u_{j_{8}}}+p_\pi(u_{j_1},\dots,u_{j_8})
\end{align*}
holds for any $\pi\in\calS_8$, where $p_\pi$ is some $\End(\OO)$-valued polynomial in inner products and compositions of less than eight left multiplications. Together, we have
\begin{align*}
0=\sum_{\pi\in\calS_8}\sgn(\pi)\, L_{u_{j_{\pi(1)}}}\cdots L_{u_{j_{\pi(8)}}}=8!\,L_{u_{j_{1}}}\cdots L_{u_{j_{8}}}+\sum_{\pi\in\calS_8}\sgn(\pi)\,p_\pi(u_{j_1},\dots,u_{j_8}),
\end{align*}
i.e., $L_{u_{j_{1}}}\cdots L_{u_{j_{8}}}$ expressed in terms of inner products and at most seven $L$'s. Inductively, the same statement holds for any composition $L_{u_{j_1}}\cdots L_{u_{j_r}}$ with $r\geq8$. Similarly we may assume that $j_1<\cdots<j_r$ and $k_1<k_2$.

\item Note that (a) and (b)  together cover all chains containing no Cayley calibrations. Due to higher complexity of the other cases, let us formalize the reductive method we used in item (b). Namely, let us introduce an equivalence relation on $P_{s,r}$ as follows: We put $p\sim q$ if $p-q$ is expressible in the invariants \eqref{eq:invb}--\eqref{eq:inve}. For instance, \eqref{eq:PhiIP} implies
\begin{align}
\label{eq:Phisim}
\Phi(x_{k_1},x_{k_2},x_{k_3},x_{k_4})\sim\ip{x_{k_1}}{(x_{k_2}\b{x_{k_3}})x_{k_4}}.
\end{align}
Similarly, the following consequence of \eqref{eq:Phisim} and \eqref{eq:RwRz} holds:
\begin{align}
\label{eq:Phiusim}
\Phi(x_{k_1},u_j x_{k_2},x_{k_3},x_{k_4})\sim\Phi(x_{k_1},x_{k_2},u_j x_{k_3},x_{k_4}),
\end{align}
and extends to other pairs of entries by skew-symmetry.

\item \emph{Chains with precisely one $\Phi$.} First, if the chain contains two $x$'s, according to Proposition \ref{pro:IPcontraction} and \eqref{eq:Phiusim}, it is necessarily equivalent, for some $r\geq0$, to
\begin{align*}
\sum_i\Phi\big(e_i,e_i,L_{u_{j_r}}\cdots L_{u_{j_r}}(x_{k_1}),x_{k_2}\big)=0.
\end{align*}
Clearly the same applies when there is no $x$. Similarly, any chain with four $x$'s is equivalent to
\begin{align*}
\Phi\big(L_{u_{j_1}}\cdots L_{u_{j_r}}(x_{k_1}),x_{k_2},x_{k_3},x_{k_4}\big), \quad r\geq 0.
\end{align*}
If $r=2$, however, we get
\begin{align*}
&2\Phi\big(x_{k_1},u_{j_1}(u_{j_2}x_{k_2}),x_{k_3},x_{k_4}\big)\\
&\qquad\sim\Phi\big(x_{k_1},u_{j_1}(u_{j_2}x_{k_2}),x_{k_3},x_{k_4}\big)+\Phi\big(x_{k_1},u_{j_2}(u_{j_1}x_{k_2}),x_{k_3},x_{k_4}\big)\\
&\qquad=\Phi\big(x_{k_1},{(L_{u_{j_1}}L_{u_{j_2}} + L_{u_{j_2}}L_{u_{j_1}})(x_{k_2})},x_{k_3},x_{k_4}\big)\\
&\qquad=-2\ip{u_{j_1}}{u_{j_2}}\Phi(x_{k_1},x_{k_2},x_{k_3},x_{k_4})\\
& \qquad\sim0.
\end{align*}
By induction this easily extends also to $r>2$ and thus we may assume $r\leq1$.

\item \emph{Chains with several $\Phi$'s.} Observe that $\Phi(x_{k_1}\times x_{k_2}\times x_{k_3},x_{k_4},x_{k_5},x_{k_6})\in P_{0,6}$ and thus Theorem \ref{thm:FFTspin} yields
\begin{align*}
\Phi(x_{k_1}\times x_{k_2}\times x_{k_3},x_{k_4},x_{k_5},x_{k_6})\sim0.
\end{align*}
Consequently, we have
\begin{align*}
\sum_i\Phi(u_je_i,x_{k_1},x_{k_2},x_{k_3})\Phi(e_i,x_{k_4},x_{k_5},x_{k_6})&=\sum_i\ip{u_je_i}{x_{k_1}\times x_{k_2}\times x_{k_3}}\ip{e_i}{x_{k_4}\times x_{k_5}\times x_{k_6}}\\
&=-\sum_i\ip{e_i}{u_j(x_{k_1}\times x_{k_2}\times x_{k_3})}\ip{e_i}{x_{k_4}\times x_{k_5}\times x_{k_6}}\\
&=-\ip{u_j(x_{k_1}\times x_{k_2}\times x_{k_3})}{x_{k_4}\times x_{k_5}\times x_{k_6}}\\
&=-\Phi\big(u_j(x_{k_1}\times x_{k_2}\times x_{k_3}),x_{k_4}, x_{k_5}, x_{k_6}\big)\\
&\sim-\Phi(x_{k_1}\times x_{k_2}\times x_{k_3},u_jx_{k_4}, x_{k_5}, x_{k_6})\\
&\sim0.
\end{align*}
Using Proposition \ref{pro:IPcontraction}, it follows easily by induction that any chain can be reduced to chains containing at most one $\Phi$. This completes the proof.\qedhere
\end{enuma}
\end{proof}

\section{$\Spin(7)$-invariant alternating forms}

\label{s:Spin7invforms}

Based on the first fundamental theorem (Theorem \ref{thm:Plm}) proved in the previous section, the algebra of $\Spin(7)$-invariant alternating forms on the space $\Imag\OO\oplus\OO\oplus\Imag\OO\oplus\OO$ will be now described in terms of a generating set. In order to express the generators in a closed form, the crucial step is to embed the space into a larger algebra of octonion-valued forms.

Let $V$ be a finite-dimensional, real vector space. The space $\OO\otimes \largewedge V^*$ is naturally a graded algebra. Its product will again be denoted by the wedge symbol, even though it is no longer associative. Observe that $\largewedge V^*\subset \OO\otimes \largewedge V^*$ is an associative subalgebra. For any $\RR$-linear function $F\maps\OO\OO$ and $u\otimes\varphi\in\OO\otimes\largewedge V^*$ we define $F(u\otimes\phi)=F(u)\otimes\phi$ (in what follows, the tensor product symbol in expressions of this type will usually be suppressed). Observe that any octonion-valued forms $\alpha,\beta\in\OO\otimes\largewedge V^*$ of degree $k$ and $l$ satisfy
\begin{align}
\label{eq:conjAB}
\b{\alpha\wedge\beta}&=(-1)^{kl}\b\beta\wedge\b\alpha,\\
\label{eq:RealAB}
 \Real(\alpha\wedge\beta)&=\Real(\b\alpha\wedge\b\beta).
\end{align} 
Finally, for $\alpha_1,\dots,\alpha_r\in\OO\otimes\largewedge V^*$ we denote
\begin{align*}
\calL(\alpha_1,\dots,\alpha_r)&=\big(\cdots(\alpha_1\wedge\alpha_2)\wedge\cdots\big)\wedge\alpha_r,\\
 \calR(\alpha_1,\dots,\alpha_r)&=\alpha_1\wedge\big(\cdots\wedge(\alpha_{r-1}\wedge\alpha_r)\cdots\big),
\end{align*}
and for $s\in\NN_0$ we abbreviate
\begin{align*}
\calL(\alpha_1,\dots,\alpha_{j-1},\alpha_j[s],\alpha_{j+1},\dots,\alpha_r)=\calL(\alpha_1,\dots,\alpha_{j-1},\overbrace{\alpha_j,\dots,\alpha_j}^{s\text{-times}},\alpha_{j+1},\dots,\alpha_r),
\end{align*}
and similarly for $\calR$.

For the rest of this section, we will consider the vector space $V=\Imag\OO\oplus\OO\oplus\Imag\OO\oplus\OO$ in which case one has the obvious tetra-grading
\begin{align*}
\largewedge V^*=\bigoplus_{k,l,m,n}\largewedge^{k,l,m,n}V^*.
\end{align*}
Let $\theta_0,\theta_1,\varphi_0,\varphi_1\in\OO\otimes\largewedge^1V^*$ be the octonionic coordinate  functions and denote by
\begin{align}
\label{eq:O-coordinates}
\theta_0=\sum_{a=1}^7e_a\theta_0^a,\quad\theta_1=\sum_{a=0}^7e_a\theta_1^a,\quad\varphi_0=\sum_{a=1}^7e_a\varphi_0^a,\quad\text{and}\quad\varphi_1=\sum_{a=0}^7e_a\varphi_1^a
\end{align}
their decompositions in the standard basis of $\OO$.

\begin{theorem}
\label{thm:Spin7invforms}
The algebra $\big(\largewedge V^*\big)^{\Spin(7)}$ is generated by the following 96 elements:
\begin{enuma}
\item 1 element
\begin{align*}
[1,0,1,0]=-\Real(\theta_0\wedge\varphi_0);
\end{align*}
\item 8 elements
\begin{align*}
[k,0,7-k,0]=\frac{1}{k!(7-k)!}\sum_{\pi\in\calS_7}\sgn\pi\,\theta_0^{\pi(1)}\wedge\cdots\wedge\theta_0^{\pi(k)}\wedge\varphi_0^{\pi(k+1)}\wedge\cdots\wedge\varphi_0^{\pi(7)},\quad k\in\{0,\dots, 7\};
\end{align*}
\item 36 elements
\begin{align*}
[k_1,1,k_2,1]&=\frac1{k_1!k_2!}\Real\calR\big(\b{\theta_1},\theta_0[k_1],\varphi_0[k_2],\varphi_1\big),\quad k_1+k_2\in\{0,\dots,7\};
\end{align*}
\item 36 elements
\begin{align*}
[k_1,2,k_2,0]&=\frac{1}{2k_1!k_2!}\Real\calR\big(\b{\theta_1},\theta_0[k_1],\varphi_0[k_2],\theta_1\big),\\
[k_1,0,k_2,2]&=\frac{1}{2k_1!k_2!}\Real\calR\big(\b{\varphi_1},\theta_0[k_1],\varphi_0[k_2],\varphi_1\big),\quad k_1+k_2\in\{1,2,5,6\};
\end{align*}
\item 15 elements
\begin{align*}
[k_1,4,k_2,0]&=\frac1{24}\Real\calL\big(\theta_0[k_1],\varphi_0[k_2],\theta_1,\b{\theta_1},\theta_1,\b{\theta_1}\big),\\
[k_1,3,k_2,1]&=\frac16\Real\calL\big(\theta_0[k_1],\varphi_0[k_2],\theta_1,\b{\theta_1},\theta_1,\b{\varphi_1}\big),\\
[k_1,2,k_2,2]&=\frac14\Real\calL\big(\theta_0[k_1],\varphi_0[k_2],\theta_1,\b{\theta_1},\varphi_1,\b{\varphi_1}\big),\\
[k_1,1,k_2,3]&=\frac16\Real\calL\big(\theta_0[k_1],\varphi_0[k_2],\theta_1,\b{\varphi_1},\varphi_1,\b{\varphi_1}\big),\\
[k_1,0,k_2,4]&=\frac1{24}\Real\calL\big(\theta_0[k_1],\varphi_0[k_2],\varphi_1,\b{\varphi_1},\varphi_1,\b{\varphi_1}\big),\quad k_1+k_2\in\{0,1\}.
\end{align*}
\end{enuma}
\end{theorem}

\begin{proof}
Throughout the proof,  a convention similar to the one in Section \ref{s:InvSpin7} will be in effect; namely, $u$ and $v$ will denote elements of the vector representation $\Imag\OO$, while $x$ and $y$ will denote elements of the spin representation $\OO$ of $\Spin(7)$. Consider an arbitrary $\phi\in\big(\largewedge^{k,l,m,n}V^*\big)^{\Spin(7)}$. Then
\begin{align*}
p&=p(u_1,\dots,u_k,v_1,\dots,v_m,x_1,\dots,x_l,y_1,\dots,y_n)\\
&=\phi\left(\begin{pmatrix}u_1\\0\\0\\0\end{pmatrix},\dots,\begin{pmatrix}u_k\\0\\0\\0\end{pmatrix},\begin{pmatrix}0\\x_1\\0\\0\end{pmatrix},\dots,\begin{pmatrix}0\\x_l\\0\\0\end{pmatrix},\begin{pmatrix}0\\0\\v_1\\0\end{pmatrix},\dots,\begin{pmatrix}0\\0\\v_m\\0\end{pmatrix},\begin{pmatrix}0\\0\\0\\y_1\end{pmatrix},\dots,\begin{pmatrix}0\\0\\0\\y_n\end{pmatrix}\right)
\end{align*}
is clearly an element of $ P_{k+m,l+n}$. According to Theorem \ref{thm:Plm}, $p$ must be a polynomial in the generating elements listed therein, hence a linear combination of the following monomials:
\begin{align}
\label{eq:monomial}
\prod_{j=1}^Jp_j(u_{\kappa(k_{j-1}+1)},\dots, u_{\kappa(k_j)},v_{\mu(m_{j-1}+1)},\dots, v_{\mu(m_j)},x_{\lambda(l_{j-1}+1)},\dots, x_{\lambda(l_j)},y_{\nu(n_{j-1}+1)},\dots, y_{\nu(n_j)}),
\end{align}
for some permutations $\kappa\in\calS_k$, $\lambda\in\calS_l$, $\mu\in\calS_m$, $\nu\in\calS_n$, some integers
\begin{gather*}
0=k_0\leq k_1\leq\cdots\leq k_J=k,\\
0=l_0\leq l_1\leq\cdots\leq l_J=l,\\
0=m_0\leq m_1\leq\cdots\leq m_J=m,\\
0=n_0\leq n_1\leq\cdots\leq n_J=n,
\end{gather*}
and some generators $p_j\in P_{\Delta k_j+\Delta m_j,\Delta l_j+\Delta n_j}$, where $\Delta k_j=k_j-k_{j-1}$, etc. 
 
For every vector space $W$ with basis $w_1,\ldots, w_N$ and dual basis $\alpha_1,\ldots \alpha_N$, the identity
$$\psi =\frac{1}{K!}  \sum_{i_1,\ldots,i_K=1}^N \psi(w_{i_1},\ldots, w_{i_K}) \alpha_{i_1} \wedge  \cdots \wedge \alpha_{i_K}$$
obviously holds for every $\psi\in \largewedge^K W^*$.  Applying this  to $\phi\in \big(\largewedge^{k,l,m,n}V^*\big)^{\Spin(7)}$, we obtain
$$\phi=\frac{1}{k!l!m!n!} \sum_\I p_{\I}  \,\theta_0^{i^{(u)}_1}\wedge\cdots\wedge\theta_0^{i^{(u)}_{k}}\wedge \theta_1^{i^{(x)}_1}\wedge\cdots\wedge\theta_1^{i^{(x)}_{l}}\wedge \varphi_0^{i^{(v)}_1}\wedge\cdots\wedge\varphi_0^{i^{(v)}_{m}}\wedge \varphi_1^{i^{(y)}_1}\wedge\cdots\wedge\varphi_1^{i^{(y)}_{n}},
$$
where the summation is taken over all  multi-indices
\begin{align*}
\I&=\left(i^{(u)}_1,\dots,i^{(u)}_k,i^{(x)}_1,\dots,i^{(x)}_l,i^{(v)}_1,\dots,i^{(v)}_m,i^{(y)}_1,\dots,i^{(y)}_n\right)\\
&\in\{1,\dots,7\}^k\times\{0,\dots,7\}^l\times\{1,\dots,7\}^m\times\{0,\dots,7\}^n,
\end{align*}
and 
$$ p_\I= p\left( e_{i^{(u)}_1},\dots,e_{i^{(u)}_k},e_{i^{(x)}_1},\dots,e_{i^{(x)}_l},e_{i^{(v)}_1},\dots,e_{i^{(v)}_m},e_{i^{(y)}_1},\dots,e_{i^{(y)}_n}\right).$$
In view of \eqref{eq:monomial}, $\phi$ is therefore generated by the following forms:
\begin{align}
\begin{split}
&\sum_{\I'}q\left(e_{i^{(u)}_1},\dots, e_{i^{(u)}_{k'}},e_{i^{(v)}_1},\dots, e_{i^{(v)}_{m'}},e_{i^{(x)}_1},\dots, e_{i^{(x)}_{l'}},e_{i^{(y)}_1},\dots, e_{i^{(y)}_{n'}}\right)
\\
&\qquad\times\theta_0^{i^{(u)}_1}\wedge\cdots\wedge\theta_0^{i^{(u)}_{k'}}\wedge \theta_1^{i^{(x)}_1}\wedge\cdots\wedge\theta_1^{i^{(x)}_{l'}}\wedge \varphi_0^{i^{(v)}_1}\wedge\cdots\wedge\varphi_0^{i^{(v)}_{m'}}\wedge \varphi_1^{i^{(y)}_1}\wedge\cdots\wedge\varphi_1^{i^{(y)}_{n'}},
\end{split}
\end{align}
where $k',l',m',n'$ are non-negative integers,
\begin{align*}
\I'&=\left(i^{(u)}_1,\dots,i^{(u)}_{k'},i^{(x)}_1,\dots,i^{(x)}_{l'},i^{(v)}_1,\dots,i^{(v)}_{m'},i^{(y)}_1,\dots,i^{(y)}_{n'}\right)\\
&\in\{1,\dots,7\}^{k'}\times\{0,\dots,7\}^{l'}\times\{1,\dots,7\}^{m'}\times\{0,\dots,7\}^{n'},
\end{align*}
and $q\in P_{k'+m',l'+n'}$ is a generating polynomial from Theorem \ref{thm:Plm}.

In the rest of the proof we will separately investigate all possible combinations of a generator $q$ and integers $k',l',m',n'$ that may occur.

\begin{enuma}
\item
Consider $q=\ip{u_1}{u_2}\in P_{2,0}$. Then $k'+m'=2$ and $l'=n'=0$. First, if $k'=2$ and $m'=0$,
\begin{align*}
\sum_{i_1,i_2=1}^7\ip{e_{i_1}}{e_{i_2}}\theta_0^{i_1}\wedge\theta_0^{i_2}=\sum_{i_1=1}^7\theta_0^{i_1}\wedge\theta_0^{i_1}=0,
\end{align*}
and similarly if $k'=0$ and $m'=2$. Second, for $k'=m'=1$ we have
\begin{align*}
\sum_{i_1,i_2=1}^7\ip{e_{i_1}}{e_{i_2}}\theta_0^{i_1}\wedge\varphi_0^{i_2}&=\sum_{i_1,i_2=1}^7\Real(\b{e_{i_1}}e_{i_2})\theta_0^{i_1}\wedge\varphi_0^{i_2}\\
&=-\sum_{i_1,i_2=1}^7\Real(e_{i_1}e_{i_2})\theta_0^{i_1}\wedge\varphi_0^{i_2}\\
&=-\Real(\theta_0\wedge\varphi_0).
\end{align*}

\item Let $q=\det(u_1,\dots,u_7)\in P_{7,0}$. Then $l'=n'=0$, $m'=7-k'$, and for $0\leq k'\leq 7$ one has
\begin{align*}
&\sum_{i_1,\dots,i_7=1}^7\det(e_{i_1},\dots,e_{i_7})\,\theta_0^{i_1}\wedge\cdots\wedge\theta_0^{i_{k'}}\wedge\varphi_0^{i_{k'+1}}\wedge\cdots\wedge\varphi_0^{i_7}\\
&\qquad=\sum_{\pi\in\calS_7}\det(e_{\pi(1)},\dots,e_{\pi(7)})\,\theta_0^{\pi(1)}\wedge\cdots\wedge\theta_0^{\pi(k')}\wedge\varphi_0^{\pi(k'+1)}\wedge\cdots\wedge\varphi_0^{\pi(7)},\\
&\qquad=\sum_{\pi\in\calS_7}\sgn\pi\,\theta_0^{\pi(1)}\wedge\cdots\wedge\theta_0^{\pi(k')}\wedge\varphi_0^{\pi(k'+1)}\wedge\cdots\wedge\varphi_0^{\pi(7)}.
\end{align*}

\item Consider $q=\ip{x_1}{L_{u_1}\cdots L_{u_r}(x_2)}\in P_{r,2}$ for some $0\leq r\leq7$ and assume $n'=l'=1$. For $0\leq k'\leq r$ one has $m'=r-k'$ and
\begin{align*}
&\sum_{j_1,j_2=0}^7\,\sum_{i_1,\dots,i_r=1}^7 \ip{e_{j_1}}{L_{e_{i_1}}\cdots L_{e_{i_r}}(e_{j_2})}\theta_0^{i_1}\wedge\cdots\wedge\theta_0^{i_{k'}}\wedge\theta_1^{j_1}\wedge\varphi_0^{i_{k'+1}}\wedge\cdots\wedge\varphi_0^{i_r}\wedge\varphi_1^{j_2}\\
&\quad= (-1)^{k'}\sum_{j_1,j_2=0}^7\,\sum_{i_1,\dots,i_r=1}^7\Real\big[L_{\b{e_{j_1}}}L_{e_{i_1}}\cdots L_{e_{i_r}}(e_{j_2})\big]\theta_1^{j_1}\wedge\theta_0^{i_1}\wedge\cdots\wedge\theta_0^{i_{k'}}\wedge\varphi_0^{i_{k'+1}}\wedge\cdots\wedge\varphi_0^{i_r}\wedge\varphi_1^{j_2}\\
&\quad= (-1)^{k'}\Real\calR\big(\b{\theta_1},\theta_0[k'],\varphi_0[r-k'],\varphi_1\big).
\end{align*}

\item
Let $q$ be as in the previous item and assume $l'=2$ and $n'=0$.  Exactly in the same way as before we obtain
\begin{align*}
(-1)^{k'}\Real\calR\big(\b{\theta_1},\theta_0[k'],\varphi_0[r-k'],\theta_1\big),\quad 0\leq k'\leq r\leq 7.
\end{align*}
Let us explain why in this case we may assume $r\in\{1,2,5,6\}$ in fact. First, if $r=0$, we have
\begin{align*}
\Real\calR\big(\b{\theta_1},\theta_0[0],\varphi_0[0],\theta_1\big)=\Real\big(\b{\theta_1}\wedge\theta_1\big)=\frac12\big(\b{\theta_1}\wedge\theta_1-\b{\theta_1}\wedge\theta_1\big)=0
\end{align*}
 by \eqref{eq:conjAB}. Second, using
\begin{align*}
L_uL_v+L_vL_u=-2\ip uv,\quad u,v\in\Imag\OO,
\end{align*}
for $j_1,j_2\in\{0,\dots,7\}$ and $i_1,\dots,i_r\in\{1,\dots,7\}$ one has
\begin{align*}
\Real\big[L_{\b{e_{j_2}}}L_{e_{i_1}}\cdots L_{e_{i_r}}(e_{j_1})\big]&=\ip{e_{j_2}}{L_{e_{i_1}}\cdots L_{e_{i_r}}(e_{j_1})}\\
&=(-1)^r\ip{L_{e_{i_r}}\cdots L_{e_{i_1}}(e_{j_2})}{e_{j_1}}\\
&=(-1)^{\frac r2(r+1)}\ip{e_{j_1}}{L_{e_{i_1}}\cdots L_{e_{i_r}}(e_{j_2})}+\tilde q(e_{j_1},e_{j_2},e_{i_1},\dots,e_{i_r})\\
&=(-1)^{\frac r2(r+1)}\Real\big[L_{\b{e_{j_1}}}L_{e_{i_1}}\cdots L_{e_{i_r}}(e_{j_2})\big]+\tilde q(e_{j_1},e_{j_2},e_{i_1},\dots,e_{i_r})
\end{align*}
where $\tilde q$ is a linear combination of polynomials $\tilde q_1\tilde q_2$ with $\tilde q_1\in P_{2,0}$ and $\tilde q_2\in P_{r-2,2}$. Hence,
\begin{align*}
&2\Real\calR\big(\b{\theta_1},\theta_0[k'],\varphi_0[r-k'],\theta_1\big)\\
&\quad=\sum\left\{\Real\big[L_{\b{e_{j_1}}}L_{e_{i_1}}\cdots L_{e_{i_r}}(e_{j_2})\big]-(-1)^{\frac r2(r+1)}\Real\big[L_{\b{e_{j_1}}}L_{e_{i_1}}\cdots L_{e_{i_r}}(e_{j_2})\big]-\tilde q\right\}\\
&\qquad\qquad\times \theta_1^{j_1}\wedge\theta_0^{i_1}\wedge\cdots\wedge\theta_0^{i_{k'}}\wedge\varphi_0^{i_{k'+1}}\wedge\cdots\wedge\varphi_0^{i_r}\wedge\theta_1^{j_2}
\end{align*}
is a multiple of $[1,0,1,0]\wedge[k'-1,2,r-1-k',0]$ if $r\in\{3,4,7\}$. The case $l'=0$ and $n'=2$ is completely analogous.

\item
Finally, observe from the proof of Theorem \ref{thm:Plm}, see in particular \eqref{eq:Phisim}, that in the statement of the theorem, the polynomials
\begin{align*}
\ip{1}{((x_{k_1}\b{x_{k_2}})x_{k_3})\b{x_{k_4}}}\quad\text{and}\quad\ip{1}{(((u_j x_{k_1})\b{x_{k_2}})x_{k_3})\b{x_{k_4}}}
\end{align*}
may equivalently replace $\Phi(x_{k_1},x_{k_2},x_{k_3},x_{k_4})$ and $\Phi(u_jx_{k_1},x_{k_2},x_{k_3},x_{k_4})$, respectively. Let, first, $q=\ip{1}{((x_1\b{x_2})x_3)\b{x_4}}\in P_{0,4}$, $k'=m'=0$, and $n'=4-l'$. If $l'=4$, we obtain
\begin{align*}
\sum_{j_1,\dots,j_4=0}^7\ip{1}{((e_{j_1}\b{e_{j_2}})e_{j_3})\b{e_{j_4}}}\theta_1^{j_1}\wedge\theta_1^{j_2}\wedge\theta_1^{j_3}\wedge\theta_1^{j_4}=\Real\calL\big(\theta_1,\b{\theta_1},\theta_1,\b{\theta_1}\big).
\end{align*}
Similarly, it is straightforward that one arrives at
\begin{align*}
\Real\calL\big(\theta_1,\b{\theta_1},\theta_1,\b{\varphi_1}\big),\,\Real\calL\big(\theta_1,\b{\theta_1},\varphi_1,\b{\varphi_1}\big),\,\Real\calL\big(\theta_1,\b{\varphi_1},\varphi_1,\b{\varphi_1}\big),\text{ and }\Real\calL\big(\varphi_1,\b{\varphi_1},\varphi_1,\b{\varphi_1}\big),
\end{align*}
depending on whether $l'$ equals 3, 2, 1, or 0, respectively. Completely analogous is the second case $q=\ip{1}{(((u x_1)\b{x_2})x_3)\b{x_4}}\in P_{1,4}$; for example, if $k'=0$, $l'=2$, $m'=1$, and $n'=2$, one has
\begin{align*}
\sum_{j_1,\dots,j_4=0}^7\,\sum_{i=1}^7\ip{1}{(((e_i e_{j_1})\b{e_{j_2}})e_{j_3})\b{e_{j_4}}}\theta_1^{j_1}\wedge\theta_1^{j_2}\wedge\varphi_0^i\wedge\varphi_1^{j_3}\wedge\varphi_1^{j_4}=\Real\calL\big(\varphi_0,\theta_1,\b{\theta_1},\varphi_1,\b{\varphi_1}\big).
\end{align*}
\end{enuma}
Since all the constructed forms are clearly $\Spin(7)$-invariant and since we have at the same time exhausted all the possibilities provided by Theorem \ref{thm:Plm}, the proof is finished.
\end{proof}

\begin{remark}
\label{re:ZZ}
Note that each generator satisfies $[k,l,m,n]\in\largewedge^{k,l,m,n}V^*$ and that its coefficients in the standard basis generated by $\theta_0^1,\dots,\theta_0^7,\theta_1^0,\dots,\theta_1^7,\varphi_0^1,\dots,\varphi_0^7,\varphi_1^0,\dots,\varphi_1^7$ are integers.
\end{remark}

\section{$\b{\Spin(9)}$-invariant differential forms}

\label{s:Spin9invforms}

As a matter of fact, we have already constructed the algebra of $\b{\Spin(9)}$-invariant forms on the sphere bundle $S\OO^2$. It remains that the results of the previous section are given the right meaning which is the goal of this section. Moreover, the forms will be described in such a way that will later be particularly convenient for differentiation.

Let us first recall a standard construction of invariant differential forms on a homogeneous space (for reference see \cite[\S3]{Wang:OnInvariant} and \cite[\S21]{Lee:Smooth}). Let $G$ be a Lie group with a transitive left action on a smooth manifold $M$. Fix a point $p\in M$ and denote $H=\Stab_pG$. Then $M\cong G/H$ and one has the projection $\pi\maps G M\mape g{g(p)}$. We will need the following folklore fact. Note that the left and right translation are with respect to the product on $G$.

\begin{lemma}
\label{lem:betatilda}
Let $\beta\in\Omega^*(M)^G$. Then the form $\wt\beta=\pi^*\beta\in\Omega^*(G)$ satisfies
\begin{align}
\nonumber L_g^*\wt\beta&=\wt\beta,\quad g\in G;\\
\nonumber R_h ^*\wt\beta&=\wt\beta,\quad h\in H;\\
\label{eq:horizontalBeta}X\lrcorner\wt\beta&=0,\quad X\in\ker d\pi.
\end{align}
Conversely, if $\wt\beta\in\Omega^*(G)$ satisfies these properties, there is a unique $\beta\in\Omega^*(M)^G$ with $\wt\beta=\pi^*\beta$.
\end{lemma}

In our case, $G=\b{\Spin(9)}=\Spin(9)\ltimes\OO^2$ acts on $M=S\OO^2$ as $(g,x)\cdot (y,v)=(gy+x,gv)$. Choosing $p=(0,E_0)\in S\OO^2$ with $E_0=\big(\begin{smallmatrix}1\\0\end{smallmatrix}\big)$, we have $H\cong\Spin(7)$ according to \eqref{eq:stab}, and $\pi\maps{G}{S\OO^2}$ given by $\pi(g,x)=(x,gE_0)$. Let further $\fgg=\spin(9)\ltimes\OO^2$ be the Lie algebra of $G$ and let $\omega=(\varphi,\theta)$ be the left $G$-invariant, $\fgg$-valued Maurer--Cartan 1-form on $G$. Keeping the index convention \eqref{eq:indexconvention}, the Maurer--Cartan equation \eqref{eq:MaurerCartan} translates into 
\begin{align}
\label{eq:MCE}
d\theta_k^a=-\sum_{j=0}^1\sum_{c=0}^7\varphi_{k,j}^{a,c}\wedge\theta_j^c\quad\text{and}\quad d\varphi_{k,l}^{a,b}=-\sum_{j=0}^1\sum_{c=0}^7\varphi_{k,j}^{a,c}\wedge\varphi_{j,l}^{c,b}.
\end{align} 

Let us define five octonion-valued forms, i.e., elements of $\OO\otimes\Omega(G)$ as follows:
\begin{align}
\label{eq:O-1-forms}
\alpha=\theta_0^0,\quad \theta_0=\sum_{a=1}^7e_a\theta_0^a,\quad\theta_1=\sum_{a=0}^7e_a\theta_1^a,\quad\varphi_0=\sum_{a=1}^7e_a\varphi_0^a,\quad\text{and}\quad\varphi_1=\sum_{a=0}^7e_a\varphi_1^a
\end{align}
where we denote $\varphi_k^a=\varphi_{k,0}^{a,0}$ (observe that $\varphi_{0,0}^{0,0}=0$). Note that the (non-associative, real) algebra of $\OO$-valued differential forms on a smooth manifold is defined in precisely the same manner as in the alternating case considered in the previous section. In particular, \eqref{eq:conjAB} and  \eqref{eq:RealAB} hold verbatim in this more general setting as well.

\begin{lemma}
\label{pro:H-action}
For any $\hat h=\left(\begin{pmatrix}\chi_{h^{-1}}&0\\0&h^{-1}\end{pmatrix},0\right)\in\Stab_p\b{\Spin(9)}$ one has
\begin{align*}
R^*_{\hat h}\alpha=\alpha, \quad  R^*_{\hat h}\theta_0=\chi_h\theta_0,\quad R^*_{\hat h}\theta_1=h\theta_1, \quad  R^*_{\hat h}\varphi_0=\chi_h\varphi_0,\quad R^*_{\hat h}\varphi_1=h\varphi_1.
\end{align*}

\end{lemma}
\begin{proof}
Observe that $\theta=\big(\begin{smallmatrix}\alpha+\theta_0\\\theta_1\end{smallmatrix}\big)$, $\varphi E_0=\big(\begin{smallmatrix}\varphi_0\\\varphi_1\end{smallmatrix}\big)$, and  use \eqref{eq:rightTranslations}.
\end{proof}

Strictly speaking, \eqref{eq:O-1-forms} seems in conflict with the previously established notation. However, $\alpha$ is nothing else than the pull-back under $\pi$ of the contact form on $S\OO^2$ given by \eqref{eq:ContactForm}. As for the other forms, consider $[k,l,m,n]\in\Omega\big(\b{\Spin(9)}\big)$ given formally as in Section \ref{s:Spin7invforms}, but with \eqref{eq:O-coordinates} replaced by \eqref{eq:O-1-forms}.

\begin{proposition}
	The forms $[k,l,m,n]\in\Omega\big(\b{\Spin(9)}\big)$ satisfy the assumptions of Lemma~\ref{lem:T8}, hence correspond to unique $\b{\Spin(9)}$-invariant forms on $S\OO^2$ which will be denoted by the same symbol.
\end{proposition} 
\begin{proof} Let us verify that the assumptions of Lemma~\ref{lem:T8} are satisfied. First, since $[k,l,m,n]$ is built from the components of the Maurer--Cartan form, it is left invariant. Second, the obvious map $\largewedge V^* \to \Omega^1( \b{\Spin(9)})$ interwines the $\Spin(7)$-actions. As a consequence, $R_h^*[k,l,m,n]=[k,l,m,n]$ for all $h\in \Spin(7)$. 
	Finally, observe that $d\pi=\big(dx,(dg)E_0\big)$. Therefore, the forms $\theta_k^a,\varphi_{k,0}^{a,0}\in\Omega^1( \b{\Spin(9)})$ are horizontal in the sense of \eqref{eq:horizontalBeta}. 
\end{proof}

\begin{theorem}
\label{thm:Spin9invforms}
The algebra $\Omega^*(S\OO^2)^{\b{\Spin(9)}}$ is generated by the following 97 elements:
\begin{align*}
\begin{alignedat}{2}
&\alpha;\\
&[1,0,1,0];\\
&[k,0,7-k,0],&\quad &k\in\{0,\dots,7\};\\
&[k_1,1,k_2,1],&\quad &k_1+k_2\in\{0,\dots,7\};\\
&[k_1,l,k_2,2-l],&\quad & k_1+k_2\in\{1,2,5,6\},l\in\{0,2\};\\
&[k_1,l,k_2,4-l],&\quad & k_1+k_2\in\{0,1\},l\in\{0,\dots,4\}.
\end{alignedat}
\end{align*}
\end{theorem}

\begin{proof}

Observe that the 96 generators $[k,l,m,n]$ from Theorem \ref{thm:Spin7invforms} may also be regarded as elements of $\largewedge(T_pS\OO^2)^*$,  $p=(0,E_0)$. Together with $\alpha=\ip{E_0}{d\rho}$ which is the projection to the first factor of $T_pS\OO^2=\RR\oplus\Imag\OO\oplus\OO\oplus\Imag\OO\oplus\OO$, these forms then generate the algebra $\big[\largewedge(T_pS\OO^2)^*\big]^{\Spin(7)}$. Because $\b{\Spin(9)}$ acts transitively on $S\OO^2$, one has the algebra isomorphism
\begin{align}
\label{eq:evaluation}
\Omega^*(S\OO^2)^{\b{\Spin(9)}}\cong\big[\largewedge(T_pS\OO^2)^*\big]^{\Spin(7)}
\end{align}
given by evaluating a form at the point $p$. We claim that this isomorphism maps $\alpha\mapsto\alpha$ and $[k,l,m,n]\mapsto [k,l,m,n]$  which would complete the proof. To prove the claim, let $X\in T_p S\OO^2$ and chose a curve $(g(t),x(t))\in \b{\Spin(9)}$ through $(id, 0)$ at $t=0$ with $X= d\pi (g'(0),x'(0))$. First, from the definition of the Maurer--Cartan form it is obvious that   $\varphi^{a,0}_{k,0}(g'(0),x'(0))= \varphi^a_k(X)$ and 
 $\theta^{a}_{k}(g'(0),x'(0))= \theta^{a}_{k}(X)$. Second, given $ \beta \in \Omega^r(S\OO^2)^{\b{\Spin(9)}}$ corresponding to an invariant form $\wt\beta$ on  $\b{\Spin(9)}$, tangent vectors $X_1,\ldots, X_r\in  T_p S\OO^2$ and corresponding curves $(g_i(t),x_i(t))$ as before, we have 
 $$\beta (X_1,\ldots, X_r)= \wt\beta\big((g_1'(0),x_1'(0)), \ldots, (g_r'(0),x_r'(0)\big).$$
Combining these two facts,  the claim follows.
\end{proof}

\section{Exterior derivative}
\label{s:d}

The goal of this section is to provide a simple recipe for differentiating the $\b{\Spin(9)}$-invariant differential forms described above. Recall that we only know the forms explicitly in a point. However, as explained in the previous section, the invariant forms may also be  pulled back to the group $\b{\Spin(9)}$ where the Maurer--Cartan equations \eqref{eq:MCE} are available.

It is a general fact that the algebra of  left invariant differential forms $\Omega^*(G)^G$ on a Lie group $G$ is generated by the components  of the Mauer--Cartan form defined with respect to some  basis of $\mathfrak g$. In the case $G=\b{\Spin(9)}$, according to Corollary~\ref{cor:spin9}, the algebra of invariant forms is generated by the collection of the following $52$ (linearly independent) $1$-forms:
	\begin{align*}
	\begin{alignedat}{2}
	& \alpha = \theta_0^0, & & \\
	&\theta^a_0,&\quad &1\leq a \leq 7, \\
    &\theta_1^a, &\quad & 0\leq a\leq 7,\\
    &\varphi_0^a=\varphi^{a,0}_{0,0},&\quad  & 1\leq a\leq 7,\\
    &\varphi_1^a=\varphi^{a,0}_{1,0},&\quad  & 0\leq a\leq 7,\\
    &\varphi^{a,b}_{1,1}, &\quad & 1\leq b< a\leq 7.\\
	\end{alignedat}
	\end{align*}
Consequently, there exists a unique endomorphism $h$ of the algebra 	$\Omega^*(\b{\Spin(9)})^{\b{\Spin(9)}}$ such that
	\begin{align*}
	\begin{alignedat}{2}
		&h(\theta^a_0)= \theta^a_0,&\quad &0\leq a \leq 7, \\
		&h(\theta_1^a)=\theta_1^a, &\quad & 0\leq a\leq 7,\\
		&h(\varphi_0^a)=\varphi_0^a,&\quad  & 1\leq a\leq 7,\\
		&h(\varphi_1^a)=\varphi_1^a,&\quad  & 0\leq a\leq 7,\\
		&h(\varphi^{a,b}_{1,1})=0, &\quad & 1\leq b< a\leq 7.\\
	\end{alignedat}
\end{align*}
Further, since the exterior derivative preserves the subspace of invariant forms, we may define the operator $d_h:= h \circ d$ on 	$\Omega^*(\b{\Spin(9)})^{\b{\Spin(9)}}$.

\begin{theorem}
\label{thm:d}
 The anti-derivation $d_h$ satisfies 
	$d\circ \pi^* = d_h \circ \pi^*$ on $\Omega^*(S\OO^2)^{\b{\Spin(9)}}$ and is determined by the equations
	\begin{align*}
		d_h  (\alpha+\theta_0)&=-\varphi_0\wedge(\alpha+\theta_0)-\theta_1\wedge\b{\varphi_1},\\
		d_h \theta_1&=(\alpha+\theta_0)\wedge\varphi_1-\varphi_0\wedge\theta_1+\theta_1\wedge\varphi_0,\\
		d_h \varphi_0&=-\varphi_0\wedge\varphi_0-\varphi_1\wedge\b{\varphi_1},\\
		d_h \varphi_1&=\varphi_1\wedge\varphi_0.
	\end{align*}
\end{theorem}
\begin{proof}
Let $\beta\in\Omega^*(S\OO^2)^{\b{\Spin(9)}}$. Observe that 
\begin{equation}\label{eq:dhMod} d (\pi^*\beta) - d_h( \pi^*\beta)\equiv 0\end{equation}
modulo the ideal $\ker h$, generated by the $1$-forms $\varphi_{1,1}^{a,b}$. 
Consider the subbundles
$$ H=\{\varphi_{1,1}^{a,b}=0\} \quad \text{and} \quad  V= \{\varphi_{k,0}^{a,0}=0\ \text{and}\ \theta_k^a=0 \} $$
of the tangent bundle of $\b{\Spin(9)}$ and note that the latter splits as $H\oplus V$. On the one hand, $d\circ\pi^*=\pi^* \circ d$ and the  definition of $d_h$ imply  $X\lrcorner (d(\pi^*\beta) - d_h(\pi^*\beta))=0$ for $X\in V$. On the other hand, \eqref{eq:dhMod} implies $X\lrcorner (d (\pi^*\beta) - d_h( \pi^*\beta))=0$ for $X\in H$. Therefore, $d (\pi^*\beta) - d_h( \pi^*\beta)=0$, as claimed.

It remains to evaluate $d_h$ on the octonionic $1$-forms. It follows from Corollary \ref{cor:spin9} that
\begin{alignat*}{2}
	\varphi_{1,0}^{a,b}&=\varphi_{1,0}^{\b ba,0},&\quad &0\leq a\leq 7,\quad 1\leq b\leq 7,\\
	\varphi_{0,1}^{a,b}&=-\varphi_{1,0}^{\b ab,0},&\quad &0\leq a,b\leq 7,\\
	\varphi_{k,k}^{a,a}&=0, &\quad&0\leq a\leq 7,\quad 0\leq k\leq 1,\\
	\varphi_{1,1}^{a,b}&\equiv0,&\quad &1\leq a<b\leq 7,\\
	\varphi_{1,1}^{a,0}&\equiv2\varphi_{0,0}^{a,0},&\quad& 1\leq a\leq7,\\
	\varphi_{1,1}^{0,a}&\equiv-2\varphi_{0,0}^{a,0},&\quad& 1\leq a\leq7,\\
	\varphi_{0,0}^{a,b}&\equiv-\varphi_{0,0}^{\b ba,0},&\quad&1\leq b<a\leq 7,\\
	\varphi_{0,0}^{a,b}&\equiv\varphi_{0,0}^{\b ab,0},&\quad&1\leq a<b\leq 7,\\
	\varphi_{0,0}^{0,a}&=-\varphi_{0,0}^{a,0},&\quad&1\leq a \leq 7,
\end{alignat*}
modulo  $\ker h$. Observe that, since for any $a,b>0$ with $a\neq b$ one has $\b ba=-a\b b$, $\b a b=a\b b$ and $\b a=-a$, the last three relations can be merged into
\begin{align*}
	\varphi_{0,0}^{a,b}\equiv \varphi_{0,0}^{a\b b,0},\quad 0\leq a,b\leq 7.
\end{align*}
 Consequently, all summations running from 0 to 7, we compute
\begin{align*}
	d_h (\alpha+\theta_0)&=\sum_a e_ahd\theta_0^a\\
	&=-\sum_{a,c}e_ah\varphi_{0,0}^{a,c}\wedge h\theta_0^c-\sum_{a,c}e_ah\varphi_{0,1}^{a,c}\wedge h\theta_1^c\\
	&=-\sum_{a,c}e_a\varphi_{0,0}^{a\b c,0}\wedge\theta_0^c+\sum_{a,c}e_a\varphi_{1,0}^{\b ac,0}\wedge \theta_1^c\\
	&=-\sum_{b,c}e_be_c\varphi_{0,0}^{b,0}\wedge\theta_0^c+\sum_{b,c}e_c\b{e_b}\varphi_{1,0}^{b,0}\wedge\theta_1^c\\
	&=-\varphi_0\wedge(\alpha+\theta_0)-\sum_{b,c}e_c\b{e_b}\theta_1^c\wedge\varphi_{1,0}^{b,0}\\
	&=-\varphi_0\wedge(\alpha+\theta_0)-\theta_1\wedge\b{\varphi_1}.
\end{align*}
 Concerning the substitution in the fourth step, observe that $\{\norm{a\b c}\mid 0\leq a\leq 7\}=\{0,\dots,7\}$ holds for any $c$. Thus, $ \sum_{a}e_a\varphi_{0,0}^{a\b c,0}=\sum_{a} e_ae_{\b c}e_c\varphi_{0,0}^{a\b c,0}= \sum_b e_b e_c \varphi_{0,0}^{b,0}$ and similarly for the second sum. The proof of the remaining three relations is analogous, using  \eqref{eq:ipbar}, \eqref{eq:conjAB}, and $\b{\varphi_0}=-\varphi_0$.

\end{proof}

\begin{example}

Using the previous theorem, an easy computation shows
\begin{align}
\label{eq:dalpha}
d\alpha=-[1,0,1,0]-[0,1,0,1].
\end{align}
\end{example}

\section{The algebra of $\Spin(9)$-invariant valuations}

\label{s:main}

Since all necessary preparations have now been made, we can finally proceed to the computational proof of our main result. It is in this section that the structure of the algebra $\Val^{\Spin(9)}$ of $\Spin(9)$-invariant valuations on the octonionic plane is determined explicitly, following the general strategy outlined in Section \ref{ss:CA}. 

\begin{proof}[ Proof of Theorem \ref{thm:main}]
Recall that $\Val^{\Spin(9)}=\bigoplus_{k=0}^{16}\Val_k^{\Spin(9)}$, $\Val_{16}^{\Spin(9)}=\RR\vol_{16}$, and that any $\phi\in\Val_k^{\Spin(9)}$, $0\leq k \leq 15$, is represented by some $\omega\in\Omega^{k,15-k}(S\OO^2)^{\tr}$. Since $\Spin(9)$ is compact, $\omega$ can be chosen to be $\Spin(9)$-invariant in which case also $D\omega$ and the unique form $\alpha\wedge\xi\in\Omega^{15}(S\OO^2)^{\tr}$ satisfying $D\omega=d(\omega+\alpha\wedge\xi)$ are such. Hence, according to Theorem \ref{thm:Spin9invforms}, all these forms are polynomials in the 97 generators listed therein. In particular, $\xi$ is the unique element of degree 14 in the subalgebra $\Omega_h\subset\Omega^*(S\OO^2)^{\b{\Spin(9)}}$ generated by the 96 generators other than $\alpha$ with
\begin{align}
\label{eq:eqforD}
\alpha\wedge d\omega+\alpha\wedge d\alpha\wedge\xi=0.
\end{align}

Applying the isomorphism \eqref{eq:evaluation}, any $\omega\in\Omega^*(S\OO^2)^{\b{\Spin(9)}}$ can be expressed in the standard coordinates on $\largewedge(T_pS\OO^2)^*$ corresponding to $T_pS\OO^2=\RR\oplus\Imag\OO\oplus\OO\oplus\Imag\OO\oplus\OO$. Using Theorem \ref{thm:Spin7invforms}, we compute explicitly the coordinate expressions of the 97 generators of $\Omega^*(S\OO^2)^{\b{\Spin(9)}}$. Then we consider all monomials in the generators that have (as forms) degree 14, and pick a basis among these monomials, determining thus the space $\Omega_h^{14}=\Omega_h\cap\Omega^{14}(S\OO^2)$. Recall also that an invariant form is differentiated subject to the structure equations of Theorem \ref{thm:d} and observe that these equations make perfect sense if the coordinate isomorphism \eqref{eq:evaluation} is applied. All in all, we can compute the Rumin differential of any $\omega\in\Omega^{15}(S\OO^2)^{\b{\Spin(9)}}$ as follows:
\begin{enums1}
\item Compute $d\omega$.
\item Solve the linear problem \eqref{eq:eqforD} for $\xi\in\Omega_h^{14}$.
\item Compute $d\xi$.
\item Compute $D\omega=d\omega+d\alpha\wedge\xi-\alpha\wedge d\xi$.
\end{enums1}
With this in hand, we can compute via \eqref{eq:convolutionForms} and \eqref{eq:Dconv} the convolution product $\phi_1*\phi_2$ of any $\phi_i\in\Val^{\Spin(9)}_{k_i}$, $1\leq k_i\leq 15$, and determine linear independence of such valuations according to Theorem \ref{thm:kernel}.

We consider on $\RR[t,s,v,u_1,u_2,w_1,w_2,w_3,x_1,x_2,y,z]$ the graded lexicographical monomial order and apply Corollary \ref{cor:ideal} to the subalgebra $\calA$ of $(\Val^{\Spin(9)},*)$ generated by
\begin{align*}
\hat  t&=\frac1{14}\int_{N(\Cdot)}[0,4,0,0]\wedge[0,4,0,0]\wedge[7,0,0,0]\in\Val^{\Spin(9)}_{15},\\
\hat  s&=\frac1{14}\int_{N(\Cdot)}[0,4,0,0]\wedge[0,4,0,0]\wedge[6,0,1,0]\in\Val^{\Spin(9)}_{14},\\
\hat  v&=\frac1{14}\int_{N(\Cdot)}[0,4,0,0]\wedge[0,4,0,0]\wedge[5,0,2,0]\in\Val^{\Spin(9)}_{13},\\
\hat  u_1&=\frac1{14}\int_{N(\Cdot)}[0,4,0,0]\wedge[0,4,0,0]\wedge[4,0,3,0]\in\Val^{\Spin(9)}_{12},\\
\hat  u_2&=\int_{N(\Cdot)}[0,4,0,0]\wedge[0,1,2,1]\wedge[7,0,0,0]\in\Val^{\Spin(9)}_{12},\\
\hat  w_1&=\frac1{14}\int_{N(\Cdot)}[0,4,0,0]\wedge[0,4,0,0]\wedge[3,0,4,0]\in\Val^{\Spin(9)}_{11},\\
\hat  w_2&=\int_{N(\Cdot)}[0,4,0,0]\wedge[0,0,2,2]\wedge[7,0,0,0]\in\Val^{\Spin(9)}_{11},\\
\hat  w_3&=\int_{N(\Cdot)}[0,4,0,0]\wedge[0,0,0,4]\wedge[7,0,0,0]\in\Val^{\Spin(9)}_{11},\\
\hat  x_1&=\frac1{14}\int_{N(\Cdot)}[0,4,0,0]\wedge[0,4,0,0]\wedge[2,0,5,0]\in\Val^{\Spin(9)}_{10},\\
\hat  x_2&=\int_{N(\Cdot)}[0,4,0,0]\wedge[1,1,1,1]\wedge[4,0,3,0]\in\Val^{\Spin(9)}_{10},\\
\hat  y&=\frac1{14}\int_{N(\Cdot)}[0,4,0,0]\wedge[0,4,0,0]\wedge[1,0,6,0]\in\Val^{\Spin(9)}_9,\\
\hat  z&=\frac1{14}\int_{N(\Cdot)}[0,4,0,0]\wedge[0,4,0,0]\wedge[0,0,7,0]\in\Val^{\Spin(9)}_8.
\end{align*}
The algorithm of Corollary \ref{cor:ideal} yields a basis of $\calA$ and a generating set of the ideal $I$ such that $\calA\cong\RR[t,s,v,u_1,u_2,w_1,w_2,w_3,x_1,x_2,y,z]/I$. By Remark \ref{rem:IdealBasis}, the generating set of $I$ is finite.  We minimize it using Lemma \ref{lem:minimal} and list the remaining 78 generators in Appendix~\ref{app:relations}. Since none of the generators of the ideal contains a non-trivial linear term, the set of 12 homogeneous generators of the ring is minimal as well.  Finally, since the basis of $\calA$ has 143 elements, comparison with \eqref{eq:BettiSpin9} shows $\calA=\Val^{\Spin(9)}$ which finishes the proof.
\end{proof}

Let us list here explicitly the monomial basis of $\Val^{\Spin(9)}$ constructed in the previous proof, keeping the notation of the proof and suppressing the symbol `$*$' for the sake of brevity.

\begin{corollary}
\label{cor:basis}
One has
\begin{align*}
\Val_{16}^{\Spin(9)}&=\linspan\left\{\vol_{16}\right\},\\
\Val_{15}^{\Spin(9)}&=\linspan\left\{ \hat  t\right\},\\
\Val_{14}^{\Spin(9)}&=\linspan\left\{\hat   t^2, \hat  s\right\},\\
\Val_{13}^{\Spin(9)}&=\linspan\left\{\hat   t^3, \hat  t\hat   s, \hat  v\right\},\\
\Val_{12}^{\Spin(9)}&=\linspan\left\{\hat  t^4, \hat t^2 \hat s, \hat t\hat  v,\hat  s^2, \hat u_1,\hat  u_2\right\},\\
\Val_{11}^{\Spin(9)}&=\linspan\left\{ \hat t^5, \hat t^3 \hat s, \hat t^2 \hat v, \hat t \hat s^2,\hat  t\hat  u_1,\hat  t\hat  u_2, \hat s\hat  v,\hat  w_1,\hat  w_2,\hat  w_3\right\},\\
\Val_{10}^{\Spin(9)}&=\linspan\left\{ \hat t^6, \hat t^4 \hat s, \hat t^3 \hat v, \hat t^2 \hat s^2,\hat  t^2 \hat u_1,\hat  t^2 \hat u_2,\hat  t\hat  s\hat  v,\hat  t\hat  w_1,\hat  t \hat w_2,\hat  t \hat w_3, \hat s^3,\hat  s\hat  u_1, \hat s \hat u_2, \hat x_1, \hat x_2\right\},\\
\Val_9^{\Spin(9)}&=\linspan\left\{ \hat t^7, \hat t^5 \hat s,\hat  t^4 \hat v,\hat  t^3 \hat s^2, \hat t^3 \hat u_1,
\hat  t^3 \hat u_2,\hat  t^2 \hat s\hat  v,\hat  t^2\hat  w_1, \hat t^2 \hat w_2,\hat  t^2\hat  w_3,
\hat  t \hat s^3, \hat t\hat  s\hat  u_1,\hat  t \hat s \hat u_2, \hat t\hat  x_1,\hat  t\hat  x_2,\right.\\
&\qquad\qquad\quad \left. \hat s^2 \hat v,\hat  s \hat w_1, \hat s \hat w_2, \hat s \hat w_3, \hat y\right\},\\
\Val_8^{\Spin(9)}&=\linspan\left\{ \hat t^8, \hat t^6\hat  s,\hat  t^5 \hat v, \hat t^4 \hat s^2, \hat t^4 \hat u_1, 
\hat t^4 \hat u_2, \hat t^3 \hat s\hat  v,\hat  t^3\hat  w_1,\hat  t^3 \hat w_2,\hat  t^3 \hat w_3,
\hat  t^2\hat  s^3, \hat t^2 \hat s \hat u_1,\hat  t^2\hat  s \hat u_2, \hat t^2 \hat x_1, \hat t^2 \hat x_2,\right.\\
&\qquad\qquad\quad \left.
\hat  t \hat s^2 \hat v,\hat  t \hat s \hat w_1, \hat t\hat  s \hat w_2, \hat t\hat  s \hat w_3, \hat t \hat y,
\hat  s^4,\hat  s^2 \hat  u_1, \hat s^2 \hat  u_2, \hat  s \hat  x_1,\hat  v \hat  w_2,\hat  v \hat  w_3, \hat z\right\},\\
\Val_7^{\Spin(9)}&=\linspan\left\{ \hat t^9, \hat t^7 \hat s,\hat  t^6 \hat v, \hat t^5 \hat s^2, \hat t^5 \hat u_1,
 \hat t^5 \hat u_2, \hat t^4 \hat s \hat v, \hat t^4 \hat w_1,\hat  t^4 \hat w_2, \hat t^4 \hat w_3, \hat t^3 \hat s^3, \hat t^3 \hat s \hat u_1, \hat t^3 \hat s \hat u_2, \hat t^3 \hat x_1,\hat  t^3 \hat x_2,\right.\\
&\qquad\qquad\quad \left.\hat  t^2 \hat s^2 \hat v, \hat t^2 \hat s \hat w_1,\hat  t^2 \hat s \hat w_2,\hat  t^2 \hat s \hat w_3,\hat  t^2 y\right\},\\
\Val_6^{\Spin(9)}&=\linspan\left\{\hat  t^{10},\hat  t^8 \hat s,\hat  t^7\hat  v,\hat  t^6 \hat s^2,\hat  t^6 \hat u_1, \hat t^6 \hat u_2,\hat  t^5\hat  s \hat v, \hat t^5 \hat w_1, \hat t^5 \hat w_2,\hat  t^5 \hat w_3,\hat  t^4 \hat s^3, \hat t^4 \hat s\hat  u_1,\hat  t^4 \hat s \hat u_2,\hat  t^4 \hat x_1, \hat t^4 \hat x_2\right\},\\
\Val_5^{\Spin(9)}&=\linspan\left\{ \hat t^{11}, \hat t^9 \hat s,\hat  t^8 \hat v, \hat t^7 \hat s^2, \hat t^7 \hat u_1,\hat  t^7 \hat u_2,\hat  t^6 \hat s \hat v, \hat t^6 \hat w_1,\hat  t^6 \hat w_2, \hat t^6 \hat w_3\right\},\\
\Val_4^{\Spin(9)}&=\linspan\left\{ \hat t^{12},\hat  t^{10} \hat s, \hat t^9 \hat v, \hat t^8 \hat s^2, \hat t^8 \hat u_1, \hat t^8 \hat u_2\right\},\\
\Val_3^{\Spin(9)}&=\linspan\left\{ \hat t^{13}, \hat t^{11} \hat s, \hat t^{10} \hat v\right\},\\
\Val_2^{\Spin(9)}&=\linspan\left\{ \hat t^{14}, \hat t^{12} \hat s\right\},\\
\Val_1^{\Spin(9)}&=\linspan\left\{ \hat t^{15}\right\},\\
\Val_0^{\Spin(9)}&=\linspan\left\{ \hat t^{16}\right\}.
\end{align*}
\end{corollary}

\begin{remark}
\begin{enuma}
\item A straightforward computation in coordinates yields
\begin{align*}
\frac1{14}[0,4,0,0]\wedge[0,4,0,0]\wedge[7,0,0,0]
=-\theta_0^1\wedge\cdots\theta_0^7\wedge\theta_1^0\wedge\cdots\wedge\theta_1^7.
\end{align*}
Consequently, a comparison with the well-known expressions of the intrinsic volumes in terms of invariant differential forms \cite{Zaehle:Integral} (cf. also \cite{BernigHug:Tensor}) shows that
\begin{align}
\label{eq:tmu15}
\hat t=-2\mu_{15}.
\end{align}
\item Observe that the effect of the hard Lefschetz theorem \cite[Corollary 0.2]{BernigBroecker:Rumin} is immediately seen in the above basis.
\end{enuma}
\end{remark}

\begin{remark}
All the explicit computations with forms needed for the proof of Theorem \ref{thm:main} were accomplished using the computer algebra system {\sc Maple}. Let us briefly comment on certain aspects of this part of the proof.
\begin{enuma}
\item First, observe that we only work with rational vectors throughout the proof. Indeed, by Remark \ref{re:ZZ}, we start with vectors of integers and we perform on them rational operations, see Theorem \ref{thm:d} and \eqref{eq:eqforD}. In particular, the numbers $\lambda_{k,i}(x)$ defined in Lemma \ref{lem:ideal} are in $\QQ$. This  leaves no room for an error caused by imprecise manipulation with the forms, e.g., by rounding.
\item
Second, there are simplifications available for the computation of the Rumin differential as we will now explain. First, since $d\alpha\in\Omega^{1,1}(S\OO^2)$ and $d\omega\in\Omega^{k,16-k}(S\OO^2)$, we in fact have $\xi\in\Omega^{k-1,15-k}(S\OO^2)$. Second,  denote $W=T_pS\OO^2$. Corresponding to the decomposition $W=(\RR\oplus\Imag\OO)\oplus\OO\oplus\Imag\OO\oplus\OO$ one has
\begin{align*}
\largewedge^{k-1,15-k}W^*&=\bigoplus_{m}\left[\bigoplus_j\largewedge^{j,k-1-j,m+j,15-k-m-j}W^*\right]
\end{align*}
and similarly for $\largewedge^{k,16-k}W^*$. The multiplication by $d\alpha\in\largewedge^{1,0,1,0}W^*\oplus\largewedge^{0,1,0,1}W^*$ mapping the former space into the latter is then graded with respect to the outer gradings. Effectively, this reduces the linear problem \eqref{eq:eqforD} to inverting rational matrices of size at most 110 which is done instantly with {\sc Maple}.
\item 
A serious obstruction, on the contrary, turned out to arise from the size of the space when taking the wedge product of two general forms. Note that
\begin{align*}
\dim\largewedge(\RR\oplus\Imag\OO\oplus\OO\oplus\Imag\OO\oplus\OO)^*=2^{31}=2\,147\,483\,648
\end{align*}
and so the computer in general needs to perform a huge number of basic operations.  In fact, all our attempts to multiply forms using conventional {\sc Maple} procedures and exterior-algebra packages failed to terminate in a reasonable time. Instead, the following idea proved to be extremely useful: The standard basis of the above space is in a bijective correspondence with
\begin{align*}
\big\{(K,L,M,N)\in\ZZ^4\mid 0\leq K,L,M,N\leq 255\text{ and } M\text{ is even}\big\}
\end{align*}
by identifying
\begin{align*}
\theta_0^{i^{(u)}_1}\wedge\cdots\wedge\theta_0^{i^{(u)}_k}\wedge \theta_1^{i^{(x)}_1}\wedge\cdots\wedge\theta_1^{i^{(x)}_l}\wedge \varphi_0^{i^{(v)}_1}\wedge\cdots\wedge\varphi_0^{i^{(v)}_m}\wedge \varphi_1^{i^{(y)}_1}\wedge\cdots\wedge\varphi_1^{i^{(y)}_n},
\end{align*}
where $\theta_0^0=\alpha$, $0\leq i_1^{(u)}<\cdots < i_k^{(u)}\leq 7$, $0\leq i_1^{(x)}<\cdots < i_l^{(x)}\leq 7$, $1\leq i_1^{(v)}<\cdots < i_m^{(v)}\leq 7$, and $0\leq i_1^{(y)}<\cdots < i_n^{(y)}\leq 7$, 
with $(K,L,M,N)\in\ZZ^4$, where
\begin{align*}
K=\sum_{j=1}^k 2^{i_j^{(u)}},\quad L=\sum_{j=1}^l 2^{i_j^{(x)}},\quad M=\sum_{j=1}^m 2^{i_j^{(v)}},\quad \text{and}\quad N=\sum_{j=1}^n 2^{i_j^{(y)}}.
\end{align*}
In this representation, the wedge product of $(K,L,M,N)$ and $(K',L',M',N')$ is proportional to $(K+K',L+L',M+M',N+N')$. Moreover, the proportionality factor $\varepsilon\in\{-1,0,1\}$ (which is easy to determine) can be stored and then read within each elementary operation. In a similar fashion, $*_1$ of $(K,L,M,N)$ is either $(255-K,255-L,M,N)$ or its opposite.
\end{enuma}
\end{remark}

\section{The principal kinematic formula}

\label{s:pkf}

 Invoking a special case of the ftaig (Theorem \ref{thm:ftaig}), we will now use the description of the algebra $\Val^{\Spin(9)}$ obtained in the preceding section to determine explicitly the principal kinematic formula on the octonionic plane.

For $0\leq k\leq16$, let $\psi_k^{(1)},\dots,\psi_k^{(b_k)}$ be the basis of $\Val_{16-k}^{\Spin(9)}$ given by Corollary \ref{cor:basis}. Observe that for any $0\leq k\leq8$ and $1\leq i\leq b_k$ one has
\begin{equation}\label{eq:psi_k}
\psi_{8}^{(i)}=\hat t^{8-k}*\psi_{k}^{(i)}\quad\text{and}\quad\psi_{16-k}^{(i)}=\hat t^{8-k}*\psi_{8}^{(i)}.
\end{equation}
The matrix of the Poincar\'e pairing in the basis $\psi_0^{(1)},\psi_1^{(1)},\psi_2^{(1)},\psi_2^{(2)},\dots,\psi_{16}^{(1)}$ has the following block anti-diagonal form:
\begin{align*}
M=\begin{pmatrix}&&&&M_{0}\\&&&\reflectbox{$\ddots$}&\\&&M_k&&\\&\reflectbox{$\ddots$}&&&\\M_{16}&&&&\end{pmatrix}\in\RR^{143\times143},
\end{align*}
where $M_{k}=M_{16-k}\in \RR^{b_k\times b_k}$ is given by
\begin{align}
\label{eq:Mkij}
(M_k)_{i,j}\chi=\psi_{k}^{(i)}*\psi_{16-k}^{(j)}.
\end{align}
In fact,  \eqref{eq:psi_k} implies $(M_k)_{i,j}=(M_8)_{i,j}$  for $ 0\leq k\leq 8$ and $1\leq i,j\leq b_k$; in other words, the middle block contains all the others.

Using Theorem \ref{thm:main}, it is a straightforward, yet computationally tedious task to compute the matrix $M_8$ explicitly and the same holds for inverting its submatrices. Observe that, using \eqref{eq:nthpower} and \eqref{eq:tmu15},  \eqref{eq:Mkij} can be rewritten as
\begin{align*}
(M_k)_{i,j}\hat t^{16}=16!\, \omega_{16}\,\psi_{k}^{(i)}*\psi_{16-k}^{(j)}=\frac{16!}{8!}\pi^8\,\psi_{k}^{(i)}*\psi_{16-k}^{(j)}
\end{align*}
which is more convenient for computations since we  naturally obtain the products $\psi_{k}^{(i)}*\psi_{16-k}^{(j)}$ as (rational) multiples of $\hat t^{16}$. Following these considerations, we compute $M^{-1}$ and then use \eqref{eq:PKF} to write down the principal kinematic formula
\begin{align}
\label{eq:principal}
\int_{\b{\Spin(9)}}\chi(A\cap\b g B)=\sum_{k=0}^{16}\sum_{i,j=1}^{ b_k}(M_{k}^{-1})_{i,j}\psi^{(i)}_{k}(A)\,\psi^{(j)}_{16-k}(B),\quad A,B\in\K(\OO^2).
\end{align}
The explicit expression of the (largest) part of \eqref{eq:principal} corresponding to $k=8$ is, for illustration, displayed in Appendix \ref{app:pkf}.

\section{Kubota-type $\Spin(9)$-invariant valuations}

\label{s:Tk}

The goal of this section is to study the Kubota-type valuations
\begin{align}
\label{eq:KubotaTk}
T_k(A)=\int_{\OO P^1}\mu_k(\pi_O A) dO, \quad 0\leq k\leq 8,
\end{align}
introduced by Alesker \cite{Alesker:Spin9}. First, we will prove their non-triviality, in the sense that for $k\geq 2$ they are not proportional to the intrinsic volumes, i.e., not $\SO(16)$-invariant. Second, we will express them in the monomial basis of $\Val^{\Spin(9)}$  given by Corollary \ref{cor:basis}.

Recall that the $\Spin(9)$-invariant measure $dO$ on $\OO P^1$ is the pullback under the coordinate map $O_a\mapsto a$ of $c\big(1+\sqnorm a\big)^{-8}d a$, where $d a$ is the Lebesgue measure on $\OO=\RR^8$ and $c\in\RR$ is a normalizing constant  (cf. \cite{Grinberg:Spherical}). Using spherical coordinates, it is easy to see that $dO$ is a probability measure if $c=\frac{840}{\pi^4}$ which we will assume. $dO$ further induces a $\b{\Spin(9)}$-invariant measure on the affine octonionic projective line $\b{\OO P^1}=\{O+x\mid O\in\OO P^1,x\in\OO\}$ via
\begin{align*}
\int_{\b{\OO P^1}} f(\b O)d\b O=\int_{\OO P^1}\left(\int_{O^\perp} f(O+x)d x\right)dO.
\end{align*}

\begin{proposition}
\label{pro:LambdaTk}
Let $1\leq k\leq8$. One has
\begin{align}
\label{eq:LambdaTk}
\Lambda T_k&=[9-k]\,T_{k-1},\\
\label{eq:LambdakT8}
\Lambda^kT_8&=[k]!\,T_{8-k}.
\end{align}
\end{proposition}

\begin{proof}
Clearly, \eqref{eq:LambdakT8} is an immediate consequence of \eqref{eq:LambdaTk}. To prove \eqref{eq:LambdaTk}, note that for any $A\in\calK(\OO^2)$ and $\lambda \geq0$ we have
\begin{align*}
T_k(A+\lambda D^n)=\begin{bmatrix}8\\k\end{bmatrix}\int_{\OO P^1}
\int_{\Grass_k(O)}\mu_{k}( \pi_O A+ \lambda D^8_O)d EdO,
\end{align*}
where $D_O^8$ is the unit ball in $O$. Differentiating in $\lambda =0$ and using \eqref{eq:SteinerMu}  yields the claim.
\end{proof}

\begin{proposition}
\label{pro:mukTk}
Let $2\leq k\leq8$. The valuations $\mu_k,T_k\in\Val^{\Spin(9)}$ are linearly independent.
\end{proposition}

\begin{proof}
First of all, in view of \eqref{eq:SteinerMu}  and \eqref{eq:LambdaTk} it is enough to establish the case  $k=2$.  It thus suffices to show that $\Klain_{T_2}$ is not constant. To this end, consider the following two convex bodies in $\OO^2$:
\begin{align*}
A_1=\conv\left\{\begin{pmatrix}0\\0\end{pmatrix},\begin{pmatrix}e_0\\0\end{pmatrix},\begin{pmatrix}e_1\\0\end{pmatrix}\right\}\quad\text{and}\quad A_2=\conv\left\{\begin{pmatrix}0\\0\end{pmatrix},\begin{pmatrix}e_0\\0\end{pmatrix},\begin{pmatrix}0\\e_0\end{pmatrix}\right\}.
\end{align*}
Since $\mu_2(A_i)=\vol_2(A_i)=\frac12$ clearly holds for $i=1,2$, we have to show that $T_2(A_1)\neq T_2(A_2)$.

Fix $a\in\OO$ and denote $c_a=(1+|a|^2)^{-1}$ for the rest of the proof. Since $\sqrt{c_a}\big(\begin{smallmatrix}e_0\\e_0a\end{smallmatrix}\big),\dots,\sqrt{c_a}\big(\begin{smallmatrix}e_7\\e_7a\end{smallmatrix}\big)$ is an orthonormal basis of $O_a\in\OO P^1$, the orthogonal projection $\pi_a\maps{\OO^2}{O_a}$ is given by
\begin{align*}
\pi_a\begin{pmatrix}x\\y\end{pmatrix}=c_a\sum_{j=0}^7\big(\ip{x}{e_j}+\ip{y}{e_ja}\big)\begin{pmatrix}e_j\\e_ja\end{pmatrix}.
\end{align*}
In particular,
\begin{align}
\label{eq:pia1}
\pi_a\begin{pmatrix}e_0\\0\end{pmatrix}&=c_a\begin{pmatrix}e_0\\e_0a\end{pmatrix},\\
\label{eq:pia2}
\pi_a\begin{pmatrix}e_1\\0\end{pmatrix}&=c_a\begin{pmatrix}e_1\\e_1a\end{pmatrix},\\
\label{eq:pia3}
\pi_a\begin{pmatrix}0\\e_0\end{pmatrix}&=c_a\Real(a)\begin{pmatrix}e_0\\e_0a\end{pmatrix}-c_a\sum_{j=1}^7\ip{e_j}{a}\begin{pmatrix}e_j\\e_ja\end{pmatrix}.
\end{align}
Since the projections \eqref{eq:pia1} and \eqref{eq:pia2} are obviously perpendicular, one has
\begin{align*}
2\mu_2(\pi_aA_1)=\Norm{c_a\begin{pmatrix}e_0\\e_0a\end{pmatrix}}\cdot\Norm{c_a\begin{pmatrix}e_1\\e_1a\end{pmatrix}}=c_a.
\end{align*}
Consequently, using spherical coordinates in $\OO$, we obtain
\begin{align*}
2T_2(A_1)&=\frac{840}{\pi^4}\int_{\OO}c_a^9d a=\frac12.
\end{align*}
Similarly, since the two parts of the projection \eqref{eq:pia3} are parallel and perpendicular, respectively, to \eqref{eq:pia1}, one has
\begin{align*}
2\mu_2(\pi_aA_2)=\Norm{c_a\begin{pmatrix}e_0\\e_0a\end{pmatrix}}\cdot\Norm{c_a\sum_{j=1}^7\ip{e_j}{a}\begin{pmatrix}e_j\\e_ja\end{pmatrix}}=c_a\sqrt{\sum_{j=1}^7\ip{e_j}{a}^2}=c_a\norm{\Imag(a)}.
\end{align*}
 Hence, using spherical coordinates in $\Imag\OO$, we get
\begin{align*}
2T_2(A_2)=\frac{840}{\pi^4}\int_{\OO}c_a^9\norm{\Imag(a)}da=\frac{7}{16}.
\end{align*}
\end{proof}

\begin{remark}
Keeping the notation from the previous proof, consider $E_i=\linspan A_i$, $i=1,2$. Then $\Klain_{T_2}(E_1)=\frac12$ and $\Klain_{T_2}(E_1)=\frac{7}{16}$. Using the same considerations as in \cite[\S5]{BernigVoide:Spin}, namely that $\Klain_{\mu_2}=1$, $\Klain_{\tau_\OO}(E_1)=0$,  $\Klain_{\tau_\OO}(E_2)=1$, and $\dim\Val_2^{\Spin(9)}=2$, we conclude
\begin{align}
T_2=\frac12\mu_2-\frac1{16}\tau_\OO,
\end{align}
where $\tau_\OO\in\Val_2^{\Spin(9)}$ is the octonionic pseudo-volume also introduced by Alesker \cite{Alesker:Spin9}.
\end{remark}

\begin{lemma}
\label{lem:T8}
Consider the convex body $D=\left\{\big(\begin{smallmatrix}0\\x\end{smallmatrix}\big)\in\OO^2\mid x\in\OO,\norm x\leq 1\right\}\in\K(\OO^2)$. Then
\begin{align}
T_8= \frac1{\omega_8}\sum_{i,j=1}^{27}(M_8^{-1})_{i,j} \psi^{(j)}_8(D) \psi^{(i)}_{8}.
\end{align}
\end{lemma}

\begin{proof}
Fix $A\in\calK(\OO^2)$. First, observe that
\begin{align*}
T_8(A)=\int_{\OO P^1}\vol_8(\pi_{O^\perp}A) dO=\int_{\OO P^1}\left(\int_{O^\perp}\chi\big(A\cap(x+O)\big)d x\right)dO=\int_{\b{\OO P^1}}\chi(A\cap \b O)d\b O.
\end{align*}
From this, the Steiner formula \eqref{eq:Steiner}, and the fact that $\linspan D=O_\infty\in\OO P^1$ we infer
\begin{align*}
\int_{\b{\Spin(9)}}\chi(A\cap\b g\, \lambda D)=\lambda ^8\vol_8(D)\int_{\b{\OO P^1}}\chi(A\cap \b O)d\b O+O(\lambda ^{7}).
\end{align*}
On the other hand, applying the principal kinematic formula  \eqref{eq:principal} to $B=\lambda  D$, $\lambda >0$, we get
\begin{align*}
\int_{\b{\Spin(9)}}\chi(A\cap\b g\, \lambda D)&=\sum_{k=0}^{16}\sum_{i,j=1}^{b_k}(M_k^{-1})_{i,j} \psi^{(i)}_{k}(A)\, \psi^{(j)}_{16-k}(\lambda  D)\\
&=\lambda ^8\sum_{i,j=1}^{b_8}(M_8^{-1})_{i,j} \psi^{(i)}_{8}(A)\, \psi^{(j)}_8(D)+O(\lambda ^{7}).
\end{align*}
Since $\vol_8(D)=\omega_8$, dividing by $\lambda ^8$ and sending $\lambda \rightarrow\infty$ yield the result.
\end{proof}

\begin{remark}
\begin{enuma}
\item The valuation $A\mapsto\int_{\b{\OO P^1}}\chi(A\cap\b O)d\b O$ was denoted $U_8$ in \cite{Alesker:Spin9}.
\item  Observe that only the part of the the principal kinematic formula that is explicitly displayed in Appendix \ref{app:pkf} is relevant for the previous proof.
\end{enuma}
\end{remark}

\begin{theorem}
\label{thm:Tk}
Denote $ \hat T_k=2^{8-k}[8-k]!\,\omega_8\,T_k$, $0\leq k\leq 8$. Then we have
\begin{align*}
\hat T_8&=-\textstyle\frac{97}{3244032}\hat  t^8-\textstyle\frac{727}{4055040} \hat t^6 \hat s-\textstyle\frac{43}{42240}\hat t^5 \hat v+\textstyle\frac{53}{1622016}\hat t^4 \hat s^2 -\textstyle\frac{4703}{1013760} \hat t^4 \hat u_1 +\textstyle\frac1{10560} \hat t^3 \hat s \hat v\\
&\quad-\textstyle\frac{3229}{190080} \hat t^3 \hat w_1-\textstyle\frac{29}{12165120} \hat t^2 \hat s^3+\textstyle\frac{329}{1520640} \hat t^2 \hat s \hat u_1-\textstyle\frac{593}{12672} \hat t^2 \hat x_1-\textstyle\frac1{380160} \hat t \hat s^2 \hat v+\textstyle\frac{13}{38016} \hat t \hat s \hat w_1\\
&\quad-\textstyle\frac{91}{1056} \hat t \hat y+\textstyle\frac1{48660480} \hat s^4-\textstyle\frac1{608256} \hat s^2  \hat u_1+\textstyle\frac7{25344}  \hat s  \hat x_1-\textstyle\frac7{88}\hat z,
\\
 \hat T_7&=\textstyle\frac{13727}{506383488}\hat t^9+\textstyle\frac{373}{2344368}\hat t^7 \hat s+\textstyle\frac{2063}{2344368}\hat t^6 \hat v-\textstyle\frac{101}{3516552}\hat t^5 \hat s^2+\textstyle\frac{1133}{293046}\hat t^5 \hat u_1\\
&\quad-\textstyle\frac{275}{3516552} \hat t^4 \hat s \hat v+\textstyle\frac{140965}{10549656} \hat t^4 \hat w_1+\textstyle\frac5{2637414} \hat  t^3 \hat s^3-\textstyle\frac{215}{1318707} \hat t^3 \hat s\hat  u_1+\textstyle\frac{125}{3757} \hat t^3 \hat x_1\\
&\quad+ \textstyle\frac5{2637414} \hat t^2 \hat s^2 \hat v-\textstyle\frac5{22542} \hat t^2 \hat s \hat w_1+ \textstyle\frac5{102}\hat t^2 \hat y,
\\
\hat T_6&=-\textstyle\frac{67}{3317760}\hat t^{10}-\textstyle\frac{41}{368640} \hat t^8 \hat s-\textstyle\frac{13}{23040} \hat t^7 \hat v+\textstyle\frac7{368640} \hat t^6 \hat s^2-\textstyle\frac{101}{46080} \hat t^6 \hat u_1\\
&\quad+\textstyle\frac1{23040} \hat t^5 \hat s \hat v-\textstyle\frac1{160} \hat t^5 \hat w_1-\textstyle\frac1{1105920}\hat t^4 \hat s^3 +\textstyle\frac1{15360} \hat t^4 \hat s \hat u_1-\textstyle\frac1{96}\hat t^4 \hat x_1,
\\
\hat T_5&=\textstyle\frac{313}{25280640}\hat t^{11}+\textstyle\frac{383}{6320160}\hat t^9 \hat s+\textstyle\frac{61}{234080}\hat t^8 \hat v -\textstyle\frac1{105336}\hat t^7 \hat s^2+\textstyle\frac1{1254}\hat t^7 \hat u_1 -\textstyle\frac 1{75240}\hat t^6 \hat s \hat v +\textstyle\frac7{5016}\hat t^6 \hat w_1,
\\
\hat T_4&=-\textstyle\frac{79}{12773376} \hat t^{12}-\textstyle\frac5{193536} \hat t^{10} \hat s-\textstyle\frac1{12096} \hat t^9 \hat v +\textstyle\frac1{387072} \hat t^8 \hat s^2-\textstyle\frac1{6912} \hat t^8 \hat u_1,
\\
\hat T_3&=\textstyle\frac{107}{42007680}\hat t^{13}+\textstyle\frac7{807840}\hat t^{11} \hat s+\textstyle\frac1{73440}\hat t^{10}\hat v,\\
\hat T_2&=-\textstyle\frac5{7907328} \hat t^{14}- \textstyle\frac1{608256} \hat t^{12}\hat s,\\
\hat T_1&=-\textstyle\frac1{7413120} \hat t^{15},\\
\hat T_0&=\textstyle\frac1{7413120} \hat t^{16}.
\end{align*}
\end{theorem}

\begin{proof}
Let $D\in\calK(\OO^2)$ be as in Lemma \ref{lem:T8} and observe that $N(D)=N_1\cup N_2$ where
\begin{align*}
N_1&=\left\{\left(\begin{pmatrix}0\\x\end{pmatrix},\begin{pmatrix}v\\r x\end{pmatrix}\right)\mid r>0,x,v\in\OO,\norm{x}=1,r^2+\sqnorm v=1\right\},\\
N_2&=\left\{\left(\begin{pmatrix}0\\x\end{pmatrix},\begin{pmatrix}v\\0\end{pmatrix}\right)\mid x,v\in\OO,\norm{x}\leq\norm{v}=1\right\}.
\end{align*}
Take any $\omega\in\Omega^{8,7}(S\OO^2)^{\b{\Spin(9)}}$. First, clearly, $\int_{N(D)}\omega=\int_{N_2}\omega$. Second, in  the point $p=\big(0,(\begin{smallmatrix}1\\0\end{smallmatrix})\big)$, $\omega$ can be written uniquely as
\begin{align*}
\omega=c\,\theta_1^0\wedge\cdots\wedge\theta_1^7\wedge\varphi_0^1\wedge\cdots\wedge\varphi_0^7 + \wt\omega,
\end{align*}
where $c\in\RR$ and $\wt\omega=0$ on $\OO\times S^7$. Finally, since the subgroup
\begin{align*}
\left\{\left(\begin{pmatrix}\chi_g&0\\0& g\end{pmatrix},\begin{pmatrix}0\\y\end{pmatrix}\right)\mid g\in \Spin(7),\ y\in \OO \right\}\subset \b{\Spin(9)},
\end{align*}
acts transitively on $\OO \times S^7$ and $\omega$ is invariant, we have
\begin{align*}
\int_{N(D)}\omega=c\vol_{8+7}(N_2)=c\cdot 8(\omega_8)^2.
\end{align*}
This gives us $\psi_8^{(j)}(D)$, $j=1,\dots,27$, and, according to Lemma \ref{lem:T8}, the expression of $\hat T_8=\omega_8T_8$ in the basis $ \psi_8^{(1)},\dots, \psi_8^{(27)}$ of $\Val_8^{\Spin(9)}$.

 As for the remaining valuations, by \eqref{eq:tmu15} and \eqref{eq:LambdakT8} for $0\leq k\leq 7$ we have $ \hat T_k=(-\hat t)^{8-k}*\hat T_8$. The rest is then a matter of computation according to Theorem \ref{thm:main}.
\end{proof}

\begin{remark}
Observe that, as expected, both sides of the last equation in Theorem \ref{thm:Tk} equal $70\pi^8\chi$.
\end{remark}

\section{Questions and open problems}

\label{s:Q}

Let us conclude by collecting here several questions  related to the results of our article.

\begin{enuma}
\item  A natural completion of Theorem  \ref{thm:Plm} would be the second fundamental theorem for the isotropy representation $\Imag\OO\oplus\OO$ of $\Spin(7)$, i.e., a description of the relations among the generating invariants. This task,  while easy in the classical case $\Imag\OO$ (see \cite[\S II.17]{Weyl:ClassicalGroups}),  becomes already rather complicated when the spin representation $\OO$ alone is considered, as done by Schwarz  \cite{Schwarz:InvariantG2andSpin7}.
\item Related to (a) is the following remark: Keeping the notation of Theorem \ref{thm:Spin7invforms}, the dimensions of the subspaces
\begin{align*}
\largewedge^{K,M}V^*=\bigoplus_{\substack{k+l=K\\m+n=M}}\largewedge^{k,l,m,n}V^*
\end{align*}
were computed by Bernig and Voide \cite[Proposition 4.2]{BernigVoide:Spin}. Comparison with their result shows that there exist (in fact numerous) relations among the generators listed in Theorem \ref{thm:Spin7invforms}. For example, there are 12 monomials in the generators having bi-degree $(6,1)$ but $\dim\largewedge^{6,1}V^*=10$. In fact, a {\sc Maple} computation in coordinates shows that
\begin{align*}
2\,[0,2,1,0]\wedge[2,2,0,0]-[1,2,0,0]\wedge[1,2,1,0]+3\,[1,0,1,0]\wedge[1,4,0,0]&=0,\\
2\,[1,1,0,1]\wedge[0,4,0,0]-[1,2,0,0]\wedge[0,3,0,1]+3\,[0,1,0,1]\wedge[1,4,0,0]&=0.
\end{align*}
Our attempts to prove these relations by hand have not been met  with success. No structural result about the set of relations is known to us either.

\item  
The description of the algebra $\Val^{\Spin(9)}$ we give in Theorem \ref{thm:main} is explicit, yet extremely complicated. However, it is far from clear whether it can be simplified by choosing a more convenient set of generators  (of either the ring or the ideal). All our attempts in this direction have resulted in simplifying only a few generators of the ideal $I$, perturbing at the same time the others. A structural understanding of the generators, e.g., a combinatorial model that yields a generating set of $I$ would be highly desirable. In particular, it would be interesting to understand why the sequences $(d_k)$ and $(\beta_{1,k})$ of numbers of $k$-homogeneous generators of the ring and of the ideal, respectively, are both unimodal, namely,
\begin{align}
\label{eq:dk}
\renewcommand{\arraystretch}{1.3}
\begin{array}{|c||c|c|c|c|c|c|c|c|c|}
\hline
k &1&2&3&4&5&6&7&8&\geq 9
\\\hline
 d_k & 1&1&1&2&3&2&1&1&0
\\\hline
	\end{array}
\end{align}
and
\begin{align*}
\renewcommand{\arraystretch}{1.3}
\begin{array}{|c||c|c|c|c|c|c|c|c|c|c|c|c|c|c|}
\hline
k &\leq 5&6&7&8&9&10&11&12&13&14&15&16&\geq 17
\\\hline
 \beta_{1,k} & 0&1&2&5&20&22&10&8&5&3&1&1&0
\\\hline
	\end{array}\,,
\end{align*} 
as well as why the the curious symmetry in the range $2\leq k\leq 8$ arises in \eqref{eq:dk}. The $\beta_{1,k}$ are the so-called first graded Betti numbers of the ideal $I$. It might be helpful for identifying a combinatorial model (but probably computationally very expensive) to determine the rest of the Betti diagram $(\beta_{j,k})$ of $I$, cf. \cite{Eisenbud:Syzygies}.

\item A  recent paper of Fu and the second-named author \cite{FuW:Riemannian} describes for any Riemannian manifold $M$ a distinguished subspace of curvature measures on which the Lipschitz--Killing algebra acts in a universal fashion. Applied to the Cayley plane $\OO P^2$, this construction yields a canonical subspace of $\b{\Spin(9)}$-invariant angular valuations on the octonionic plane $\OO^2$. It seems plausible that similar geometric considerations will reveal further structure on the algebra  $\Val^{\Spin(9)}$.

\item 
It follows from Theorem \ref{thm:Tk} that the valuations $T_k$, $0\leq k\leq 8$, lie in the subalgebra of $\Val^{\Spin(9)}$ generated by the generators $\hat t,\hat x,\hat v,\hat u_1,\hat w_1,\hat x_1,\hat y,\hat z$ defined by  the differential forms
\begin{align*}
[0,4,0,0]\wedge[0,4,0,0]\wedge[7-l,0,l,0], \quad 0\leq l\leq 7. 
\end{align*}
So far we do not have any explanation of this fact either.

\end{enuma}

\appendix

\section{Generators of the ideal defining the algebra $\Val^{\Spin(9)}$}

\label{app:relations}

\noindent\textbf{Degree $6$.}
{\footnotesize
\begin{align*}
f_1&=v^2+2tw_1 - 2su_1.
\end{align*}
}

\noindent\textbf{Degree $7$.}
{\footnotesize
\begin{align*}
f_2&=vu_1+5t x_1 -3s w_1,
\\
f_3&=vu_2+\textstyle\frac{109}{32}t^7 +\textstyle\frac{23}2t^5 s +\textstyle\frac{317}{48}t^4 v- \textstyle\frac{125}{96}t^3 s^2- \textstyle\frac{317}2t^3 u_1 +\textstyle\frac{161}4t^3 u_2- 5t^2 s v- \textstyle\frac{4835}6t^2 w_1 +\textstyle\frac{127}4t^2 w_2 +291t^2 w_3\\
&\quad- \textstyle\frac 1{48}t s^3 -\textstyle\frac{95}{12}t s u_1 +t s u_2- 1440t x_1 +\textstyle\frac{159}4t x_2 -\textstyle\frac 1{48}s^2 v- \textstyle\frac32s w_1 +\textstyle\frac{21}4s w_2 +3s w_3.
\end{align*}
}

\noindent\textbf{Degree $8$.}
{\footnotesize
\begin{align*}
f_4&=sx_2+\textstyle\frac{67}{660}t^8 +\textstyle\frac{151}{330}t^6 s +\textstyle\frac{37}{99} t^5 v +\textstyle\frac{65}{198}t^4 s^2 -\textstyle\frac{233}{66}t^4 u_1 +t^4 u_2 +\textstyle\frac{26}{99}t^3 s v -\textstyle\frac{650}{33}t^3 w_1+\textstyle\frac 23t^3 w_2+\textstyle\frac{62}{11}t^3 w_3 -\textstyle\frac{1}{36}t^2 s^3 \\
&\quad -\textstyle\frac{122}{33}t^2 s u_1 +t^2 s u_2 -36t^2 x_1 +t^2 x_2- \textstyle\frac19t s^2 v-\textstyle\frac{650}{33}t s w_1 +\textstyle\frac23t s w_2 +\textstyle\frac{62}{11}t s w_3 -\textstyle\frac16s^2  u_1 -36s  x_1,\\
f_5&=vw_1+9t y-4s  x_1,
\\
f_6&=u_1^2+16t y-6s  x_1,
\\
f_7&=u_1u_2-\textstyle\frac{2983}{880}t^8 -\textstyle\frac{20051}{1760}t^6 s-\textstyle\frac{3365}{528}t^5 v +\textstyle\frac{1559}{1056} t^4 s^2 +\textstyle\frac{20939}{132}t^4 u_1 -\textstyle\frac{481}{12}t^4 u_2 +\textstyle\frac{1535}{264}t^3 s v +\textstyle\frac{318149}{396}t^3 w_1-\textstyle\frac{379}{12}t^3 w_2 \\
&\quad-\textstyle\frac{12769}{44}t^3 w_3 +\textstyle\frac{1}{144}t^2 s^3 +\textstyle\frac{2099}{198}t^2 s u_1 -\textstyle\frac12t^2 s u_2+\textstyle\frac{8573}{6}t^2 x_1 -\textstyle\frac{159}{4}t^2 x_2-\textstyle\frac{1}{144}t s^2 v +\textstyle\frac{2867}{396}t s w_1 -\textstyle\frac{59}{12}t s w_2 \\
&\quad-\textstyle\frac{9}{44}t s w_3 -18t y -\textstyle\frac{1}{36}s^2  u_1 +6 s  x_1  +\textstyle\frac13v  w_2 +4v  w_3,
\\
f_8&=u_2^2-\textstyle\frac{4199}{660}t^8 -\textstyle\frac{46507}{2640}t^6 s +\textstyle\frac{2765}{99}t^5 v +\textstyle\frac{821}{396}t^4 s^2 +\textstyle\frac{17003}{33}t^4 u_1 -\textstyle\frac{571}{6}t^4 u_2 +\textstyle\frac{925}{132}t^3 s v +\textstyle\frac{430433}{198}t^3 w_1-\textstyle\frac{250}{3}t^3 w_2 \\
&\quad-\textstyle\frac{14731}{22}t^3 w_3 -\textstyle\frac1{18} t^2 s^3 +\textstyle\frac{41}{18}t^2 s u_1 -\textstyle\frac13 t^2 s u_2 +\textstyle\frac{38537}{11}t^2 x_1 -95 t^2 x_2 +\textstyle\frac{1}{12}t s^2 v-\textstyle\frac{1807}{66}t s w_1 -19t s w_2 +\textstyle\frac{17}{22}t s w_3 \\
&\quad-\textstyle\frac{5616}{11}t y+\textstyle\frac{1}{144}s^4 +\textstyle\frac13s^2  u_1 -\textstyle\frac16s^2  u_2 -\textstyle\frac{210}{11}s  x_1  +2v  w_2-\textstyle\frac{36}{11}v  w_3.
\end{align*}
}

\noindent\textbf{Degree $9$.}
{\footnotesize
\begin{align*}
f_9&=ts^4+\textstyle\frac{25}{9}t^9 +\textstyle\frac{52}{7}t^7 s  -16t^6 v -10t^5 s^2 -64t^5 u_1  -32t^4 s v +\textstyle\frac{128}{3}t^4 w_1 -\textstyle\frac{28}{3}t^3 s^3+\textstyle\frac{64}{3}t^3 s u_1+\textstyle\frac{16}{3}t^2 s^2 v,
\\
f_{10}&=ts^2u_1+\textstyle\frac{2}{1521}t^9 -\textstyle\frac{580}{1183}t^7 s -\textstyle\frac{445}{169}t^6 v -\textstyle\frac{174}{169}t^5 s^2 -\textstyle\frac{2471}{169}t^5 u_1 -\textstyle\frac{534}{169}t^4 s v -\textstyle\frac{16208}{507}t^4 w_1 +\textstyle\frac{88}{507} t^3 s^3 -\textstyle\frac{1822}{507}t^3 s u_1  \\
&\quad+\textstyle\frac{600}{13} t^3 x_1 +\textstyle\frac{257}{507}t^2 s^2 v +\textstyle\frac{152}{13}t^2 s w_1,
\\
f_{11}&=ts^2u_2+\textstyle\frac{145939}{18252}t^9 +\textstyle\frac{207881}{7098}t^7 s +\textstyle\frac{10897}{338}t^6 v -\textstyle\frac{471}{676}t^5 s^2 -\textstyle\frac{144604}{507}t^5 u_1 +79t^5 u_2-\textstyle\frac{864}{169}t^4 s v-\textstyle\frac{2558476}{1521}t^4 w_1 \\
&\quad+26t^4 w_2  +552t^4 w_3 +\textstyle\frac{167}{1521}t^3 s^3 -\textstyle\frac{6926}{1521}t^3 s u_1 -2t^3 s u_2  -\textstyle\frac{43200}{13}t^3 x_1+90t^3 x_2 +\textstyle\frac{2531}{3042}t^2 s^2 v -\textstyle\frac{180}{13}t^2 s w_1 \\
&\quad-2t^2 s w_2 -24 t^2 s w_3,
\\
f_{12}&=tsx_1-\textstyle\frac{25897}{8791380}t^9 -\textstyle\frac{6180}{341887}t^7 s -\textstyle\frac{78332}{732615}t^6 v +\textstyle\frac{25481}{2930460}t^5 s^2 -\textstyle\frac{219157}{488410}t^5 u_1 +\textstyle\frac{12499}{293046}t^4 s v-\textstyle\frac{59099}{48841}t^4 w_1-\textstyle\frac{7}{293046} t^3 s^3 \\
&\quad+\textstyle\frac{16481}{97682}t^3 s u_1 -\textstyle\frac{1575}{3757}t^3 x_1  -\textstyle\frac{7}{293046}t^2 s^2 v +\textstyle\frac{1889}{3757}t^2 s w_1 +\textstyle\frac{168}{17}t^2 y,
\\
f_{13}&=tvw_2-\textstyle\frac{345859705}{464184864}t^9 -\textstyle\frac{201050353}{85960160}t^7 s -\textstyle\frac{17103107}{96705180}t^6 v +\textstyle\frac{40427333}{77364144}t^5 s^2 +\textstyle\frac{329863723}{8058765}t^5 u_1-\textstyle\frac{152}{15}t^5 u_2 +\textstyle\frac{78728239}{38682072}t^4 s v \\
&\quad+\textstyle\frac{3748338689}{19341036}t^4 w_1  -13 t^4 w_2 -\textstyle\frac{3111}{44}t^4 w_3+\textstyle\frac{9205}{879138}t^3 s^3 +\textstyle\frac{81780211}{19341036}t^3 s u_1-\textstyle\frac12t^3 s u_2 +\textstyle\frac{2594261}{7514}t^3 x_1 -\textstyle\frac{19}{2}t^3 x_2\\
&\quad -\textstyle\frac{4007}{1172184}t^2 s^2 v +\textstyle\frac{507563}{165308}t^2 s w_1 -\textstyle\frac{13}{6}t^2 s w_2 -\textstyle\frac{295}{44}t^2 s w_3+\textstyle\frac{990}{17}t^2 y,
\\
f_{14}&=tvw_3-\textstyle\frac{2269121}{7033104}t^9 -\textstyle\frac{130671239}{123079320}t^7 s -\textstyle\frac{30778613}{52748280}t^6 v +\textstyle\frac{2169641}{10549656}t^5 s^2 +\textstyle\frac{21735751}{1465230} t^5 u_1 -\textstyle\frac{176}{45}t^5 u_2 +\textstyle\frac{1002280}{1318707} t^4 s v \\
&\quad+\textstyle\frac{65308993}{879138}t^4 w_1 -\textstyle\frac{11}{3} t^4 w_2 -\textstyle\frac{55}{2}t^4 w_3 -\textstyle\frac{9869}{5274828}t^3 s^3 +\textstyle\frac{205373}{146523}t^3 s u_1 +\textstyle\frac{498421}{3757}t^3 x_1-\textstyle\frac{11}{3} t^3 x_2 -\textstyle\frac{9869}{5274828}t^2 s^2v \\
&\quad+\textstyle\frac{3014}{3757}t^2 s w_1-\textstyle\frac{11}{9}t^2 s w_2 +\textstyle\frac73t^2 s w_3 +\textstyle\frac{165}{17}t^2 y,
\\
f_{15}&=tz+\textstyle\frac{28331}{1107713880}t^9 +\textstyle\frac{1249}{6153966}t^7 s +\textstyle\frac{87019}{61539660}t^6 v -\textstyle\frac{101}{123079320}t^5 s^2 +\textstyle\frac{115473}{13675480}t^5 u_1 -\textstyle\frac{55}{24615864}t^4 s v+\textstyle\frac{388147}{9230949}t^4 w_1 \\
&\quad+\textstyle\frac{1}{18461898}t^3 s^3 -\textstyle\frac{43}{9230949}t^3 s u_1 +\textstyle\frac{26449}{157794}t^3 x_1 +\textstyle\frac{1}{18461898}t^2 s^2 v -\textstyle\frac{1}{157794}t^2 s w_1 +\textstyle\frac{179}{357}t^2 y,
\\
f_{16}&=s^3v+\textstyle\frac{2978}{507}t^9 +\textstyle\frac{10272}{169}t^7 s +\textstyle\frac{32411}{169}t^6 v +\textstyle\frac{8254}{169}t^5 s^2 +\textstyle\frac{48612}{169}t^5 u_1 +\textstyle\frac{6479}{169}t^4 s v-\textstyle\frac{111376}{169}t^4 w_1 -\textstyle\frac{3492}{169}t^3 s^3 \\
&\quad-\textstyle\frac{51884}{169}t^3 s u_1 -\textstyle\frac{4560}{13}t^3 x_1 -\textstyle\frac{8731}{169}t^2 s^2 v -\textstyle\frac{864}{13}t^2 s w_1,
\\
f_{17}&=s^2w_1+\textstyle\frac{423211}{2197845}t^9 +\textstyle\frac{348476}{341887}t^7 s +\textstyle\frac{1288358}{244205}t^6 v -\textstyle\frac{203141}{244205}t^5 s^2 +\textstyle\frac{3244737}{244205}t^5 u_1-\textstyle\frac{199599}{48841}t^4 s v -\textstyle\frac{3397013}{146523}t^4 w_1 \\
&\quad+\textstyle\frac{308}{146523}t^3 s^3 -\textstyle\frac{2370856}{146523}t^3 s u_1 -\textstyle\frac{988500}{3757}t^3 x_1+\textstyle\frac{308}{146523}t^2 s^2 v -\textstyle\frac{158718}{3757}t^2 s w_1-\textstyle\frac{1320}{17}t^2 y,
\\
f_{18}&=s^2w_2-\textstyle\frac{1424473757}{580231080}t^9 -\textstyle\frac{3491695369}{451290840}t^7 s +\textstyle\frac{99305936}{24176295}t^6 v -\textstyle\frac{148345751}{193410360} t^5 s^2+\textstyle\frac{1235838493}{8058765}t^5 u_1 -\textstyle\frac{544}{15}t^5 u_2 \\
&\quad-\textstyle\frac{9534368}{4835259}t^4 s v +\textstyle\frac{3137856266}{4835259}t^4 w_1 -49t^4 w_2 -\textstyle\frac{2610}{11}t^4 w_3 +\textstyle\frac{54913}{390728}t^3 s^3 -\textstyle\frac{38941391}{4835259}t^3 s u_1 -2t^3 s u_2 \\
&\quad+\textstyle\frac{3601952}{3757}t^3 x_1 -34t^3 x_2+\textstyle\frac{55672}{439569}t^2 s^2 v -\textstyle\frac{995606}{41327}t^2 s w_1 -\textstyle\frac{28}{3}t^2 s w_2 -\textstyle\frac{146}{11} t^2 s w_3 -\textstyle\frac{1584}{17}t^2 y,
\\
f_{19}&=s^2w_3-\textstyle\frac{271307}{450840}t^9 -\textstyle\frac{9413891}{4733820}t^7 s -\textstyle\frac{228614}{507195}t^6 v +\textstyle\frac{68117}{4057560}t^5 s^2 +\textstyle\frac{1697401}{56355}t^5 u_1 -\textstyle\frac{352}{45}t^5 u_2+\textstyle\frac{7022}{101439}t^4 s v \\
&\quad+\textstyle\frac{4702732}{33813}t^4 w_1-\textstyle\frac{22}{3}t^4 w_2-55t^4 w_3-\textstyle\frac{193}{101439}t^3 s^3-\textstyle\frac{3838}{3757}t^3 s u_1+\textstyle\frac{63756}{289}t^3 x_1-\textstyle\frac{22}{3}t^3 x_2 -\textstyle\frac{193}{101439}t^2 s^2 v\\
&\quad -\textstyle\frac{1188}{289}t^2 s w_1 -\textstyle\frac{22}{9}t^2 s w_2+\textstyle\frac{14}{3}t^2 s w_3 -\textstyle\frac{264}{17}t^2 y,
\\
f_{20}&=sy-\textstyle\frac{28331}{19780605}t^9 -\textstyle\frac{4996}{439569}t^7 s -\textstyle\frac{174038}{2197845}t^6 v +\textstyle\frac{101}{2197845}t^5 s^2-\textstyle\frac{115473}{244205}t^5 u_1 +\textstyle\frac{55}{439569}t^4 s v -\textstyle\frac{3105176}{1318707}t^4 w_1 \\
&\quad-\textstyle\frac{4}{1318707}t^3 s^3 +\textstyle\frac{344}{1318707}t^3 s u_1 -\textstyle\frac{105796}{11271}t^3 x_1 -\textstyle\frac{4}{1318707}t^2 s^2 v +\textstyle\frac{4}{11271}t^2 s w_1-\textstyle\frac{1279}{51}t^2 y,
\\
f_{21}&=vx_1-\textstyle\frac{198317}{26374140}t^9-\textstyle\frac{8743}{146523}t^7 s -\textstyle\frac{609133}{1465230}t^6 v +\textstyle\frac{707}{2930460}t^5 s^2 -\textstyle\frac{2424933}{976820}t^5 u_1+\textstyle\frac{385}{586092}t^4 s v-\textstyle\frac{5434058}{439569}t^4 w_1\\
&\quad-\textstyle\frac{7}{439569}t^3 s^3 +\textstyle\frac{602}{439569}t^3 s u_1 -\textstyle\frac{185143}{3757}t^3 x_1 -\textstyle\frac{7}{439569} t^2 s^2 v +\textstyle\frac{7}{3757}t^2 s w_1-\textstyle\frac{2251}{17}t^2 y,
\\
f_{22}&=vx_2-\textstyle\frac{1752942221}{2320924320}t^9 -\textstyle\frac{12823165981}{2707745040}t^7 s -\textstyle\frac{27192059569}{1160462160}t^6 v -\textstyle\frac{1693971839}{2320924320}t^5 s^2-\textstyle\frac{590022453}{5372510} t^5 u_1 -\textstyle\frac{391}{180}t^5 u_2\\
&\quad -\textstyle\frac{102946787}{58023108}t^4 s v-\textstyle\frac{2315759090}{4835259}t^4 w_1 +\textstyle\frac{77}{12}t^4 w_2-\textstyle\frac{477}{22}t^4 w_3 +\textstyle\frac{1397087}{21099312}t^3 s^3-\textstyle\frac{35304473}{19341036}t^3 s u_1-\textstyle\frac23t^3 s u_2 \\
&\quad-\textstyle\frac{6613466}{3757}t^3 x_1-\textstyle\frac{49}{12}t^3 x_2 +\textstyle\frac{2569271}{21099312}t^2 s^2 v +\textstyle\frac{95061}{41327}t^2 s w_1 +\textstyle\frac{221}{36}t^2 s w_2 -\textstyle\frac{991}{66}t^2 s w_3-\textstyle\frac{67176}{17} t^2 y,
\\
f_{23}&=u_1w_1-\textstyle\frac{28331}{1861704}t^9-\textstyle\frac{6245}{51714}t^7 s-\textstyle\frac{87019}{103428}t^6 v +\textstyle\frac{101}{206856} t^5 s^2-\textstyle\frac{115473}{22984}t^5 u_1+\textstyle\frac{275}{206856}t^4 s v-\textstyle\frac{1940735}{77571}t^4 w_1 \\
&\quad-\textstyle\frac{5}{155142}t^3 s^3 +\textstyle\frac{215}{77571} t^3 s u_1 -\textstyle\frac{132245}{1326}t^3 x_1 -\textstyle\frac{5}{155142}t^2 s^2 v +\textstyle\frac{5}{1326}t^2 s w_1-\textstyle\frac{805}{3}t^2 y,
\\
f_{24}&=u_1w_2+\textstyle\frac{1696241603}{2320924320}t^9 +\textstyle\frac{3770703661}{1805163360}t^7 s -\textstyle\frac{406638451}{193410360}t^6 v -\textstyle\frac{24948791}{48352590}t^5 s^2-\textstyle\frac{875507521}{16117530}t^5 u_1 +\textstyle\frac{32}{3}t^5 u_2\\
&\quad -\textstyle\frac{77926367}{38682072} t^4 s v -\textstyle\frac{4772084077}{19341036} t^4 w_1 +\textstyle\frac{41}{3}t^4 w_2+\textstyle\frac{3279}{44}t^4 w_3-\textstyle\frac{4129}{390728}t^3 s^3 -\textstyle\frac{79225709}{19341036}t^3 s u_1 +\textstyle\frac12 t^3 s u_2 \\
&\quad-\textstyle\frac{10794811}{22542}t^3 x_1+10t^3 x_2 +\textstyle\frac{1460}{439569}t^2 s^2 v-\textstyle\frac{1338605}{495924}t^2 s w_1 +\textstyle\frac{17}{6}t^2 s w_2 +\textstyle\frac{287}{44}t^2 s w_3-\textstyle\frac{4082}{17} t^2 y,
\\
f_{25}&=u_1w_3+\textstyle\frac{49421519}{158244840}t^9 +\textstyle\frac{244104829}{246158640}t^7 s +\textstyle\frac{232604}{1318707}t^6 v -\textstyle\frac{2701219}{13187070}t^5 s^2 -\textstyle\frac{1638935}{97682}t^5 u_1 +\textstyle\frac{176}{45} t^5 u_2 \\
&\quad-\textstyle\frac{7994645}{10549656}t^4 s v -\textstyle\frac{107533550}{1318707}t^4 w_1 +\textstyle\frac{11}{3}t^4 w_2 +\textstyle\frac{83}{3}t^4 w_3+\textstyle\frac{9583}{5274828}t^3 s^3 -\textstyle\frac{1842208}{1318707} t^3 s u_1-\textstyle\frac{1712623}{11271}t^3 x_1 \\
&\quad+\textstyle\frac{11}{3}t^3 x_2 +\textstyle\frac{9583}{5274828}t^2 s^2 v-\textstyle\frac{17941}{22542}t^2 s w_1 +\textstyle\frac{11}{9}t^2 s w_2-\textstyle\frac{11}{6}t^2 s w_3-\textstyle\frac{1969}{51}t^2 y,
\\
f_{26}&=u_2w_1+\textstyle\frac{10493459}{48352590}t^9 +\textstyle\frac{124941559}{112822710}t^7 s +\textstyle\frac{23738857}{6447012}t^6 v +\textstyle\frac{30171637}{193410360}t^5 s^2 +\textstyle\frac{181494101}{21490040}t^5 u_1 +\textstyle\frac{14}{15}t^5 u_2 \\
&\quad+\textstyle\frac{1710261}{4298008}t^4 s v +\textstyle\frac{179353831}{9670518} t^4 w_1 -\textstyle\frac{10}{3}t^4 w_2+\textstyle\frac{119}{11}t^4 w_3-\textstyle\frac{2101}{439569}t^3 s^3 +\textstyle\frac{2720161}{4835259}t^3 s u_1 +\textstyle\frac{771061}{7514} t^3 x_1 \\
&\quad+2t^3 x_2-\textstyle\frac{2101}{439569} t^2 s^2 v-\textstyle\frac{118517}{123981}t^2 s w_1-\textstyle\frac{10}{3}t^2 s w_2+\textstyle\frac{64}{11}t^2 s w_3+\textstyle\frac{6858}{17}t^2 y,
\\
f_{27}&=u_2w_2-\textstyle\frac{3676123513}{2785109184}t^9-\textstyle\frac{5148236867}{676936260}t^7 s-\textstyle\frac{5587823539}{154728288}t^6 v-\textstyle\frac{2045441273}{1547282880}t^5 s^2-\textstyle\frac{33255409283}{193410360}t^5 u_1-\textstyle\frac{154}{15}t^5 u_2\\
&\quad-\textstyle\frac{24508851}{4298008}t^4 s v-\textstyle\frac{44687560949}{58023108}t^4 w_1-\textstyle\frac{335}{24}t^4 w_2-\textstyle\frac{1509}{22}t^4 w_3-\textstyle\frac{2651825}{42198624}t^3 s^3-\textstyle\frac{1075520785}{58023108}t^3 s u_1-\textstyle\frac{13}{24}t^3 s u_2\\
&\quad-\textstyle\frac{19781025}{7514}t^3 x_1-\textstyle\frac{77}{8}t^3 x_2-\textstyle\frac{4800829}{42198624}t^2 s^2 v-\textstyle\frac{5379211}{165308}t^2 s w_1-\textstyle\frac{77}{24}t^2 s w_2-\textstyle\frac{101}{22}t^2 s w_3-\textstyle\frac{79524}{17}t^2 y,
\\
f_{28}&=u_2w_3-\textstyle\frac{1258556777}{4641848640}t^9-\textstyle\frac{126463971}{75215140}t^7 s-\textstyle\frac{1362469253}{154728288}t^6 v-\textstyle\frac{20957051}{309456576}t^5 s^2-\textstyle\frac{870536637}{21490040}t^5 u_1-\textstyle\frac{14}{15}t^5 u_2\\
&\quad-\textstyle\frac{10413167}{38682072}t^4 s v-\textstyle\frac{3032611687}{19341036}t^4 w_1-\textstyle\frac{7}{8}t^4 w_2-\textstyle\frac{147}{22}t^4 w_3-\textstyle\frac{17645}{14066208}t^3 s^3-\textstyle\frac{34289245}{38682072}t^3 s u_1\\
&\quad-\textstyle\frac{10267753}{22542}t^3 x_1-\textstyle\frac{7}{8}t^3 x_2-\textstyle\frac{17645}{14066208}t^2 s^2 v-\textstyle\frac{818093}{495924}t^2 s w_1-\textstyle\frac{7}{24}t^2 s w_2+\textstyle\frac{7}{22}t^2 s w_3-\textstyle\frac{12606}{17}t^2 y.
\end{align*}
}

\noindent\textbf{Degree $10$.}
{\footnotesize
\begin{align*}
f_{29}&=t^3s^2v-\textstyle\frac{97}{560}t^{10}-\textstyle\frac{171}{112}t^8 s-\textstyle\frac{297}{70}t^7 v-\textstyle\frac{81}{80}t^6 s^2-\textstyle\frac{21}{5}t^6 u_1-\textstyle\frac{3}{10}t^5 s v+\textstyle\frac{112}{5}t^5 w_1+\textstyle\frac{41}{80}t^4 s^3+\textstyle\frac{38}{5}t^4 s u_1,
\\
f_{30}&=t^3 s w_1-\textstyle\frac{113}{17920} t^{10}-\textstyle\frac{491}{17920}t^8 s-\textstyle\frac{59}{480} t^7 v+\textstyle\frac{53}{1536} t^6 s^2-\textstyle\frac{303}{2240}t^6 u_1+\textstyle\frac{171}{1120}t^5 s v+\textstyle\frac{53}{35}t^5 w_1-\textstyle\frac{1}{17920}t^4 s^3\\
&\quad+\textstyle\frac{1129}{2240}t^4 s u_1+\textstyle\frac{55}{7}t^4 x_1,
\\
f_{31}&=t^3 s w_2+\textstyle\frac{7823}{134400}t^{10}+\textstyle\frac{1091}{8960}t^8 s-\textstyle\frac{437}{840}t^7 v-\textstyle\frac{217}{3840}t^6 s^2 -\textstyle\frac{6171}{1120}t^6 u_1 +\textstyle\frac{101}{80}t^6 u_2-\textstyle\frac{3}{14}t^5 s v-\textstyle\frac{1996}{105} t^5 w_1+\textstyle\frac{13}{5}t^5 w_2\\
&\quad+\textstyle\frac{42}{5}t^5 w_3-\textstyle\frac{23}{8960}t^4 s^3-\textstyle\frac{293}{672}t^4 s u_1+\textstyle\frac{9}{80}t^4 s u_2-\textstyle\frac{804}{35}t^4 x_1+\textstyle\frac{9}{10}t^4 x_2,
\\
f_{32}&=t^3 s w_3+\textstyle\frac{3499}{89600}t^{10}+\textstyle\frac{2237}{17920}t^8 s-\textstyle\frac{1}{24}t^7 v+\textstyle\frac{5}{1536}t^6 s^2-\textstyle\frac{16307}{6720}t^6 u_1+\textstyle\frac{253}{480}t^6 u_2-\textstyle\frac{103}{5040}t^5 s v-\textstyle\frac{1154}{105} t^5 w_1+\textstyle\frac{11}{15}t^5 w_2\\
&\quad+\textstyle\frac{16}{5}t^5 w_3-\textstyle\frac{107}{53760} t^4 s^3-\textstyle\frac{451}{6720}t^4 s u_1+\textstyle\frac{11}{160}t^4 s u_2-\textstyle\frac{638}{35}t^4 x_1+\textstyle\frac{11}{20}t^4 x_2,
\\
f_{33}&=t^3 y+\textstyle\frac{769}{6451200}t^{10}+\textstyle\frac{1961}{2150400}t^8 s+\textstyle\frac{271}{44800} t^7 v-\textstyle\frac{1}{307200}t^6 s^2+\textstyle\frac{9061}{268800}t^6 u_1-\textstyle\frac{1}{134400}t^5 s v+\textstyle\frac{423}{2800}t^5 w_1 \\
&\quad+\textstyle\frac{1}{6451200}t^4 s^3-\textstyle\frac{1}{89600}t^4 s u_1+\textstyle\frac{281}{560}t^4 x_1,
\\
f_{34}&=s^5+\textstyle\frac{4699}{28}t^{10}+\textstyle\frac{19701}{28}t^8 s-\textstyle\frac{468}{7} t^7 v-\textstyle\frac{593}{4}t^6 s^2-\textstyle\frac{20718}{7} t^6 u_1-\textstyle\frac{10940}{7}t^5 s v-\textstyle\frac{13696}{7}t^5 w_1-\textstyle\frac{13373}{28} t^4 s^3\\
&\quad-\textstyle\frac{2606}{7}t^4 s u_1-\textstyle\frac{1920}{7} t^4 x_1,
\\
f_{35}&=s^3  u_1+\textstyle\frac{14219}{4032}t^{10}+\textstyle\frac{1109}{2240}t^8 s-\textstyle\frac{5077}{140}t^7 v-\textstyle\frac{13757}{320}t^6 s^2-\textstyle\frac{150531}{280}t^6 u_1-\textstyle\frac{4435}{28}t^5 s v-\textstyle\frac{73704}{35}t^5 w_1+\textstyle\frac{5837}{6720}t^4 s^3\\
&\quad-\textstyle\frac{98571}{280}t^4 s u_1-\textstyle\frac{5808}{7}t^4 x_1,
\\
f_{36}&=s^3  u_2+\textstyle\frac{605559}{11200} t^{10}+\textstyle\frac{2365561}{11200}t^8 s+\textstyle\frac{903953}{4200}t^7 v+\textstyle\frac{89257}{4800}t^6 s^2-\textstyle\frac{1275297}{700}t^6 u_1+\textstyle\frac{24121}{40}t^6 u_2+\textstyle\frac{32163}{1400}t^5 s v\\
&\quad-\textstyle\frac{1903824}{175}t^5 w_1+\textstyle\frac{1548}{5}t^5 w_2+4788t^5 w_3-\textstyle\frac{165853}{11200}t^4 s^3+\textstyle\frac{146421}{700}t^4 s u_1-\textstyle\frac{2667}{40}t^4 s u_2-\textstyle\frac{835284}{35}t^4 x_1+\textstyle\frac{3483}{5}t^4 x_2,
\\
f_{37}&=s^2  x_1-\textstyle\frac{214609}{2419200} t^{10}-\textstyle\frac{55729}{89600}t^8 s-\textstyle\frac{64511}{16800}t^7 v+\textstyle\frac{2107}{38400}t^6 s^2-\textstyle\frac{640487}{33600}t^6 u_1+\textstyle\frac{3547}{16800}t^5 s v-\textstyle\frac{25271}{350}t^5 w_1-\textstyle\frac{187}{806400}t^4 s^3\\
&\quad+\textstyle\frac{5787}{11200}t^4 s u_1-\textstyle\frac{11667}{70}t^4 x_1,
\\
f_{38}&=sv  w_2-\textstyle\frac{6926111}{9676800}t^{10}-\textstyle\frac{3964277}{1075200} t^8 s-\textstyle\frac{615323}{33600} t^7 v+\textstyle\frac{73879}{51200}t^6 s^2-\textstyle\frac{9098347}{134400}t^6 u_1-\textstyle\frac{679}{640}t^6 u_2+\textstyle\frac{25661}{5600}t^5 s v\\
&\quad-\textstyle\frac{861193}{4200}t^5 w_1+\textstyle\frac{17}{4}t^5 w_2-\textstyle\frac{243}{20}t^5 w_3-\textstyle\frac{87677}{3225600}t^4 s^3+\textstyle\frac{1113461}{134400}t^4 s u_1-\textstyle\frac{19}{640} t^4 s u_2-\textstyle\frac{122803}{280} t^4 x_1-\textstyle\frac{99}{80}t^4 x_2,
\\
f_{39}&=sz+\textstyle\frac{1987}{19353600} t^{10}+\textstyle\frac{547}{716800}t^8 s+\textstyle\frac{653}{134400}t^7 v-\textstyle\frac{1}{307200} t^6 s^2+\textstyle\frac{6821}{268800}t^6 u_1-\textstyle\frac{1}{134400} t^5 s v+\textstyle\frac{283}{2800}t^5 w_1\\
&\quad+\textstyle\frac{1}{6451200} t^4 s^3-\textstyle\frac{1}{89600}t^4 s u_1+\textstyle\frac{141}{560}t^4 x_1,
\\
f_{40}&=vy+\textstyle\frac{1987}{3225600} t^{10}+\textstyle\frac{1641}{358400}t^8s+\textstyle\frac{653}{22400}t^7v-\textstyle\frac{1}{51200}t^6s^2+\textstyle\frac{6821}{44800}t^6u_1-\textstyle\frac{1}{22400}t^5sv+\textstyle\frac{849}{1400}t^5w_1+\textstyle\frac{1}{1075200}t^4s^3\\
&\quad-\textstyle\frac{3}{44800}t^4su_1+\textstyle\frac{423}{280}t^4x_1,
\\
f_{41}&=u_1x_1+\textstyle\frac{1987}{1290240}t^{10}+\textstyle\frac{1641}{143360}t^8s+\textstyle\frac{653}{8960}t^7v-\textstyle\frac{1}{20480}t^6s^2+\textstyle\frac{6821}{17920}t^6u_1-\textstyle\frac{1}{8960}t^5sv+\textstyle\frac{849}{560}t^5w_1+\textstyle\frac{1}{430080}t^4s^3\\
&\quad-\textstyle\frac{3}{17920}t^4su_1+\textstyle\frac{423}{112}t^4x_1,
\\
f_{42}&=u_1x_2-\textstyle\frac{24797}{28800}t^{10}-\textstyle\frac{142987}{67200}t^8s+\textstyle\frac{11981}{1800}t^7v+\textstyle\frac{16741}{28800}t^6s^2+\textstyle\frac{799313}{8400}t^6u_1-\textstyle\frac{213}{20}t^6u_2+\textstyle\frac{28171}{12600}t^5sv+\textstyle\frac{454733}{1050}t^5w_1\\
&\quad-\textstyle\frac{88}{15}t^5w_2-66t^5w_3-\textstyle\frac{509}{67200}t^4s^3+\textstyle\frac{29501}{8400}t^4su_1+\textstyle\frac{4}{15}t^4su_2+\textstyle\frac{178733}{210}t^4x_1-\textstyle\frac{737}{60}t^4x_2,
\\
f_{43}&=u_2x_1+\textstyle\frac{1472719}{9676800}t^{10}+\textstyle\frac{170223}{358400}t^8s-\textstyle\frac{12263}{67200}t^7v-\textstyle\frac{7109}{153600}t^6s^2-\textstyle\frac{200273}{19200}t^6u_1+\textstyle\frac{131}{80}t^6u_2-\textstyle\frac{649}{3200}t^5sv-\textstyle\frac{30877}{600}t^5w_1\\
&\quad+\textstyle\frac{4}{5}t^5w_2+\textstyle\frac{49}{5}t^5w_3+\textstyle\frac{49}{153600}t^4s^3-\textstyle\frac{5161}{19200}t^4su_1-\textstyle\frac{1}{80}t^4su_2-\textstyle\frac{4049}{40}t^4x_1+\textstyle\frac{19}{10}t^4x_2,
\\
f_{44}&=u_2x_2+\textstyle\frac{14018959}{1814400}t^{10}+\textstyle\frac{740737}{28800}t^8s+\textstyle\frac{1200169}{151200}t^7v-\textstyle\frac{153107}{86400}t^6s^2-\textstyle\frac{20598827}{50400}t^6u_1+\textstyle\frac{129703}{1440}t^6u_2-\textstyle\frac{1101323}{151200}t^5sv\\
&\quad-\textstyle\frac{1075106}{525}t^5w_1+\textstyle\frac{667}{15}t^5w_2+635 t^5w_3+\textstyle\frac{68219}{604800}t^4s^3-\textstyle\frac{131573}{16800}t^4su_1-\textstyle\frac{2141}{1440}t^4su_2-\textstyle\frac{131847}{35}t^4x_1+\textstyle\frac{1961}{20}t^4x_2,
\\
f_{45}&=w_1^2+\textstyle\frac{1987}{967680}t^{10}+\textstyle\frac{547}{35840}t^8s+\textstyle\frac{653}{6720}t^7v-\textstyle\frac{1}{15360}t^6s^2+\textstyle\frac{6821}{13440}t^6u_1-\textstyle\frac{1}{6720}t^5sv+\textstyle\frac{283}{140}t^5w_1+\textstyle\frac{1}{322560}t^4s^3\\
&\quad-\textstyle\frac{1}{4480}t^4su_1+\textstyle\frac{141}{28}t^4x_1,
\\
f_{46}&=w_1w_2-\textstyle\frac{667913}{4300800}t^{10}-\textstyle\frac{661729}{1433600}t^8s+\textstyle\frac{214337}{806400}t^7v+\textstyle\frac{246481}{1843200}t^6s^2+\textstyle\frac{6013373}{537600}t^6u_1-\textstyle\frac{7411}{3840}t^6u_2+\textstyle\frac{111437}{268800}t^5sv\\
&\quad+\textstyle\frac{2691721}{50400}t^5w_1-\textstyle\frac{41}{60}t^5w_2-\textstyle\frac{629}{40}t^5w_3-\textstyle\frac{17849}{4300800}t^4s^3+\textstyle\frac{813583}{1612800}t^4su_1+\textstyle\frac{187}{1280}t^4su_2+\textstyle\frac{337997}{3360}t^4x_1-\textstyle\frac{333}{160}t^4x_2,
\\
f_{47}&=w_1w_3-\textstyle\frac{2481497}{16588800}t^{10}-\textstyle\frac{6598189}{12902400}t^8s-\textstyle\frac{262457}{806400}t^7v+\textstyle\frac{10301}{204800}t^6s^2+\textstyle\frac{11184631}{1612800}t^6u_1-\textstyle\frac{935}{576}t^6u_2+\textstyle\frac{492167}{2419200}t^5sv\\
&\quad+\textstyle\frac{1828759}{50400}t^5w_1-\textstyle\frac{77}{90}t^5w_2-\textstyle\frac{1369}{120}t^5w_3-\textstyle\frac{12389}{38707200}t^4s^3+\textstyle\frac{473407}{1612800}t^4su_1+\textstyle\frac{11}{960}t^4su_2+\textstyle\frac{224323}{3360}t^4x_1-\textstyle\frac{209}{120}t^4x_2,
\\
f_{48}&=w_2^2+\textstyle\frac{20627}{19353600}t^{10}+\textstyle\frac{202963}{716800}t^8s+\textstyle\frac{1116131}{403200}t^7v+\textstyle\frac{98093}{921600}t^6s^2+\textstyle\frac{4410269}{268800}t^6u_1-\textstyle\frac{737}{960}t^6u_2+\textstyle\frac{45331}{134400}t^5sv\\
&\quad+\textstyle\frac{489001}{8400}t^5w_1-\textstyle\frac{11}{30}t^5w_2-\textstyle\frac{11}{2}t^5w_3-\textstyle\frac{4591}{6451200}t^4s^3+\textstyle\frac{143413}{268800}t^4su_1+\textstyle\frac{11}{960}t^4su_2+\textstyle\frac{55861}{560}t^4x_1-\textstyle\frac{33}{40}t^4x_2,
\\
f_{49}&=w_2w_3-\textstyle\frac{2921}{1105920}t^{10}-\textstyle\frac{14461}{860160}t^8s-\textstyle\frac{14419}{161280}t^7v+\textstyle\frac{7}{40960}t^6s^2-\textstyle\frac{5743}{15360}t^6u_1+\textstyle\frac{1}{2560}t^5sv-\textstyle\frac{179}{160}t^5w_1\\
&\quad-\textstyle\frac{1}{122880}t^4s^3+\textstyle\frac{3}{5120}t^4su_1-\textstyle\frac{179}{96}t^4x_1,
\\
f_{50}&=w_3^2+\textstyle\frac{5159}{8294400}t^{10}+\textstyle\frac{3157}{921600}t^8s+\textstyle\frac{1001}{57600}t^7v-\textstyle\frac{539}{921600}t^6s^2+\textstyle\frac{7777}{115200}t^6u_1-\textstyle\frac{77}{57600}t^5sv+\textstyle\frac{77}{400}t^5w_1+\textstyle\frac{77}{2764800}t^4s^3\\
&\quad-\textstyle\frac{77}{38400}t^4su_1+\textstyle\frac{77}{240}t^4x_1.
\end{align*}
}

\noindent\textbf{Degree $11$.}
{\footnotesize
\begin{align*}
f_{51}&=t^5 s u_2+\textstyle\frac{7439}{48510} t^{11} +\textstyle\frac{10789}{16170} t^9 s+\textstyle\frac{3376}{2695} t^8 v+\textstyle\frac{18}{49} t^7 s^2-\textstyle\frac{34}{77} t^7 u_1+\textstyle\frac37  t^7 u_2 +\textstyle\frac{82}{105} t^6 s v-\textstyle\frac{96}{11} t^6 w_1-\textstyle\frac{96}{11} t^6 w_3,
\\
f_{52}&=vz,
\\
f_{53}&=u_1y,
\\
f_{54}&=u_2y+\textstyle\frac{4979}{110602800} t^{11}+\textstyle\frac{52627}{13825350} t^9 s+\textstyle\frac{25226}{768075} t^8 v-\textstyle\frac{29}{614460} t^7 s^2+\textstyle\frac{13933}{87780} t^7 u_1-\textstyle\frac{29}{438900} t^6 s v+\textstyle\frac{2447}{6270} t^6 w_1+\textstyle\frac16 t^6 w_3,
\\
f_{55}&=w_1x_1,
\\
f_{56}&=w_1x_2-\textstyle\frac{15581}{9216900} t^{11}-\textstyle\frac{1757219}{27650700} t^9 s-\textstyle\frac{1294471}{2513700} t^8 v-\textstyle\frac{6035}{1106028} t^7 s^2-\textstyle\frac{318643}{131670} t^7 u_1+\textstyle\frac{4}{105} t^7 u_2-\textstyle\frac{41879}{3950100} t^6 s v\\
&\quad-\textstyle\frac{17912}{3135} t^6 w_1+\textstyle\frac 1{45}t^6 w_2-\textstyle\frac{262}{165} t^6 w_3,
\\
f_{57}&=w_2x_1+\textstyle\frac{138679}{126403200} t^{11}-\textstyle\frac{82583}{126403200} t^9 s-\textstyle\frac{648539}{15800400} t^8 v-\textstyle\frac{367}{790020} t^7 s^2-\textstyle\frac{137509}{526680} t^7 u_1+\textstyle\frac 1{30}t^7 u_2+\textstyle\frac{5633}{4514400} t^6 s v\\
&\quad-\textstyle\frac{8137}{12540} t^6 w_1-\textstyle\frac{37}{360} t^6 w_2+\textstyle\frac{98}{165} t^6 w_3,
\\
f_{58}&=w_2x_2+\textstyle\frac{361001}{13825350} t^{11}+\textstyle\frac{881611}{4608450} t^9 s+\textstyle\frac{617984}{768075} t^8 v+\textstyle\frac{4376}{51205} t^7 s^2+\textstyle\frac{659159}{263340} t^7 u_1+\textstyle\frac{31}{252} t^7 u_2+\textstyle\frac{693521}{3950100} t^6 s v\\
&\quad+\textstyle\frac{11612}{3135} t^6 w_1+\textstyle\frac{17}{45} t^6 w_2+\textstyle\frac{22}{15} t^6 w_3,
\\
f_{59}&=w_3x_1-\textstyle\frac{19129}{80438400} t^{11}-\textstyle\frac{27149}{8937600} t^9 s-\textstyle\frac{1331623}{60328800} t^8 v-\textstyle\frac{124}{377055} t^7 s^2-\textstyle\frac{13759}{143640} t^7 u_1+\textstyle\frac{11}{630} t^7 u_2+\textstyle\frac{737}{2154600} t^6 s v\\
&\quad-\textstyle\frac{71}{380} t^6 w_1+\textstyle\frac{11}{540} t^6 w_2+\textstyle\frac{91}{360} t^6 w_3,
\\
f_{60}&=w_3x_2+\textstyle\frac{746203}{15800400} t^{11}+\textstyle\frac{4453159}{15800400} t^9 s+\textstyle\frac{3331721}{2633400} t^8 v+\textstyle\frac{1493}{87780} t^7 s^2+\textstyle\frac{143107}{35112} t^7 u_1+\textstyle\frac1{40}t^7 u_2+\textstyle\frac{2971}{94050} t^6 s v\\
&\quad+\textstyle\frac{2023}{285} t^6 w_1+\textstyle\frac{21}{55} t^6 w_3.
\end{align*}
}

\noindent\textbf{Degree $12$.}
{\footnotesize
\begin{align*}
f_{61}&=u_1z,
\\
f_{62}&=u_2z+\textstyle\frac{15671}{55883520} t^{12}+\textstyle\frac{863}{846720} t^{10} s+\textstyle\frac{1493}{635040} t^9 v-\textstyle\frac{263}{5080320} t^8 s^2+\textstyle\frac{31}{15120} t^8 u_1+\textstyle\frac{11}{6048} t^8 u_2,
\\
f_{63}&=w_1y,
\\
f_{64}&=w_2y+\textstyle\frac{31403}{111767040} t^{12}+\textstyle\frac{10663}{8467200} t^{10} s+\textstyle\frac{1033}{254016} t^9 v-\textstyle\frac{299}{10160640} t^8 s^2+\textstyle\frac{31}{4320} t^8 u_1+\textstyle\frac{11}{12096} t^8 u_2,
\\
f_{65}&=w_3y+\textstyle\frac{40631}{243855360} t^{12}+\textstyle\frac{136279}{203212800} t^{10} s+\textstyle\frac{5687}{3048192} t^9 v-\textstyle\frac{6017}{243855360} t^8 s^2+\textstyle\frac{3817}{1451520} t^8 u_1+\textstyle\frac{121}{145152} t^8 u_2,
\\
f_{66}&=x_1^2,
\\
f_{67}&=x_1x_2-\textstyle\frac{15671}{2794176} t^{12}-\textstyle\frac{863}{42336} t^{10} s-\textstyle\frac{1493}{31752} t^9 v+\textstyle\frac{263}{254016} t^8 s^2-\textstyle\frac{31}{756} t^8 u_1-\textstyle\frac{55}{1512} t^8 u_2,
\\
f_{68}&=x_2^2+\textstyle\frac{815827}{3492720} t^{12}+\textstyle\frac{11353}{10584} t^{10} s+\textstyle\frac{13688}{3969} t^9 v-\textstyle\frac{1943}{317520} t^8 s^2+\textstyle\frac{1609}{270} t^8 u_1+\textstyle\frac{356}{945} t^8 u_2.
\end{align*}
}

\noindent\textbf{Degree $13$.}
{\footnotesize
\begin{align*}
f_{69}&=w_1z,
\\
f_{70}&=w_2z-\textstyle\frac{101}{61261200} t^{13}-\textstyle\frac{23}{3534300} t^{11} s-\textstyle\frac1{71400}t^{10} v,
\\
f_{71}&=w_3z-\textstyle\frac{101}{267321600} t^{13}-\textstyle\frac{23}{15422400} t^{11} s-\textstyle\frac{11}{3427200} t^{10} v,
\\
f_{72}&=x_1y,
\\
f_{73}&=x_2y.
\end{align*}
}

\noindent\textbf{Degree $14$.}
{\footnotesize
\begin{align*}
f_{74}&=x_1z,
\\
f_{75}&=x_2z,
\\
f_{76}&=y^2.
\end{align*}
}

\noindent\textbf{Degree $15$.}
{\footnotesize
\begin{align*}
f_{77}&=yz.
\end{align*}
}

\noindent\textbf{Degree $16$.}
{\footnotesize
\begin{align*}
f_{78}&=z^2.
\end{align*}
}


\section{The principal kinematic formula}

\label{app:pkf}

Keeping the notation from Section \ref{s:pkf} and denoting $\mu\odot\nu=\frac12\left(\mu\otimes\nu+\nu\otimes \mu\right)$, we have
{\footnotesize
\begin{align*}
&51609600\, \pi^8\sum_{i,j=1}^{27}(M_{8}^{-1})_{i,j}\psi^{(i)}_{8}\otimes\psi^{(j)}_{8}
\\
&=\textstyle\frac{101042723}{8448} \hat t^8 \odot  \hat t^8 + \textstyle\frac{41912485}{528} \hat t^8 \odot  \hat t^6 \hat s + \textstyle\frac{131398835}{3168} \hat t^8 \odot  \hat t^5 \hat v  -\textstyle\frac{10178935}{792} \hat t^8 \odot  \hat t^4 \hat s^2  -\textstyle\frac{600048805}{528} \hat t^8 \odot  \hat t^4 \hat u_1 
\\
&+ \textstyle\frac{4538995}{16} \hat t^8 \odot  \hat t^4 \hat u_2  -\textstyle\frac{37975765}{792} \hat t^8 \odot  \hat t^3 \hat s \hat v -\textstyle\frac{189207035}{33} \hat t^8 \odot  \hat t^3 \hat w_1 + \textstyle\frac{5826055}{24} \hat t^8 \odot  \hat t^3 \hat w_2 + \textstyle\frac{7976745}{4} \hat t^8 \odot  \hat t^3 \hat w_3
 \\
&+ \textstyle\frac{133415}{352} \hat t^8 \odot  \hat t^2 \hat s^3 -\textstyle\frac{10151365}{132} \hat t^8 \odot  \hat t^2 \hat s \hat u_1+ \textstyle\frac{8215}{8} \hat t^8 \odot  \hat t^2 \hat s \hat u_2  -\textstyle\frac{112512535}{11} \hat t^8 \odot  \hat t^2 \hat x_1+\textstyle\frac{2265005}{8} \hat t^8 \odot  \hat t^2 \hat x_2 
 \\
&  + \textstyle\frac{2835155}{3168} \hat t^8 \odot \hat t \hat s^2 \hat v -\textstyle\frac{79005}{44} \hat t^8 \odot \hat t \hat s \hat w_1 + \textstyle\frac{397455}{8} \hat t^8 \odot  \hat t \hat s \hat w_2 -\textstyle\frac{360555}{11} \hat t^8 \odot \hat  t \hat s \hat w_3  +\textstyle\frac{5465925}{22} \hat t^8 \odot  \hat t \hat y  -\textstyle\frac{190945}{12672} \hat t^8 \odot  \hat s^4
\\
& + \textstyle\frac{163545}{176} \hat t^8 \odot  \hat s^2  \hat u_1 + \textstyle\frac{25135}{48} \hat t^8 \odot  \hat s^2  \hat u_2 +\textstyle\frac{882875}{11} \hat t^8 \odot \hat s  \hat x_1 -9900 \hat t^8 \odot  \hat v  \hat w_2  -\textstyle\frac{661775}{22} \hat t^8 \odot \hat v  \hat w_3 + \textstyle\frac{20210400}{11} \hat t^8 \odot  \hat z 
\\
&+\textstyle\frac{17385985}{132} \hat t^6 \hat s \odot  \hat t^6 \hat s + \textstyle\frac{436979155}{3168} \hat t^6 \hat s \odot  \hat t^5 \hat v  -\textstyle\frac{270203215}{6336} \hat t^6 \hat s \odot  \hat t^4 \hat s^2  -\textstyle\frac{662726705}{176} \hat t^6 \hat s \odot  \hat t^4 \hat u_1 + \textstyle\frac{45172565}{48} \hat t^6 \hat s \odot  \hat t^4 \hat u_2 
\\
& -\textstyle\frac{31560245}{198} \hat t^6 \hat s \odot  \hat t^3 \hat s \hat v -\textstyle\frac{7507001455}{396} \hat t^6 \hat s \odot  \hat t^3 \hat w_1 + \textstyle\frac{6451745}{8} \hat t^6 \hat s \odot  \hat t^3 \hat w_2 + \textstyle\frac{72690555}{11} \hat t^6 \hat s \odot  \hat t^3 \hat w_3 + \textstyle\frac{999775}{792} \hat t^6 \hat s \odot  \hat t^2 \hat s^3 
\\
& -\textstyle\frac{101244185}{396} \hat t^6 \hat s \odot  \hat t^2 \hat s \hat u_1 + \textstyle\frac{27955}{8} \hat t^6 \hat s \odot  \hat t^2 \hat s \hat u_2 -\textstyle\frac{2196730045}{66} \hat t^6 \hat s \odot  \hat t^2 \hat x_1 + \textstyle\frac{7511925}{8} \hat t^6 \hat s \odot  \hat t^2 \hat x_2 + \textstyle\frac{9489955}{3168} \hat t^6 \hat s \odot \hat  t \hat s^2 \hat v 
\\
&+ \textstyle\frac{37990}{9} \hat t^6 \hat s \odot\hat   t \hat s \hat w_1 + \textstyle\frac{4010285}{24} \hat t^6 \hat s \odot \hat  t \hat s \hat w_2  -\textstyle\frac{5102905}{44} \hat t^6 \hat s \odot \hat  t \hat s \hat w_3 + \textstyle\frac{113103315}{22} \hat t^6 \hat s \odot\hat   t \hat y  -\textstyle\frac{315575}{6336} \hat t^6 \hat s \odot  \hat s^4 
\\
&+ \textstyle\frac{5247175}{1584} \hat t^6 \hat s \odot  \hat s^2  \hat u_1 + \textstyle\frac{83515}{48} \hat t^6 \hat s \odot  \hat s^2  \hat u_2 + \textstyle\frac{3741990}{11} \hat t^6 \hat s \odot \hat s  \hat x_1 -\textstyle\frac{93575}{3} \hat t^6 \hat s \odot  \hat v  \hat w_2  -\textstyle\frac{2197925}{22} \hat t^6 \hat s \odot \hat v  \hat w_3
\\
&+ \textstyle\frac{228977280}{11} \hat t^6 \hat s \odot  \hat z +\textstyle\frac{372273293}{9504} \hat t^5 \hat v \odot  \hat t^5 \hat v  -\textstyle\frac{19067905}{864} \hat t^5 \hat v \odot  \hat t^4 \hat s^2 -\textstyle\frac{493836119}{264} \hat t^5 \hat v \odot  \hat t^4 \hat u_1 + \textstyle\frac{35130295}{72} \hat t^5 \hat v \odot  \hat t^4 \hat u_2  
\\
&-\textstyle\frac{66023681}{792} \hat t^5 \hat v \odot  \hat t^3 \hat s \hat v -\textstyle\frac{10474501303}{1188} \hat t^5 \hat v \odot  \hat t^3 \hat w_1 + \textstyle\frac{5095775}{12} \hat t^5 \hat v \odot  \hat t^3 \hat w_2 + \textstyle\frac{148030505}{44} \hat t^5 \hat v \odot  \hat t^3 \hat w_3 + \textstyle\frac{26071}{36} \hat t^5 \hat v \odot  \hat t^2 \hat s^3 
\\
& -\textstyle\frac{139961069}{1188} \hat t^5 \hat v \odot  \hat t^2 \hat s \hat u_1 +\textstyle\frac{24175}{12} \hat t^5 \hat v \odot  \hat t^2 \hat s \hat u_2  -\textstyle\frac{1569950995}{198} \hat t^5 \hat v \odot  \hat t^2 \hat x_1 + \textstyle\frac{1943975}{4} \hat t^5 \hat v \odot  \hat t^2 \hat x_2 + \textstyle\frac{358893}{176} \hat t^5 \hat v \odot \hat  t \hat s^2 \hat v 
\\
&+ \textstyle\frac{236431805}{1188} \hat t^5 \hat v \odot \hat  t \hat s \hat w_1 + \textstyle\frac{416265}{4} \hat t^5 \hat v \odot \hat  t \hat s \hat w_2  -\textstyle\frac{5516905}{44} \hat t^5 \hat v \odot \hat  t \hat s \hat w_3 + \textstyle\frac{646168905}{11} \hat t^5 \hat v \odot \hat  t \hat y  -\textstyle\frac{195559}{9504} \hat t^5 \hat v \odot  \hat s^4 
\\
&+ \textstyle\frac{11259665}{2376} \hat t^5 \hat v \odot  \hat s^2  \hat u_1 + \textstyle\frac{65105}{72} \hat t^5 \hat v \odot  \hat s^2  \hat u_2 + \textstyle\frac{12228860}{11} \hat t^5 \hat v \odot \hat s  \hat x_1 -\textstyle\frac{9665}{9} \hat t^5 \hat v \odot  \hat v  \hat w_2 -\textstyle\frac{1898395}{33} \hat t^5 \hat v \odot \hat v  \hat w_3
\\
&+ \textstyle\frac{2034789120}{11} \hat t^5 \hat v \odot  \hat z +\textstyle\frac{132412835}{38016} \hat t^4 \hat s^2 \odot  \hat t^4 \hat s^2 + \textstyle\frac{108459275}{176} \hat t^4 \hat s^2 \odot  \hat t^4 \hat u_1-\textstyle\frac{21953365}{144} \hat t^4 \hat s^2 \odot  \hat t^4 \hat u_2+ \textstyle\frac{40775785}{1584} \hat t^4 \hat s^2 \odot  \hat t^3 \hat s \hat v
\\
&  +\textstyle\frac{7474927855}{2376} \hat t^4 \hat s^2 \odot  \hat t^3 \hat w_1 -\textstyle\frac{1031415}{8} \hat t^4 \hat s^2 \odot  \hat t^3 \hat w_2  -\textstyle\frac{94810205}{88} \hat t^4 \hat s^2 \odot  \hat t^3 \hat w_3 -\textstyle\frac{313285}{1584} \hat t^4 \hat s^2 \odot  \hat t^2 \hat s^3+ \textstyle\frac{98639645}{2376} \hat t^4 \hat s^2 \odot  \hat t^2 \hat s \hat u_1
\\
&  -\textstyle\frac{6265}{24} \hat t^4 \hat s^2 \odot  \hat t^2 \hat s \hat u_2+\textstyle\frac{2395429675}{396} \hat t^4 \hat s^2 \odot  \hat t^2 \hat x_1 -\textstyle\frac{1223425}{8} \hat t^4 \hat s^2 \odot  \hat t^2 \hat x_2   -\textstyle\frac{1392715}{3168} \hat t^4 \hat s^2 \odot \hat  t \hat s^2 \hat v  +\textstyle\frac{2469905}{216} \hat t^4 \hat s^2 \odot \hat  t \hat s \hat w_1
\\
& -\textstyle\frac{194775}{8} \hat t^4 \hat s^2 \odot \hat  t \hat s \hat w_2  +\textstyle\frac{1053325}{88} \hat t^4 \hat s^2 \odot  \hat t \hat s \hat w_3 + \textstyle\frac{69259275}{22} \hat t^4 \hat s^2 \odot \hat  t \hat y + \textstyle\frac{76115}{9504} \hat t^4 \hat s^2 \odot  \hat s^4  -\textstyle\frac{1319135}{4752} \hat t^4 \hat s^2 \odot  \hat s^2  \hat u_1 
\\
& -\textstyle\frac{38675}{144} \hat t^4 \hat s^2 \odot  \hat s^2  \hat u_2+ \textstyle\frac{172200}{11} \hat t^4 \hat s^2 \odot \hat s  \hat x_1 + \textstyle\frac{106925}{18} \hat t^4 \hat s^2 \odot  \hat v  \hat w_2 + \textstyle\frac{1117375}{66} \hat t^4 \hat s^2 \odot \hat v  \hat w_3+ \textstyle\frac{97070400}{11} \hat t^4 \hat s^2 \odot  \hat z 
\\
&+\textstyle\frac{7302883913}{264} \hat t^4 \hat u_1 \odot  \hat t^4 \hat u_1  -\textstyle\frac{162129715}{12} \hat t^4 \hat u_1 \odot  \hat t^4 \hat u_2 + \textstyle\frac{305659649}{132} \hat t^4 \hat u_1 \odot  \hat t^3 \hat s \hat v+ \textstyle\frac{56608773187}{198} \hat t^4 \hat u_1 \odot  \hat t^3 \hat w_1 
\\
&-\textstyle\frac{69294665}{6} \hat t^4 \hat u_1 \odot  \hat t^3 \hat w_2  -\textstyle\frac{190666245}{2} \hat t^4 \hat u_1 \odot  \hat t^3 \hat w_3  -\textstyle\frac{3367217}{198} \hat t^4 \hat u_1 \odot  \hat t^2 \hat s^3 + \textstyle\frac{796349201}{198} \hat t^4 \hat u_1 \odot  \hat t^2 \hat s \hat u_1  
\\
&-\textstyle\frac{295565}{6} \hat t^4 \hat u_1 \odot  \hat t^2 \hat s \hat u_2 + \textstyle\frac{19501574935}{33} \hat t^4 \hat u_1 \odot  \hat t^2 \hat x_1 -\textstyle\frac{26975625}{2} \hat t^4 \hat u_1 \odot  \hat t^2 \hat x_2  -\textstyle\frac{28250459}{792} \hat t^4 \hat u_1 \odot \hat t \hat s^2 \hat v
\\
& + \textstyle\frac{518106715}{198} \hat t^4 \hat u_1 \odot \hat  t \hat s \hat w_1 -\textstyle\frac{4543815}{2} \hat t^4 \hat u_1 \odot \hat  t \hat s \hat w_2 + \textstyle\frac{2273065}{2} \hat t^4 \hat u_1 \odot \hat  t \hat s \hat w_3 + \textstyle\frac{6008331870}{11} \hat t^4 \hat u_1 \odot \hat  t \hat y + \textstyle\frac{413467}{528} \hat t^4 \hat u_1 \odot  \hat s^4  
\\
&-\textstyle\frac{5377685}{396} \hat t^4 \hat u_1 \odot  \hat s^2  \hat u_1  -\textstyle\frac{99455}{4} \hat t^4 \hat u_1 \odot  \hat s^2  \hat u_2  +\textstyle\frac{57302840}{11} \hat t^4 \hat u_1 \odot \hat s  \hat x_1+ \textstyle\frac{1667170}{3} \hat t^4 \hat u_1 \odot  \hat v  \hat w_2+ \textstyle\frac{15188410}{11} \hat t^4 \hat u_1 \odot \hat v  \hat w_3
\\
&+ \textstyle\frac{16174529280}{11} \hat t^4 \hat u_1 \odot  \hat z +\textstyle\frac{40431875}{24} \hat t^4 \hat u_2 \odot  \hat t^4 \hat u_2 -\textstyle\frac{6793615}{12} \hat t^4 \hat u_2 \odot  \hat t^3 \hat s \hat v -\textstyle\frac{1226983165}{18} \hat t^4 \hat u_2 \odot  \hat t^3 \hat w_1
\\
&+ \textstyle\frac{17297225}{6} \hat t^4 \hat u_2 \odot  \hat t^3 \hat w_2 + \textstyle\frac{47471025}{2} \hat t^4 \hat u_2 \odot  \hat t^3 \hat w_3+ \textstyle\frac{40190}{9} \hat t^4 \hat u_2 \odot  \hat t^2 \hat s^3  -\textstyle\frac{16143095}{18} \hat t^4 \hat u_2 \odot  \hat t^2 \hat s \hat u_1 + \textstyle\frac{74525}{6} \hat t^4 \hat u_2 \odot  \hat t^2 \hat s \hat u_2
\\
&  -\textstyle\frac{366421225}{3} \hat t^4 \hat u_2 \odot  \hat t^2 \hat x_1 + \textstyle\frac{6726225}{2} \hat t^4 \hat u_2 \odot  \hat t^2 \hat x_2 + \textstyle\frac{761065}{72} \hat t^4 \hat u_2 \odot \hat  t \hat s^2 \hat v + \textstyle\frac{307835}{18} \hat t^4 \hat u_2 \odot\hat   t \hat s \hat w_1 + \textstyle\frac{1149775}{2} \hat t^4 \hat u_2 \odot \hat  t \hat s \hat w_2  
\\
&-\textstyle\frac{655525}{2} \hat t^4 \hat u_2 \odot \hat  t \hat s \hat w_3 -2063250 \hat t^4 \hat u_2 \odot \hat  t \hat y -\textstyle\frac{25745}{144} \hat t^4 \hat u_2 \odot  \hat s^4 + \textstyle\frac{388055}{36} \hat t^4 \hat u_2 \odot  \hat s^2  \hat u_1 + \textstyle\frac{74525}{12} \hat t^4 \hat u_2 \odot  \hat s^2  \hat u_2 
\\
&+ 917000 \hat t^4 \hat u_2 \odot \hat s  \hat x_1 -\textstyle\frac{395450}{3} \hat t^4 \hat u_2 \odot  \hat v  \hat w_2  -355350 \hat t^4 \hat u_2 \odot \hat v  \hat w_3+\textstyle\frac{14256331}{297} \hat t^3 \hat s \hat v \odot  \hat t^3 \hat s \hat v+ \textstyle\frac{108352963}{9} \hat t^3 \hat s \hat v \odot  \hat t^3 \hat w_1 
\\
& -\textstyle\frac{2819225}{6} \hat t^3 \hat s \hat v \odot  \hat t^3 \hat w_2  -\textstyle\frac{132191920}{33} \hat t^3 \hat s \hat v \odot  \hat t^3 \hat w_3 -\textstyle\frac{222529}{297} \hat t^3 \hat s \hat v \odot  \hat t^2 \hat s^3 + \textstyle\frac{16188959}{99} \hat t^3 \hat s \hat v \odot  \hat t^2 \hat s \hat u_1  -\textstyle\frac{1675}{6} \hat t^3 \hat s \hat v \odot  \hat t^2 \hat s \hat u_2  
\\
&+\textstyle\frac{275442210}{11} \hat t^3 \hat s \hat v \odot  \hat t^2 \hat x_1  -\textstyle\frac{3412525}{6} \hat t^3 \hat s \hat v \odot  \hat t^2 \hat x_2  -\textstyle\frac{3443051}{2376} \hat t^3 \hat s \hat v \odot \hat  t \hat s^2 \hat v  +\textstyle\frac{10790080}{99} \hat t^3 \hat s \hat v \odot \hat  t \hat s \hat w_1  -\textstyle\frac{1582885}{18} \hat t^3 \hat s \hat v \odot \hat  t \hat s \hat w_2 
\\
&+ \textstyle\frac{1764145}{33} \hat t^3 \hat s \hat v \odot\hat   t \hat s \hat w_3 + \textstyle\frac{254846310}{11} \hat t^3 \hat s \hat v \odot \hat  t \hat y + \textstyle\frac{52823}{1584} \hat t^3 \hat s \hat v \odot  \hat s^4  -\textstyle\frac{127885}{396} \hat t^3 \hat s \hat v \odot  \hat s^2  \hat u_1  -\textstyle\frac{12385}{12} \hat t^3 \hat s \hat v \odot  \hat s^2  \hat u_2  
\\
&+\textstyle\frac{7738360}{33} \hat t^3 \hat s \hat v \odot  \hat s  \hat x_1+ \textstyle\frac{68660}{3} \hat t^3 \hat s \hat v \odot  \hat v  \hat w_2 + \textstyle\frac{2074690}{33} \hat t^3 \hat s \hat v \odot \hat v  \hat w_3+ \textstyle\frac{629475840}{11} \hat t^3 \hat s \hat v \odot  \hat z +\textstyle\frac{222457237834}{297} \hat t^3 \hat w_1 \odot  \hat t^3 \hat w_1  
\\
&-\textstyle\frac{525452755}{9} \hat t^3 \hat w_1 \odot  \hat t^3 \hat w_2  -\textstyle\frac{5283356470}{11} \hat t^3 \hat w_1 \odot  \hat t^3 \hat w_3 -\textstyle\frac{23603398}{297} \hat t^3 \hat w_1 \odot  \hat t^2 \hat s^3+ \textstyle\frac{6703353544}{297} \hat t^3 \hat w_1 \odot  \hat t^2 \hat s \hat u_1  
\\
&-\textstyle\frac{763135}{3} \hat t^3 \hat w_1 \odot  \hat t^2 \hat s \hat u_2 + \textstyle\frac{327601805680}{99} \hat t^3 \hat w_1 \odot  \hat t^2 \hat x_1  -\textstyle\frac{204106025}{3} \hat t^3 \hat w_1 \odot  \hat t^2 \hat x_2   -\textstyle\frac{58823117}{396} \hat t^3 \hat w_1 \odot \hat  t \hat s^2 \hat v 
\\
&+ \textstyle\frac{6633686990}{297} \hat t^3 \hat w_1 \odot \hat  t \hat s \hat w_1 -\textstyle\frac{105045545}{9} \hat t^3 \hat w_1 \odot\hat   t \hat s \hat w_2 + \textstyle\frac{210155900}{33} \hat t^3 \hat w_1 \odot \hat  t \hat s \hat w_3 + \textstyle\frac{43860881220}{11} \hat t^3 \hat w_1 \odot \hat  t \hat y
\\
&+ \textstyle\frac{9913909}{2376} \hat t^3 \hat w_1 \odot  \hat s^4  -\textstyle\frac{1418575}{594} \hat t^3 \hat w_1 \odot  \hat s^2  \hat u_1  -\textstyle\frac{2256995}{18} \hat t^3 \hat w_1 \odot  \hat s^2  \hat u_2 + \textstyle\frac{1412144720}{33} \hat t^3 \hat w_1 \odot \hat s  \hat x_1+ \textstyle\frac{23670160}{9} \hat t^3 \hat w_1 \odot  \hat v  \hat w_2 
\\
&+\textstyle\frac{230312380}{33} \hat t^3 \hat w_1 \odot \hat v  \hat w_3+\textstyle\frac{104970055680}{11} \hat t^3 \hat w_1 \odot  \hat z +\textstyle\frac{7244875}{6} \hat t^3 \hat w_2 \odot   \hat t^3 \hat w_2 + 20489175 \hat t^3 \hat w_2 \odot   \hat t^3 \hat w_3 + \textstyle\frac{34565}{9} \hat t^3 \hat w_2 \odot   \hat t^2 \hat s^3 
\\
& -\textstyle\frac{6291365}{9} \hat t^3 \hat w_2 \odot   \hat t^2 \hat s \hat u_1 + \textstyle\frac{24475}{3} \hat t^3 \hat w_2 \odot   \hat t^2 \hat s \hat u_2  -\textstyle\frac{314478350}{3} \hat t^3 \hat w_2 \odot   \hat t^2 \hat x_1 + 2880075 \hat t^3 \hat w_2 \odot   \hat t^2 \hat x_2   +\textstyle\frac{323075}{36} \hat t^3 \hat w_2 \odot  \hat  t \hat s^2 \hat v 
\\
&+ \textstyle\frac{1721765}{9} \hat t^3 \hat w_2 \odot   \hat t \hat s \hat w_1 + 437525 \hat t^3 \hat w_2 \odot \hat   t \hat s \hat w_2  -149075 \hat t^3 \hat w_2 \odot \hat   t \hat s \hat w_3  -2223900 \hat t^3 \hat w_2 \odot \hat   t \hat y  -\textstyle\frac{1235}{8} \hat t^3 \hat w_2 \odot   \hat s^4 
\\
&+ \textstyle\frac{165245}{18} \hat t^3 \hat w_2 \odot   \hat s^2  \hat u_1 + \textstyle\frac{10725}{2} \hat t^3 \hat w_2 \odot   \hat s^2  \hat u_2 + 988400 \hat t^3 \hat w_2 \odot  \hat s  \hat x_1 -\textstyle\frac{467500}{3} \hat t^3 \hat w_2 \odot   \hat v  \hat w_2  -300500 \hat t^3 \hat w_2 \odot  \hat v  \hat w_3
\\
&+\textstyle\frac{918940050}{11} \hat t^3 \hat w_3 \odot   \hat t^3 \hat w_3 + \textstyle\frac{93100}{3} \hat t^3 \hat w_3 \odot   \hat t^2 \hat s^3 -\textstyle\frac{70189910}{11} \hat t^3 \hat w_3 \odot   \hat t^2 \hat s \hat u_1 + 98525 \hat t^3 \hat w_3 \odot   \hat t^2 \hat s \hat u_2 
\\
& -860220900 \hat t^3 \hat w_3 \odot   \hat t^2 \hat x_1 + 23687475 \hat t^3 \hat w_3 \odot   \hat t^2 \hat x_2  + \textstyle\frac{9783655}{132} \hat t^3 \hat w_3 \odot  \hat  t \hat s^2 \hat v + \textstyle\frac{103320}{11} \hat t^3 \hat w_3 \odot  \hat  t \hat s \hat w_1
\\
& + 4089225 \hat t^3 \hat w_3 \odot  \hat  t \hat s \hat w_2  -\textstyle\frac{23682750}{11} \hat t^3 \hat w_3 \odot  \hat  t \hat s \hat w_3  -\textstyle\frac{158703300}{11} \hat t^3 \hat w_3 \odot   \hat t \hat y  -\textstyle\frac{330505}{264} \hat t^3 \hat w_3 \odot   \hat s^4 + \textstyle\frac{1660785}{22} \hat t^3 \hat w_3 \odot   \hat s^2  \hat u_1
\\
& + \textstyle\frac{86975}{2} \hat t^3 \hat w_3 \odot   \hat s^2  \hat u_2 + \textstyle\frac{70534800}{11} \hat t^3 \hat w_3 \odot  \hat s  \hat x_1 -919800 \hat t^3 \hat w_3 \odot   \hat v  \hat w_2  -\textstyle\frac{27461700}{11} \hat t^3 \hat w_3 \odot  \hat v  \hat w_3+\textstyle\frac{26489}{4752} \hat t^2 \hat s^3 \odot    \hat t^2 \hat s^3
\\
& -\textstyle\frac{296414}{297} \hat t^2 \hat s^3 \odot    \hat t^2 \hat s \hat u_1 + \textstyle\frac{200}{9} \hat t^2 \hat s^3 \odot    \hat t^2 \hat s \hat u_2  -\textstyle\frac{7172180}{99} \hat t^2 \hat s^3 \odot    \hat t^2 \hat x_1 + \textstyle\frac{13375}{3} \hat t^2 \hat s^3 \odot    \hat t^2 \hat x_2+ \textstyle\frac{36571}{1188} \hat t^2 \hat s^3 \odot \hat   t \hat s^2 \hat v 
\\
&+ \textstyle\frac{657680}{297} \hat t^2 \hat s^3 \odot   \hat t \hat s \hat w_1 + \textstyle\frac{8395}{9} \hat t^2 \hat s^3 \odot  \hat  t \hat s \hat w_2  -\textstyle\frac{34750}{33} \hat t^2 \hat s^3 \odot  \hat  t \hat s \hat w_3 + \textstyle\frac{4812180}{11} \hat t^2 \hat s^3 \odot   \hat t \hat y  -\textstyle\frac{359}{1056} \hat t^2 \hat s^3 \odot    \hat s^4  
\\
&+\textstyle\frac{37385}{594} \hat t^2 \hat s^3 \odot    \hat s^2  \hat u_1 + 15 \hat t^2 \hat s^3 \odot    \hat s^2  \hat u_2 + 8260 \hat t^2 \hat s^3 \odot   \hat s  \hat x_1+ \textstyle\frac{40}{9} \hat t^2 \hat s^3 \odot    \hat v  \hat w_2  -\textstyle\frac{19780}{33} \hat t^2 \hat s^3 \odot   \hat v  \hat w_3+ \textstyle\frac{10926720}{11} \hat t^2 \hat s^3 \odot    \hat z 
\\
&+\textstyle\frac{54046126}{297} \hat t^2 \hat s \hat u_1 \odot    \hat t^2 \hat s \hat u_1 +3905 \hat t^2 \hat s \hat u_1 \odot    \hat t^2 \hat s \hat u_2 +\textstyle\frac{6185875640}{99} \hat t^2 \hat s \hat u_1 \odot    \hat t^2 \hat x_1  -\textstyle\frac{2724175}{3} \hat t^2 \hat s \hat u_1 \odot    \hat t^2 \hat x_2  
\\
&-\textstyle\frac{94211}{396} \hat t^2 \hat s \hat u_1 \odot  \hat  t \hat s^2 \hat v + \textstyle\frac{222580090}{297} \hat t^2 \hat s \hat u_1 \odot   \hat t \hat s \hat w_1  -\textstyle\frac{1083055}{9} \hat t^2 \hat s \hat u_1 \odot   \hat t \hat s \hat w_2 + \textstyle\frac{3730420}{33} \hat t^2 \hat s \hat u_1 \odot  \hat  t \hat s \hat w_3 
\\
&+ \textstyle\frac{1463127060}{11} \hat t^2 \hat s \hat u_1 \odot  \hat  t \hat y + \textstyle\frac{178397}{2376} \hat t^2 \hat s \hat u_1 \odot    \hat s^4 + \textstyle\frac{3233635}{594} \hat t^2 \hat s \hat u_1 \odot    \hat s^2  \hat u_1  -\textstyle\frac{30865}{18} \hat t^2 \hat s \hat u_1 \odot    \hat s^2  \hat u_2 + \textstyle\frac{59595760}{33} \hat t^2 \hat s \hat u_1 \odot   \hat s  \hat x_1
\\
&+ \textstyle\frac{374480}{9} \hat t^2 \hat s \hat u_1 \odot    \hat v  \hat w_2 + \textstyle\frac{3496940}{33} \hat t^2 \hat s \hat u_1 \odot   \hat v  \hat w_3+ \textstyle\frac{3213181440}{11} \hat t^2 \hat s \hat u_1 \odot    \hat z+\textstyle\frac{275}{6} \hat t^2 \hat s \hat u_2 \odot   \hat t^2 \hat s \hat u_2  -473750 \hat t^2 \hat s \hat u_2 \odot   \hat t^2 \hat x_1 
\\
&+ 12375 \hat t^2 \hat s \hat u_2 \odot   \hat t^2 \hat x_2  + \textstyle\frac{725}{12} \hat t^2 \hat s \hat u_2 \odot  \hat t \hat s^2 \hat v + \textstyle\frac{54365}{3} \hat t^2 \hat s \hat u_2 \odot \hat  t \hat s \hat w_1  -\textstyle\frac{8525}{3} \hat t^2 \hat s \hat u_2 \odot  \hat t \hat s \hat w_2 +11275 \hat t^2 \hat s \hat u_2 \odot \hat  t \hat s \hat w_3  
\\
&-56700 \hat t^2 \hat s \hat u_2 \odot  \hat t \hat y  -\textstyle\frac{95}{72} \hat t^2 \hat s \hat u_2 \odot   \hat s^4 + \textstyle\frac{395}{6} \hat t^2 \hat s \hat u_2 \odot   \hat s^2  \hat u_1 + \textstyle\frac{275}{6} \hat t^2 \hat s \hat u_2 \odot   \hat s^2  \hat u_2 + 25200 \hat t^2 \hat s \hat u_2 \odot  \hat s  \hat x_1 -5500 \hat t^2 \hat s \hat u_2 \odot   \hat v  \hat w_2 
\\
&+ 1500 \hat t^2 \hat s \hat u_2 \odot  \hat v  \hat w_3+\textstyle\frac{14477882600}{3} \hat t^2 \hat x_1 \odot   \hat t^2 \hat x_1  -121877050 \hat t^2 \hat x_1 \odot   \hat t^2 \hat x_2 +\textstyle\frac{3463615}{66} \hat t^2 \hat x_1 \odot  \hat t \hat s^2 \hat v 
\\
&+ \textstyle\frac{13225952500}{99} \hat t^2 \hat x_1 \odot \hat  t \hat s \hat w_1 -\textstyle\frac{63199450}{3} \hat t^2 \hat x_1 \odot \hat  t \hat s \hat w_2 + 11526800 \hat t^2 \hat x_1 \odot \hat  t \hat s \hat w_3 + \textstyle\frac{242526403800}{11} \hat t^2 \hat x_1 \odot \hat  t \hat y 
\\
&+ \textstyle\frac{3826345}{396} \hat t^2 \hat x_1 \odot   \hat s^4 + \textstyle\frac{6844375}{9} \hat t^2 \hat x_1 \odot   \hat s^2  \hat u_1  -\textstyle\frac{673175}{3} \hat t^2 \hat x_1 \odot   \hat s^2  \hat u_2 + \textstyle\frac{3086557600}{11} \hat t^2 \hat x_1 \odot  \hat s  \hat x_1+ \textstyle\frac{13851200}{3} \hat t^2 \hat x_1 \odot   \hat v  \hat w_2 
\\
&+12442600 \hat t^2 \hat x_1 \odot  \hat v  \hat w_3+ \textstyle\frac{522924595200}{11} \hat t^2 \hat x_1 \odot   \hat z  +\textstyle\frac{3356925}{2} \hat t^2 \hat x_2 \odot    \hat t^2 \hat x_2  + \textstyle\frac{126575}{12} \hat t^2 \hat x_2 \odot  \hat  t \hat s^2 \hat v -\textstyle\frac{37325}{3} \hat t^2 \hat x_2 \odot \hat   t \hat s \hat w_1
\\
&+ 578875 \hat t^2 \hat x_2 \odot \hat   t \hat s \hat w_2 -325875 \hat t^2 \hat x_2 \odot  \hat  t \hat s \hat w_3 -1984500 \hat t^2 \hat x_2 \odot \hat   t \hat y -\textstyle\frac{1425}{8} \hat t^2 \hat x_2 \odot    \hat s^4+ \textstyle\frac{64525}{6} \hat t^2 \hat x_2 \odot    \hat s^2  \hat u_1 
\\
&+ \textstyle\frac{12375}{2} \hat t^2 \hat x_2 \odot    \hat s^2  \hat u_2 + 882000 \hat t^2 \hat x_2 \odot   \hat s  \hat x_1 -126500 \hat t^2 \hat x_2 \odot    \hat v  \hat w_2  -343500 \hat t^2 \hat x_2 \odot   \hat v  \hat w_3+\textstyle\frac{532517}{9504}\hat t \hat s^2 \hat v \odot   \hat t \hat s^2 \hat v 
\\
&+\textstyle\frac{5089175}{396}\hat t \hat s^2 \hat v \odot \hat   t \hat s \hat w_1 +\textstyle\frac{68795}{36}\hat t \hat s^2 \hat v \odot \hat   t \hat s \hat w_2  -\textstyle\frac{217055}{132}\hat t \hat s^2 \hat v \odot  \hat  t \hat s \hat w_3 + \textstyle\frac{19873065}{11}\hat t \hat s^2 \hat v \odot \hat  t \hat y  -\textstyle\frac{4633}{9504}\hat t \hat s^2 \hat v \odot    \hat s^4 
\\
&+ \textstyle\frac{166205}{792}\hat t \hat s^2 \hat v \odot    \hat s^2  \hat u_1 + \textstyle\frac{2735}{72}\hat t \hat s^2 \hat v \odot    \hat s^2  \hat u_2 +\textstyle\frac{1034740}{33}\hat t \hat s^2 \hat v \odot   \hat s  \hat x_1-\textstyle\frac{935}{3}\hat t \hat s^2 \hat v \odot    \hat v  \hat w_2 -\textstyle\frac{38915}{33}\hat t \hat s^2 \hat v \odot   \hat v  \hat w_3
\\
&+ \textstyle\frac{39432960}{11}\hat t \hat s^2 \hat v \odot    \hat z +\textstyle\frac{502001260}{297}\hat t \hat s \hat w_1 \odot  \hat  t \hat s \hat w_1 + \textstyle\frac{590395}{9}\hat t \hat s \hat w_1 \odot \hat   t \hat s \hat w_2 + \textstyle\frac{4781930}{33}\hat t \hat s \hat w_1 \odot \hat   t \hat s \hat w_3 + \textstyle\frac{5760977100}{11}\hat t \hat s \hat w_1 \odot  \hat  t \hat y
\\
&+ \textstyle\frac{254365}{2376}\hat t \hat s \hat w_1 \odot    \hat s^4 + \textstyle\frac{1659235}{54}\hat t \hat s \hat w_1 \odot    \hat s^2  \hat u_1 -\textstyle\frac{3995}{18}\hat t \hat s \hat w_1 \odot    \hat s^2  \hat u_2 + \textstyle\frac{245058800}{33}\hat t \hat s \hat w_1 \odot   \hat s  \hat x_1+ \textstyle\frac{59800}{9}\hat t \hat s \hat w_1 \odot    \hat v  \hat w_2 
\\
&+ \textstyle\frac{876100}{33}\hat t \hat s \hat w_1 \odot   \hat v  \hat w_3+ \textstyle\frac{11292825600}{11}\hat t \hat s \hat w_1 \odot    \hat z +\textstyle\frac{241175}{6}\hat t \hat s \hat w_2 \odot \hat   t \hat s \hat w_2  -129025\hat t \hat s \hat w_2 \odot\hat    t \hat s \hat w_3 -434700\hat t \hat s \hat w_2 \odot \hat   t \hat y 
\\
&-\textstyle\frac{2375}{72}\hat t \hat s \hat w_2 \odot    \hat s^4+ \textstyle\frac{35785}{18}\hat t \hat s \hat w_2 \odot    \hat s^2  \hat u_1 +\textstyle\frac{6875}{6}\hat t \hat s \hat w_2 \odot    \hat s^2  \hat u_2+193200\hat t \hat s \hat w_2 \odot   \hat s  \hat x_1 -\textstyle\frac{73700}{3}\hat t \hat s \hat w_2 \odot    \hat v  \hat w_2 
\\
& -94100\hat t \hat s \hat w_2 \odot   \hat v  \hat w_3+\textstyle\frac{3656700}{11}\hat t \hat s \hat w_3 \odot \hat   t \hat s \hat w_3  -\textstyle\frac{4932900}{11}\hat t \hat s \hat w_3 \odot  \hat  t \hat y + \textstyle\frac{6365}{264}\hat t \hat s \hat w_3 \odot    \hat s^4  -\textstyle\frac{114535}{66}\hat t \hat s \hat w_3 \odot    \hat s^2  \hat u_1  
\\
&-\textstyle\frac{1675}{2}\hat t \hat s \hat w_3 \odot    \hat s^2  \hat u_2 +\textstyle\frac{2192400}{11}\hat t \hat s \hat w_3 \odot   \hat s  \hat x_1 -75200\hat t \hat s \hat w_3 \odot    \hat v  \hat w_2 + \textstyle\frac{2222700}{11}\hat t \hat s \hat w_3 \odot   \hat v  \hat w_3+\textstyle\frac{457398003600}{11}\hat t \hat y \odot \hat   t \hat y 
\\
&+ \textstyle\frac{252315}{22}\hat t \hat y \odot    \hat s^4 + \textstyle\frac{44199750}{11}\hat t \hat y \odot    \hat s^2  \hat u_1  -3150\hat t \hat y \odot    \hat s^2  \hat u_2 + \textstyle\frac{11863958400}{11}\hat t \hat y \odot   \hat s  \hat x_1+ 126000\hat t \hat y \odot    \hat v  \hat w_2 + \textstyle\frac{1512000}{11}\hat t \hat y \odot   \hat v  \hat w_3
\\
&+ \textstyle\frac{1784556748800}{11}\hat t \hat y \odot    \hat z +\textstyle\frac{1109}{76032} \hat s^4 \odot    \hat s^4 -\textstyle\frac{365}{4752} \hat s^4 \odot    \hat s^2  \hat u_1  -\textstyle\frac{95}{144} \hat s^4 \odot    \hat s^2  \hat u_2 + \textstyle\frac{4445}{33} \hat s^4 \odot   \hat s  \hat x_1+ \textstyle\frac{95}{18} \hat s^4 \odot    \hat v  \hat w_2 
\\
&+\textstyle\frac{1235}{66} \hat s^4 \odot   \hat v  \hat w_3+ \textstyle\frac{211680}{11} \hat s^4 \odot    \hat z +\textstyle\frac{408305}{2376} \hat s^2  \hat u_1 \odot    \hat s^2  \hat u_1 + \textstyle\frac{1465}{36} \hat s^2  \hat u_1 \odot    \hat s^2  \hat u_2 +\textstyle\frac{1913800}{33} \hat s^2  \hat u_1 \odot   \hat s  \hat x_1-\textstyle\frac{2650}{9} \hat s^2  \hat u_1 \odot    \hat v  \hat w_2 
\\
&-\textstyle\frac{38650}{33} \hat s^2  \hat u_1 \odot   \hat v  \hat w_3+ \textstyle\frac{76204800}{11} \hat s^2  \hat u_1 \odot    \hat z +\textstyle\frac{275}{24} \hat s^2  \hat u_2 \odot    \hat s^2  \hat u_2 +1400 \hat s^2  \hat u_2 \odot   \hat s  \hat x_1 -\textstyle\frac{550}{3} \hat s^2  \hat u_2 \odot    \hat v  \hat w_2 -650 \hat s^2  \hat u_2 \odot   \hat v  \hat w_3
\\
&+\textstyle\frac{76574400}{11}\hat s  \hat x_1 \odot   \hat s  \hat x_1 -56000\hat s  \hat x_1 \odot    \hat v  \hat w_2  -\textstyle\frac{672000}{11}\hat s  \hat x_1 \odot   \hat v  \hat w_3+ \textstyle\frac{20727705600}{11}\hat s  \hat x_1 \odot    \hat z +\textstyle\frac{17600}{3}\hat v  \hat w_2 \odot     \hat v  \hat w_2 
\\
&-400\hat v  \hat w_2 \odot    \hat v  \hat w_3+\textstyle\frac{433200}{11}\hat v  \hat w_3 \odot    \hat v  \hat w_3+\textstyle\frac{1575305625600}{11}\hat z \odot     \hat z.
\end{align*}
}

\bibliography{ref_books,ref_papers}

\begin{bibdiv}
\begin{biblist}

\bib{Alesker:OnConjecture}{article}{
      author={Alesker, S.},
       title={On {P}. {M}c{M}ullen's conjecture on translation invariant
  valuations},
        date={2000},
        ISSN={0001-8708},
     journal={Adv. Math.},
      volume={155},
      number={2},
       pages={239\ndash 263},
         url={https://doi.org/10.1006/aima.2000.1918},
      review={\MR{1794712}},
}

\bib{Alesker:Irreducibility}{article}{
      author={Alesker, S.},
       title={Description of translation invariant valuations on convex sets
  with solution of {P}. {M}c{M}ullen's conjecture},
        date={2001},
        ISSN={1016-443X},
     journal={Geom. Funct. Anal.},
      volume={11},
      number={2},
       pages={244\ndash 272},
         url={https://doi.org/10.1007/PL00001675},
      review={\MR{1837364}},
}

\bib{Alesker:HL}{article}{
      author={Alesker, S.},
       title={Hard {L}efschetz theorem for valuations, complex integral
  geometry, and unitarily invariant valuations},
        date={2003},
        ISSN={0022-040X},
     journal={J. Differential Geom.},
      volume={63},
      number={1},
       pages={63\ndash 95},
         url={http://projecteuclid.org/euclid.jdg/1080835658},
      review={\MR{2015260}},
}

\bib{Alesker:Product}{article}{
      author={Alesker, S.},
       title={The multiplicative structure on continuous polynomial
  valuations},
        date={2004},
        ISSN={1016-443X},
     journal={Geom. Funct. Anal.},
      volume={14},
      number={1},
       pages={1\ndash 26},
         url={https://doi.org/10.1007/s00039-004-0450-2},
      review={\MR{2053598}},
}

\bib{Alesker:Spin9}{article}{
      author={Alesker, S.},
       title={Plurisubharmonic functions on the octonionic plane and
  {S}pin(9)-invariant valuations on convex sets},
        date={2008},
        ISSN={1050-6926},
     journal={J. Geom. Anal.},
      volume={18},
      number={3},
       pages={651\ndash 686},
         url={https://doi.org/10.1007/s12220-008-9032-0},
      review={\MR{2420758}},
}

\bib{AleskerFu:Barcelona}{book}{
      author={Alesker, S.},
      author={Fu, J. H.~G.},
       title={Integral geometry and valuations},
      series={Advanced Courses in Mathematics. CRM Barcelona},
   publisher={Birkh\"{a}user/Springer, Basel},
        date={2014},
        ISBN={978-3-0348-0873-6; 978-3-0348-0874-3},
      review={\MR{3380549}},
}

\bib{AtiyahMacDonald:Introduction}{book}{
      author={Atiyah, M.~F.},
      author={Macdonald, I.~G.},
       title={Introduction to commutative algebra},
   publisher={Addison-Wesley Publishing Co., Reading, Mass.-London-Don Mills,
  Ont.},
        date={1969},
      review={\MR{0242802}},
}

\bib{Baez:Octonions}{article}{
      author={Baez, John~C.},
       title={The octonions},
        date={2002},
        ISSN={0273-0979},
     journal={Bull. Amer. Math. Soc. (N.S.)},
      volume={39},
      number={2},
       pages={145\ndash 205},
      review={\MR{1886087}},
}

\bib{Bernig:SU}{article}{
      author={Bernig, Andreas},
       title={A {H}adwiger-type theorem for the special unitary group},
        date={2009},
        ISSN={1016-443X},
     journal={Geom. Funct. Anal.},
      volume={19},
      number={2},
       pages={356\ndash 372},
         url={https://doi.org/10.1007/s00039-009-0008-4},
      review={\MR{2545241}},
}

\bib{Bernig:ProductFormula}{article}{
      author={Bernig, Andreas},
       title={A product formula for valuations on manifolds with applications
  to the integral geometry of the quaternionic line},
        date={2009},
        ISSN={0010-2571},
     journal={Comment. Math. Helv.},
      volume={84},
      number={1},
       pages={1\ndash 19},
         url={https://doi.org/10.4171/CMH/150},
      review={\MR{2466073}},
}

\bib{Bernig:G2}{article}{
      author={Bernig, Andreas},
       title={Integral geometry under {$G_2$} and {${\rm Spin}(7)$}},
        date={2011},
        ISSN={0021-2172},
     journal={Israel J. Math.},
      volume={184},
       pages={301\ndash 316},
         url={https://doi.org/10.1007/s11856-011-0069-6},
      review={\MR{2823979}},
}

\bib{Bernig:Quaternionic}{article}{
      author={Bernig, Andreas},
       title={Invariant valuations on quaternionic vector spaces},
        date={2012},
        ISSN={1474-7480},
     journal={J. Inst. Math. Jussieu},
      volume={11},
      number={3},
       pages={467\ndash 499},
         url={https://doi.org/10.1017/S1474748011000247},
      review={\MR{2931316}},
}

\bib{BernigBroecker:Rumin}{article}{
      author={Bernig, Andreas},
      author={Br\"{o}cker, Ludwig},
       title={Valuations on manifolds and {R}umin cohomology},
        date={2007},
        ISSN={0022-040X},
     journal={J. Differential Geom.},
      volume={75},
      number={3},
       pages={433\ndash 457},
         url={http://projecteuclid.org/euclid.jdg/1175266280},
      review={\MR{2301452}},
}

\bib{BFS:Crofton}{article}{
      author={Bernig, Andreas},
      author={Faifman, Dmitry},
      author={Solanes, Gil},
       title={Crofton formulas in pseudo-riemannian space forms},
        date={2021},
     journal={Compos. Math.},
       pages={to appear},
}

\bib{BFS:PseudoRiemannian}{article}{
      author={Bernig, Andreas},
      author={Faifman, Dmitry},
      author={Solanes, Gil},
       title={Curvature measures of pseudo-{R}iemannian manifolds},
        date={2022},
        ISSN={0075-4102},
     journal={J. Reine Angew. Math.},
      volume={788},
       pages={77\ndash 127},
         url={https://doi.org/10.1515/crelle-2022-0020},
      review={\MR{4445545}},
}

\bib{BernigFu:Convolution}{article}{
      author={Bernig, Andreas},
      author={Fu, Joseph H.~G.},
       title={Convolution of convex valuations},
        date={2006},
        ISSN={0046-5755},
     journal={Geom. Dedicata},
      volume={123},
       pages={153\ndash 169},
         url={https://doi.org/10.1007/s10711-006-9115-7},
      review={\MR{2299731}},
}

\bib{BernigFu:HIG}{article}{
      author={Bernig, Andreas},
      author={Fu, Joseph H.~G.},
       title={Hermitian integral geometry},
        date={2011},
        ISSN={0003-486X},
     journal={Ann. of Math. (2)},
      volume={173},
      number={2},
       pages={907\ndash 945},
         url={https://doi.org/10.4007/annals.2011.173.2.7},
      review={\MR{2776365}},
}

\bib{BFS:ComplexSpaceForms}{article}{
      author={Bernig, Andreas},
      author={Fu, Joseph H.~G.},
      author={Solanes, Gil},
       title={Integral geometry of complex space forms},
        date={2014},
        ISSN={1016-443X},
     journal={Geom. Funct. Anal.},
      volume={24},
      number={2},
       pages={403\ndash 492},
         url={https://doi.org/10.1007/s00039-014-0251-1},
      review={\MR{3192033}},
}

\bib{BernigHug:Tensor}{article}{
      author={Bernig, Andreas},
      author={Hug, Daniel},
       title={Kinematic formulas for tensor valuations},
        date={2018},
        ISSN={0075-4102},
     journal={J. Reine Angew. Math.},
      volume={736},
       pages={141\ndash 191},
         url={https://doi.org/10.1515/crelle-2015-0023},
      review={\MR{3769988}},
}

\bib{BernigSolanes:Classification}{article}{
      author={Bernig, Andreas},
      author={Solanes, Gil},
       title={Classification of invariant valuations on the quaternionic
  plane},
        date={2014},
        ISSN={0022-1236},
     journal={J. Funct. Anal.},
      volume={267},
      number={8},
       pages={2933\ndash 2961},
         url={https://doi.org/10.1016/j.jfa.2014.06.017},
      review={\MR{3255479}},
}

\bib{BernigSolanes:H2}{article}{
      author={Bernig, Andreas},
      author={Solanes, Gil},
       title={Kinematic formulas on the quaternionic plane},
        date={2017},
        ISSN={0024-6115},
     journal={Proc. Lond. Math. Soc. (3)},
      volume={115},
      number={4},
       pages={725\ndash 762},
         url={https://doi.org/10.1112/plms.12050},
      review={\MR{3716941}},
}

\bib{BernigVoide:Spin}{article}{
      author={Bernig, Andreas},
      author={Voide, Floriane},
       title={Spin-invariant valuations on the octonionic plane},
        date={2016},
        ISSN={0021-2172},
     journal={Israel J. Math.},
      volume={214},
      number={2},
       pages={831\ndash 855},
         url={https://doi.org/10.1007/s11856-016-1363-0},
      review={\MR{3544703}},
}

\bib{Borel:Transitive}{article}{
      author={Borel, Armand},
       title={Some remarks about {L}ie groups transitive on spheres and tori},
        date={1949},
        ISSN={0002-9904},
     journal={Bull. Amer. Math. Soc.},
      volume={55},
       pages={580\ndash 587},
         url={https://doi.org/10.1090/S0002-9904-1949-09251-0},
      review={\MR{29915}},
}

\bib{Bryant:Exceptional}{article}{
      author={Bryant, Robert~L.},
       title={Metrics with exceptional holonomy},
        date={1987},
        ISSN={0003-486X},
     journal={Ann. of Math. (2)},
      volume={126},
      number={3},
       pages={525\ndash 576},
         url={https://doi.org/10.2307/1971360},
      review={\MR{916718}},
}

\bib{CastrillonLopez:Canonical8Form}{article}{
      author={Castrill\'{o}n~L\'{o}pez, M.},
      author={Gadea, P.~M.},
      author={Mykytyuk, I.~V.},
       title={The canonical eight-form on manifolds with holonomy group {${\rm
  Spin(9)}$}},
        date={2010},
        ISSN={0219-8878},
     journal={Int. J. Geom. Methods Mod. Phys.},
      volume={7},
      number={7},
       pages={1159\ndash 1183},
         url={https://doi.org/10.1142/S0219887810004786},
      review={\MR{2749383}},
}

\bib{CLM:Hadwiger2}{article}{
      author={Colesanti, Andrea},
      author={Ludwig, Monika},
      author={Mussnig, Fabian},
       title={The {H}adwiger theorem on convex functions. {II}},
        date={2021},
     journal={preprint},
      eprint={arXiv:2109.09434},
}

\bib{Eisenbud:Syzygies}{book}{
      author={Eisenbud, David},
       title={The geometry of syzygies},
      series={Graduate Texts in Mathematics},
   publisher={Springer-Verlag, New York},
        date={2005},
      volume={229},
        ISBN={0-387-22215-4},
        note={A second course in commutative algebra and algebraic geometry},
      review={\MR{2103875}},
}

\bib{Fu:Unitary}{article}{
      author={Fu, Joseph H.~G.},
       title={Structure of the unitary valuation algebra},
        date={2006},
        ISSN={0022-040X},
     journal={J. Differential Geom.},
      volume={72},
      number={3},
       pages={509\ndash 533},
         url={http://projecteuclid.org/euclid.jdg/1143593748},
      review={\MR{2219942}},
}

\bib{FuW:Riemannian}{article}{
      author={Fu, Joseph H.~G.},
      author={Wannerer, Thomas},
       title={Riemannian curvature measures},
        date={2019},
        ISSN={1016-443X},
     journal={Geom. Funct. Anal.},
      volume={29},
      number={2},
       pages={343\ndash 381},
         url={https://doi.org/10.1007/s00039-019-00484-6},
      review={\MR{3945834}},
}

\bib{FultonHarris:RT}{book}{
      author={Fulton, William},
      author={Harris, Joe},
       title={Representation theory},
      series={Graduate Texts in Mathematics},
   publisher={Springer-Verlag, New York},
        date={1991},
      volume={129},
        ISBN={0-387-97527-6; 0-387-97495-4},
         url={https://doi.org/10.1007/978-1-4612-0979-9},
        note={A first course, Readings in Mathematics},
      review={\MR{1153249}},
}

\bib{GWZ:Hopf}{article}{
      author={Gluck, Herman},
      author={Warner, Frank},
      author={Ziller, Wolfgang},
       title={The geometry of the hopf fibrations},
        date={1986},
        ISSN={0013-8584},
     journal={Enseign. Math. (2)},
      volume={32},
      number={3-4},
       pages={173\ndash 198},
      review={\MR{874686}},
}

\bib{Grinberg:Spherical}{article}{
      author={Grinberg, Eric~L.},
       title={Spherical harmonics and integral geometry on projective spaces},
        date={1983},
        ISSN={0002-9947},
     journal={Trans. Amer. Math. Soc.},
      volume={279},
      number={1},
       pages={187\ndash 203},
         url={https://doi.org/10.2307/1999378},
      review={\MR{704609}},
}

\bib{Hadwiger:Vorlesungen}{book}{
      author={Hadwiger, H.},
       title={Vorlesungen \"{u}ber {I}nhalt, {O}berfl\"{a}che und
  {I}soperimetrie},
   publisher={Springer-Verlag, Berlin-G\"{o}ttingen-Heidelberg},
        date={1957},
      review={\MR{0102775}},
}

\bib{Harvey:Spinors}{book}{
      author={Harvey, F.~Reese},
       title={Spinors and calibrations},
      series={Perspectives in Mathematics},
   publisher={Academic Press, Inc., Boston, MA},
      volume={9},
        ISBN={0-12-329650-1},
      review={\MR{1045637}},
}

\bib{HarveyLawson:Calibrated}{article}{
      author={Harvey, Reese},
      author={Lawson, H.~Blaine, Jr.},
       title={Calibrated geometries},
        date={1982},
        ISSN={0001-5962},
     journal={Acta Math.},
      volume={148},
       pages={47\ndash 157},
         url={https://doi.org/10.1007/BF02392726},
      review={\MR{666108}},
}

\bib{Klain:Even}{article}{
      author={Klain, Daniel~A.},
       title={Even valuations on convex bodies},
        date={2000},
        ISSN={0002-9947},
     journal={Trans. Amer. Math. Soc.},
      volume={352},
      number={1},
       pages={71\ndash 93},
         url={https://doi.org/10.1090/S0002-9947-99-02240-0},
      review={\MR{1487620}},
}

\bib{Knoerr:Unitarily1}{article}{
      author={Knoerr, Jonas},
       title={Unitarily invariant valuations on convex functions, {I}},
        date={2021},
     journal={preprint},
      eprint={arXiv:2112.14658},
}

\bib{Kotrbaty:O-valued}{article}{
      author={Kotrbat\'{y}, Jan},
       title={Octonion-valued forms and the canonical 8-form on {R}iemannian
  manifolds with a {$Spin(9)$}-structure},
        date={2020},
        ISSN={1050-6926},
     journal={J. Geom. Anal.},
      volume={30},
      number={4},
       pages={3616\ndash 3640},
         url={https://doi.org/10.1007/s12220-019-00209-z},
      review={\MR{4167259}},
}

\bib{KW:HR}{article}{
      author={Kotrbat\'{y}, Jan},
      author={Wannerer, Thomas},
       title={From harmonic analysis of translation-invariant valuations to
  geometric inequalities for convex bodies},
        date={2022},
     journal={Geom. Funct. Anal.},
       pages={to appear},
}

\bib{Lee:Smooth}{book}{
      author={Lee, John~M.},
       title={Introduction to smooth manifolds},
     edition={Second},
      series={Graduate Texts in Mathematics},
   publisher={Springer, New York},
        date={2013},
      volume={218},
        ISBN={978-1-4419-9981-8},
      review={\MR{2954043}},
}

\bib{McMullen:Continuous}{article}{
      author={McMullen, P.},
       title={Continuous translation-invariant valuations on the space of
  compact convex sets},
        date={1980},
        ISSN={0003-889X},
     journal={Arch. Math. (Basel)},
      volume={34},
      number={4},
       pages={377\ndash 384},
         url={https://doi.org/10.1007/BF01224974},
      review={\MR{593954}},
}

\bib{McMullen:EulerType}{article}{
      author={McMullen, Peter},
       title={Valuations and {E}uler-type relations on certain classes of
  convex polytopes},
        date={1977},
        ISSN={0024-6115},
     journal={Proc. London Math. Soc. (3)},
      volume={35},
      number={1},
       pages={113\ndash 135},
         url={https://doi.org/10.1112/plms/s3-35.1.113},
      review={\MR{448239}},
}

\bib{MontgomerySamelson:Spheres}{article}{
      author={Montgomery, Deane},
      author={Samelson, Hans},
       title={Transformation groups of spheres},
        date={1943},
        ISSN={0003-486X},
     journal={Ann. of Math. (2)},
      volume={44},
       pages={454\ndash 470},
         url={https://doi.org/10.2307/1968975},
      review={\MR{8817}},
}

\bib{Park:PHD}{book}{
      author={Park, Heunggi},
       title={Kinematic formulas for the real subspaces of complex space forms
  of dimension 2 and 3,},
   publisher={Ph.D. Thesis, University of Georgia},
        date={2002},
}

\bib{PartonPiccinni:Spin9}{article}{
      author={Parton, Maurizio},
      author={Piccinni, Paolo},
       title={{$\rm Spin(9)$} and almost complex structures on 16-dimensional
  manifolds},
        date={2012},
        ISSN={0232-704X},
     journal={Ann. Global Anal. Geom.},
      volume={41},
      number={3},
       pages={321\ndash 345},
         url={https://doi.org/10.1007/s10455-011-9285-x},
      review={\MR{2886201}},
}

\bib{Procesi:LieGroups}{book}{
      author={Procesi, Claudio},
       title={Lie groups},
      series={Universitext},
   publisher={Springer, New York},
        date={2007},
        ISBN={978-0-387-26040-2},
        note={An approach through invariants and representations},
      review={\MR{2265844}},
}

\bib{SalamonWalpuski:Notes}{inproceedings}{
      author={Salamon, Dietmar~A.},
      author={Walpuski, Thomas},
       title={Notes on the octonions},
        date={2017},
   booktitle={Proceedings of the {G}\"{o}kova {G}eometry-{T}opology
  {C}onference 2016},
   publisher={G\"{o}kova Geometry/Topology Conference (GGT), G\"{o}kova},
       pages={1\ndash 85},
}

\bib{Schneider:BM}{book}{
      author={Schneider, Rolf},
       title={Convex bodies: the {B}runn-{M}inkowski theory},
     edition={expanded},
      series={Encyclopedia of Mathematics and its Applications},
   publisher={Cambridge University Press, Cambridge},
        date={2014},
      volume={151},
        ISBN={978-1-107-60101-7},
      review={\MR{3155183}},
}

\bib{Schwarz:InvariantG2andSpin7}{article}{
      author={Schwarz, Gerald~W.},
       title={Invariant theory of {$G_2$} and {${\rm Spin}_7$}},
        date={1988},
        ISSN={0010-2571},
     journal={Comment. Math. Helv.},
      volume={63},
      number={4},
       pages={624\ndash 663},
         url={https://doi.org/10.1007/BF02566782},
      review={\MR{966953}},
}

\bib{SolanesW:Spheres}{article}{
      author={Solanes, Gil},
      author={Wannerer, Thomas},
       title={Integral geometry of exceptional spheres},
        date={2021},
        ISSN={0022-040X},
     journal={J. Differential Geom.},
      volume={117},
      number={1},
       pages={137\ndash 191},
         url={https://doi.org/10.4310/jdg/1609902019},
      review={\MR{4195754}},
}

\bib{Tasaki:Generalization1}{incollection}{
      author={Tasaki, Hiroyuki},
       title={Generalization of {K}\"{a}hler angle and integral geometry in
  complex projective spaces},
        date={2001},
   booktitle={Steps in differential geometry ({D}ebrecen, 2000)},
   publisher={Inst. Math. Inform., Debrecen},
       pages={349\ndash 361},
      review={\MR{1859314}},
}

\bib{Tasaki:Generalization2}{article}{
      author={Tasaki, Hiroyuki},
       title={Generalization of {K}\"{a}hler angle and integral geometry in
  complex projective spaces. {II}},
        date={2003},
        ISSN={0025-584X},
     journal={Math. Nachr.},
      volume={252},
       pages={106\ndash 112},
         url={https://doi.org/10.1002/mana.200310040},
      review={\MR{1903043}},
}

\bib{Wang:OnInvariant}{article}{
      author={Wang, Hsien-Chung},
       title={On invariant connections over a principal fibre bundle},
        date={1958},
        ISSN={0027-7630},
     journal={Nagoya Math. J.},
      volume={13},
       pages={1\ndash 19},
         url={http://projecteuclid.org/euclid.nmj/1118800027},
      review={\MR{107276}},
}

\bib{Weyl:ClassicalGroups}{book}{
      author={Weyl, Hermann},
       title={The classical groups. their invariants and representations},
   publisher={Princeton University Press, Princeton, N.J.},
      review={\MR{0000255}},
}

\bib{Zaehle:Integral}{article}{
      author={Z\"{a}hle, M.},
       title={Integral and current representation of {F}ederer's curvature
  measures},
        date={1986},
        ISSN={0003-889X},
     journal={Arch. Math. (Basel)},
      volume={46},
      number={6},
       pages={557\ndash 567},
         url={https://doi.org/10.1007/BF01195026},
      review={\MR{849863}},
}

\end{biblist}
\end{bibdiv}
\bibliographystyle{abbrv}

\end{document}